\documentclass[11pt,reqno]{amsart}
\usepackage{fullpage}
\usepackage[mathscr]{eucal}
\usepackage{mathrsfs}
\usepackage{amsfonts}
\usepackage{amsmath}
\usepackage{amsthm}
\usepackage{amssymb}
\usepackage{latexsym}
\usepackage{stmaryrd}
\usepackage[all]{xy}
\usepackage{array}
 
\usepackage{graphicx}

\usepackage[dvipsnames]{xcolor}
\newcommand{\red}[1]{{\color{red}#1}}

\usepackage{mathtools}

\usepackage{latexsym}
\usepackage[mathscr]{eucal}
\usepackage{palatino}

\usepackage[bookmarks=false]{hyperref}
\usepackage{hyperref}
\usepackage{xr-hyper}

\usepackage{graphicx}

\theoremstyle{plain}
\newtheorem{thm}[equation]{Theorem}
\newtheorem{cor}[equation]{Corollary}
\newtheorem{prop}[equation]{Proposition}
\newtheorem{lem}[equation]{Lemma}

\theoremstyle{definition}

\newtheorem*{defn}{Definition}

\theoremstyle{remark}

\newtheorem{examp}[equation]{Example}
\newtheorem{rem}[equation]{Remark}
\newtheorem{rems}[equation]{Remarks}

\makeatletter
\renewcommand{\subsection}{\@startsection{subsection}{2}{0pt}{-3ex
plus -1ex minus -0.2ex}{-2mm plus -0pt minus
-2pt}{\normalfont\bfseries}} \makeatother

\numberwithin{equation}{subsection}
\allowdisplaybreaks

\newlength{\dhatheight}

\newcommand{\op}{\operatorname}

\newcommand{\hdot}{{\:\raisebox{2pt}{\text{\circle*{1.5}}}}}
\newcommand{\idot}{{\:\raisebox{2pt}{\text{\circle*{1.5}}}}}

\newcommand{\rst}[1]{\ensuremath{{\mathbin|}\raise-.5ex\hbox{$#1$}}}  

\newcommand{\ce}{{\mathcal E}}

\DeclareMathOperator{\mmod}{\!\text{-}\!\op{mod}}
\DeclareMathOperator{\grmod}{\!\text{-}\op{grmod}}

\DeclareMathOperator{\Wh}{{\mathcal{W}\textit{h}}}

\DeclareMathOperator{\Tor}{\mathrm{Tor}}
\DeclareMathOperator{\Ext}{\mathrm{Ext}}
\DeclareMathOperator{\sym}{\mathrm{Sym}}

\DeclareMathOperator{\im}{\mathrm{Im}}
\DeclareMathOperator{\supp}{\mathrm{Supp}}

\DeclareMathOperator{\hol}{{\!\text{-}\mathrm{hol}}}
\DeclareMathOperator{\Ker}{\mathrm{Ker}}

\DeclareMathOperator{\End}{\mathrm{End}}

\DeclareMathOperator{\rk}{{\mathrm{rk}}\,\g}
\DeclareMathOperator{\gr}{\mathrm{gr}}

\DeclareMathOperator{\limp}{\underset{^{z\to0}}{\mathrm{lim}}}
\DeclareMathOperator{\Lie}{\mathrm{Lie}}

\DeclareMathOperator{\Ad}{\mathrm{Ad}}
\DeclareMathOperator{\ad}{\mathrm{ad}}

\DeclareMathOperator{\ord}{\mathrm{ord}}

\DeclareMathOperator{\Spec}{\mathrm{Spec}}

\DeclareMathOperator{\Q}{{\mathrm{Q}}}
\newcommand{\bt}{{\mathbf t}}
\newcommand{\F}{{\mathsf F}}
\newcommand{\iso}{{\;\stackrel{_\sim}{\to}\;}}
\newcommand{\cd}{\!\cdot\!}

\renewcommand{\mod}{{\,\operatorname{\textsl{mod}}\ }}
\newcommand{\erem}{\hfill$\lozenge$\end{rem}}
\newcommand{\eerem}{\hfill$\lozenge$\end{rem}\vskip 3pt }
\newcommand{\dis}{\displaystyle}

\newcommand{\beq}{\begin{equation}\label}
\newcommand{\eeq}{\end{equation}}
\def\ccirc{{{}_{\,{}^{^\circ}}}}
\DeclareMathOperator{\Hom}{\mathrm{Hom}}

\newcommand{\BN}{{\bar N }}
\newcommand{\Cat}{{\scr{C}_{_{\!}}\textit{at}}}
\newcommand{\BO}{\bar B }

\newcommand{\Nn}{\bar N }

\newcommand{\bn}{\bar\n }
\newcommand{\bnp}{{\n^\psi}}
\newcommand{\bnr}{{\bn_r }}
\newcommand{\fa}{{\mathfrak a}}

\newcommand{\cR}{{\mathcal R}}
\newcommand{\aff}{_{\op{aff}}}
\newcommand{\W}{{\mathcal W}}

\newcommand{\fm}{{\mathfrak m}}

\newcommand{\fZ}{{\mathfrak Z}}
\newcommand{\fk}{{\mathfrak k}}

\newcommand{\un}{\UU\n }

\newcommand{\HS}{{\mathbb{H}}^{\op{sph}}}
\newcommand{\HH}{{\mathbb{H}}}

\newcommand{\WW}{\wt W }

\newcommand{\BG}{{\mathbb G}}

\renewcommand{\th}{{\theta}}
\newcommand{\T}{{\mathcal T}}
\newcommand{\fD}{{\mathfrak D}}
\newcommand{\brp}{{\n_r^\psi}}

\newcommand{\A}{{\mathfrak A}}

\newcommand{\hsph}{\nil(\t\aff,\wt W)^{\op{sph}}}

\newcommand{\zh}{{Z_\hb\g}}

\newcommand{\bN}{{\bar N}}

\renewcommand{\o}{\otimes }
\newcommand{\bplus}{\mbox{$\bigoplus$}}
\newcommand{\ccong}{\ \cong \  }

\newcommand{\ds}{\!\sslash \!}
\newcommand{\dss}{\hskip-4pt\fatslash }
\newcommand{\drs}{\hskip-2pt\fatbslash\hskip-2pt }

\newcommand{\wt}{\widetilde }
\newcommand{\Om}{\Omega }

\newcommand{\BX}{{{\mathbb X}}}
\newcommand{\bxp}{{{\mathbb X}^{++}}}

\newcommand{\Y}{{\mathcal Y}}

\newcommand{\taff}{(\T^*\tb)_{\op{aff}}}
\newcommand{\bac}{\backslash }
\newcommand{\pt}{{\operatorname{pt}}}
\newcommand{\pe}{{\mathbf e}}
\newcommand{\nil}{{\mathcal H}}

\newcommand{\BM}{{\mathbb M}}
\newcommand{\bmg}{{{\mathbb M}_G}}
\newcommand{\mmg}{{\mathcal M}}
\newcommand{\hpp}{{\mathbb W}}
\newcommand{\hph}{{\mathbb W}_\hb }

\newcommand{\Ups}{\Sigma }

\newcommand{\chb}{\C[\hb]}
\newcommand{\ho}{|_{\hb=0} }

\newcommand{\scr}[1]{\mathscr{#1}}

\def\ccirc{{{}_{^{\,^\circ}}}}

\newcommand{\mmid}{\enspace\big|\enspace}

\newcommand{\ut}{{\mathcal{U}\t}}
\newcommand{\la}{\lambda}
\newcommand{\be}{\beta }

\newcommand{\ks}{ \bar K }

\newcommand{\cmu}{{\check \mu}}

\newcommand{\calph}{{\check \alpha}}

\newcommand{\s}{{\mathcal S}}

\newcommand{\zg}{{Z\g}}

\newcommand{\dd}{{\mathscr{D}}}
\renewcommand{\dh}{{\mathscr{D}_\hb}}
\newcommand{\oo}{{\mathcal{O}}}

\newcommand{\ug}{{\mathcal{U}}\mathfrak{g}}
\newcommand{\UU}{{\mathcal{U}}}

\newcommand{\tg}{{\widetilde \g}}

\newcommand{\js}{ J^{\op{sph}}}

\newcommand{\bbe}{{\mathsf e}}
\newcommand{\bbf}{{\mathsf f}}
\newcommand{\bbh}{{\mathsf h}}

\newcommand{\ca}{{\mathcal A}}

\newcommand{\sll}{{\mathfrak{sl}_2}}

\renewcommand{\t}{{\mathfrak t}}
\newcommand{\fn}{{\mathfrak n}}

\newcommand{\vpi}{\varpi }
\newcommand{\tb}{{\wt\bb}}

\newcommand{\hb}{\hbar }
\newcommand{\uh}{{\mathcal{U}_\hb}}

\newcommand{\Gr}{{\mathsf{Gr}}}

\newcommand{\cg}{{\mathcal{Z}}}

\newcommand{\al}{\alpha }

\newcommand{\fc}{{\mathfrak{c}}}

\newcommand{\sh}{{\operatorname{S}}_\hb}
\newcommand{\vkap}{\varkappa }

\newcommand{\kap}{\kappa }

\newcommand{\eps}{\epsilon }

\newcommand{\C}{\mathbb{C}}
\newcommand{\g}{\mathfrak{g}}
\renewcommand{\b}{{\mathfrak{b}}}
\newcommand{\n}{\mathfrak{n}}

\newcommand{\h}{\mathfrak{h}}
\newcommand{\La}{\Lambda }

\newcommand{\bb}{{\mathcal B}}

\newcommand{\inv}{^{-1}}
\newcommand{\reg}{^{\operatorname{reg}}}
\newcommand{\cf}{{\mathcal F}}
\newcommand{\Z}{{\mathbb Z}}
\newcommand{\en}{{\enspace}}

\newcommand{\vi}{${\en\sf {(i)}}\;$}
\newcommand{\vii}{${\;\sf {(ii)}}\;$}
\newcommand{\viii}{${\sf {(iii)}}\;$}
\newcommand{\iv}{${\sf {(iv)}}\;$}

\newcommand{\sset}{\subset}
\newcommand{\sminus}{\smallsetminus}
\newcommand{\intoo}{\,\xymatrix{\ar@{^{(}->}[r]&}\,}
\newcommand{\ontoo}{\,\xymatrix{\ar@{->>}[r]&}\,}
\newcommand{\into}{\,\hookrightarrow\,}
\newcommand{\too}{\,\longrightarrow\,}
\newcommand{\mto}{\mapsto}
\newcommand{\onto}{\,\,\twoheadrightarrow\,\,}

\newcommand{\Ga}{\Gamma }
\newcommand{\gm}{{\C^\times}}
\newcommand{\half}{\mbox{$\frac{1}{2}$}}

\title{Nil Hecke algebras and Whittaker $\dd$-modules}
\author{Victor Ginzburg}
\address{Department of Mathematics, University of Chicago,  Chicago, IL  60637, USA.}
\email{ginzburg@math.uchicago.edu}
\begin{document}
\maketitle
\begin{flushright} {\em To the memory of Bertram Kostant}
\end{flushright}
\bigskip

\begin{abstract}  Given a semisimple group $G$, Kostant and Kumar defined a  nil Hecke algebra 
that may be viewed as a  degenerate version of the double affine nil Hecke algebra introduced by Cherednik.
In this paper, we construct an isomorphism of the spherical subalgebra of the nil Hecke algebra with a Whittaker type
quantum Hamiltonian reduction of the algebra of differential operators on $G$. 
This result has an interpretation in terms of  geometric Satake and the Langlands dual group.
{Specifically}, the isomorphism provides a bridge between very differently looking descriptions
 of equivariant Borel-Moore homology 
of the affine flag variety (due to Kostant and Kumar) and 
of the affine Grassmannian (due to Bezrukavnikov and Finkelberg),
respectively. 

It follows from our result that the category of Whittaker $\dd$-modules on
$G$, considered by Drinfeld, is equivalent to the category of holonomic modules over the
 nil Hecke algebra,  and it is also equivalent to a certain subcategory of the category of 
Weyl group equivariant holonomic 
$\dd$-modules on the maximal torus.
\end{abstract}

{\small
\tableofcontents
}
\section{Introduction}
\subsection{Reminder on nil-Hecke algebras}\label{ssec1}

In this paper, we work over $\C$.
We use the notation $\sym \fk$, resp. $\UU\fk$, for the symmetric, resp. enveloping, algebra  of a vector space,
resp.   Lie algebra, $\fk$. 
Let  $\T^*X$, resp. $\dd_X$ and $\dd(X)$, denote the cotangent bundle,
resp. the sheaf and ring of algebraic differential operators, on a smooth algebraic variety $X$.
Throughout the paper, we fix a complex connected and simply
connected semisimple group $G$ with Lie algebra $\g$. 

Let $\h$ be a finite-dimensional vector space,
$\cR\sset \h^*$ a reduced root system with the set $\Ups$ of simple roots.
An associated Coxeter group $\W$ acts naturally on
$\h$ and it is generated by a set $s_\al,\ \al\in\Ups$, of simple reflections.
In their work on equivariant cohomology of flag varieties, 
Kostant and Kumar \cite{KK1}--\cite{KK2} introduced a noncommutative $\Z$-graded 
algebra 
$\nil(\h,\W)$ called  a  nil Hecke algebra,  cf.
 Section  \ref{daha-sec} for an overview.
The algebra $\nil(\h,\W)$ is generated
by the vector space $\h$ and  a collection,
$\th_\al,\ \al\in \Ups\aff$, of  Demazure elements.
The elements of $\h$ pairwise commute and  generate a copy of  the algebra $\sym\h$ inside $\nil(\h,\W)$.
 Demazure elements satisfy the braid relations.
The other  defining relations among the generators of  $\nil(\h,\W)$ are as follows , cf. ~\eqref{def-rel}:
\[(\th_\al)^2=0,\qquad \th_\al\cdot s_\al(h)- h\cdot\th_\al=\langle \al,h\rangle,\qquad \forall\ h\in\h,\,\al\in\Ups.\]

Let $T$ be the  (abstract) maximal
torus of $G$ and  $\t=\Lie T$. Let  $\BX^*=\BX^*(T)$ be the weight lattice and  $W$ 
the abstract  Weyl group.
Further, let $\t\aff$ be the affine  Cartan algebra, $W\aff$ the affine Weyl group,
and  $\nil(\t\aff, W\aff)$ the corresponding  nil Hecke algebra.
It is convenient to enlarge the group $W\aff$ and consider
$\wt W=W\ltimes \BX^*$, an extended affine Weyl group.
Similarly, there is an enlargement, $\nil(\t\aff,\wt W)$,
of  $\nil(\t\aff, W\aff)$.
This is  a $\Z$-graded $\chb$-algebra that may  be viewed as a 
degeneration  of the {\em nil DAHA}  introduced by  Cherednik and
 studied further by Fegin and Cherednik ~\cite{CF}.

Below, we will mostly be interested in the algebra
$\HH:=\nil(\t\aff,\wt W)|_{\hb=1}$, a specialization of
 $\nil(\t\aff,\wt W)$ at $\hb=1$.
The grading on $\nil(\t\aff,\wt W)$ induces an ascending $\Z$-filtration on $\HH$.
 By construction, one has
$\gr\HH=\nil(\t\aff,\wt W)|_{\hb=0}$, resp. $\HH_\hb=\nil(\t\aff,\wt W)$,
where $\gr A$, resp. $A_\hb$, denotes  an associated graded, resp.  Rees algebra, of a
filtered algebra $A$.
The algebra $\HH$ 
is a kind of (micro)-localization of the cross product $W\ltimes \dd(T)$.
In particular, there are algebra embeddings
$\sym\t\into\dd(T)\into \HH$, where $\sym\t$ is identified with
the algebra of  translation invariant differential operators
on the torus $T$. The filtration on $\HH$ agrees with the natural filtration on
$\dd(T)$ by  order of the differential operator.

Let  $\pe=\frac{1}{\#W}\sum_{w\in W}\, w\in\C W$ be  the symmetrizer idempotent.  The algebra  $\HS=\pe\HH\pe$, called {\em spherical subalgebra},
is a filtered algebra  with unit $\pe$. The embeddings above 
restrict to algebra embeddings
$(\sym\t)^W\into\dd(T)^W\into \HS$.

\subsection{Spherical subalgebra via Hamiltonian reduction}\label{hamintro}
Let $K$ be a  linear algebraic group and $X$  a smooth  $K$-variety.
Given a character
$\chi: \fk=\Lie K\to \C$, we write $\fk^\chi$ for the image of  $\fk$ in $\sym\fk$, resp.
$\UU\fk, \dd(X)$, etc., under the map
 $k\mto k-\chi(k)$.

Throughout the paper, we fix
 a maximal  unipotent subgroup  $N$ of $G$ and 
 a nondegenerate character  $\psi:  \n=\Lie N\to\C$.
Let $N\times N$ act on $G$ by $(n_l,n_r): g\mto n_l\inv g n_r$.
We write $N_l$, resp. $N_r$, for the first, resp. second, factor of  $N\times N$.
(We will use subscripts `$l$' and `$r$' in
other similar contexts.)
Thus, we get a pair,  $ \n^\psi_l$ and  $\n^\psi_r$, of
 commuting Lie subalgebras of $\dd(G)$. We define
the following  quantum Hamiltonian reduction
\[\hpp := 
\big(\dd(G)/\dd(G)(\n^\psi_l+\n^{\psi}_r)\big)^{ N_l\times N_r}.
\]
This  is an associative algebra that comes equipped with a natural
ascending $\Z$-filtration, called
Kazhdan filtration, cf.  Section  \ref{ham-sec}. 

We identify
$Z(\ug)$, the center of $\ug$, with the algebra of  $G$-bi-invariant
differential operators on $G$. 
The embedding of  the algebra of $G$-bi-invariant
differential operators into $\dd(G)$ induces an injective homomorphism
$Z(\ug)\to\hpp$ of filtered algebras. 

One of our main results reads as follows:
\begin{thm}\label{main-intro} There is an isomorphism 
$\hpp\iso\HS$, of
filtered algebras, that
maps the subalgebra $Z(\ug)$ $\sset \hpp\ $ to  the subalgebra $\ (\sym\t)^W\sset\HS$. The resulting map $Z(\ug)\to (\sym\t)^W$ is
the Harish-Chandra isomorphism.
\end{thm}

Theorem \ref{main-intro} has a classical (a.k.a. Poisson) counterpart that involves the moment map
$\mu_{\n\times\n}: \T^*G\to \n^*\times\n^*$, associated with the Hamiltonian $N_l\times N_r$-action on $\T^*G$ induced
by the one  on $G$. The variety $\fZ=\mu\inv(\psi\times\psi)\ds(N_l\times N_r)$,
a classical Hamiltonian reduction, comes equipped with  the structure of an integrable system.
Specifically, $\fZ$  is a smooth
symplectic algebraic variety equipped with a  natural smooth
Lagrangian fibration $\kap: \fZ\to \g^*\ds\!\Ad^* G$ whose fibers are abelian algebraic groups. The resulting
 group scheme on $\g^*\ds\!\Ad^* G$ is known as the {\em universal centralizer}, cf.  Section  \ref{cent-sec}
for a review. Our second result, to  be proved in  Section  \ref{class-intro-pf}, reads as follows:

\begin{thm}\label{class-intro} There is an  isomorphism
$\gr\hpp\ccong \C[\fZ]$
of graded Poisson algebras, which restricts to an isomorphism
$\gr(Z\g)\ccong \kap^*(\C[\g^*\ds\!\Ad^* G])$, of maximal commutative subalgebras
\end{thm}
As a consequence of the two theorems above, one obtains an isomorphism
$\gr\HS\cong\C[\fZ]$, which shows that the algebra $\HS$ may be viewed as a quantization of
the symplectic variety $\fZ$. A slightly modified form  of the isomorphism $\gr\HS\cong\C[\fZ]$
that does not however reveal a connection with nil Hecke algebras has been proved
earlier  by Bezrukavnikov, Finkelberg, and Mirkovic
\cite[Proposition 2.8]{BFM}. 

We remark that   Theorem  \ref{class-intro} neither implies, nor is
a simple consequence of Theorem  \ref{main-intro}, due to the fact that
the filtrations and gradings on the algebras
involved  are not bounded below.

\subsection{Strategy of the proof of Theorem \ref{main-intro}}
The theorems above are, in fact, formal consequences of  a combination  of  results (to be recalled in  Section  \ref{grass-sec})
of  Kostant and Kumar, \cite{KK1}, \cite{KK2}, on
the one hand, and of
Bezrukavnikov and Finkelberg, \cite{BF}, on the other hand. That approach 
is, however,  rather indirect;  it relies on 
the geometric Satake and involves the Langlands dual group. Thus, one of our primary motivations was to find a more direct
geometric approach.

The strategy of our approach to the isomorphism  $\HS\cong\hpp$ is as follows. 
First, we construct an algebra homomorphism $\dd(T)^W\to\hpp$. 
Then, we show that this homomorphism
 can be extended to  a homomorphism
$\HS\to\hpp$. 
Finally, in  Section  \ref{m-pf} we prove using
Theorem \ref{ku}  that the latter
homomorphism $\HS\to \hpp$ is an isomorphism.

We should  point out that we do {\em not} know how to construct the  homomorphism $\dd(T)^W\to\hpp$ directly. 
The construction of such a homomorphism is not obvious even quasi-classically,
where it amounts to a construction of Beilinson and Kazhdan \cite{BK}, to be recalled in
 Section  \ref{3d}. Thus, we use an indirect approach that involves the algebra
 $\dd(G/\bN)$, where  $\bN$ is a maximal unipotent subgroup opposite to $N$.
Put  $\bn=\Lie\bN$. Our construction starts with the following Hamiltonian reduction
\[A:=\big(\dd(G/\bN)/\dd(G/\bN)\n^\psi_l\big)^{ N_l}=\big(\dd(G)/\dd(G)(\n^\psi_l+\bn_r)\big)^{ N_l\times \bN_r}.
\] 

There is  a Weyl group action 
on $\dd(G/\bN)$ by algebra automorphisms that has been introduced by  Gelfand and Graev a
 long time ago.
The Gelfand-Graev $W$-action descends to  $A$ so one has an algebra $W\ltimes A$.
The key ingredient of our approach is 
 a $(W\ltimes A,\,\hpp)$-bimodule $\BM$, the {\em Miura bimodule}
defined  Section  \ref{miura-Sec}, an object 
closely related to the one introduced by D. Kazhdan and the author in  \cite{GK}.
We prove (Theorem \ref{HDD})
that the algebra  $A$ is isomorphic to $\dd(T)$; moreover,
  the isomorphism intertwines the  $W$-action on $A$ and
 the natural  $W$-action on $\dd(T)$, see Proposition \ref{w-act}.
Thus, we may view $\BM$ as a $(W\times \dd(T),\,\hpp)$-bimodule.
Further,  using a general criterion of
Proposition  \ref{hh-morita}, we  deduce 
 that the left action of $W\times \dd(T)$ on $\BM$ can be extended to an action
of $\HH$, a larger algebra.
Now, the Miura bimodule comes equipped with a canonical generator $1_\BM\in\BM$. We prove that
for any $a\in \HS\sset\HH$ there exists a uniquely determined element $a_\hpp\in\hpp$,. such that one has
$a1_\BM=1_\BM a_\hpp$. The assignment $a\mto a_\hpp$ yields the desired homomorphism
$\HS\to\hpp$.

\subsection{Relation to equivariant homology of flag varieties}\label{grass-sec}
Let $\check G$ be  the {\em Langlands dual} group of $G$ and $\check T$ the maximal torus of $\check G$.
Let $I\sset \check{G}((z))$ be an Iwahori subgroup.
Let $\bb\aff=\check{G}((z))/I$ be the affine flag variety, resp. ${\mathsf{Gr}}=\check{G}((z))/\check{G}[[z]]$ the affine Grassmannian.

Given a group $K$ and a $K$-action on a space $X$, let $H^\hdot_K(X)$, resp. $H_\idot^K(X)$, denote equivariant cohomology, resp.
homology, of $X$. One has canonical isomorphisms
\[H^\hdot_{I\rtimes \BG_m}(\op{pt})= H^\hdot_{\check T\times\BG_m}(\op{pt})=(\sym\t)[\hb],
\quad\text{resp.}\quad
H^\hdot_{\check{G}[[z]]\rtimes \BG_m}(\op{pt})= (\sym\t)^W[\hb].
\]

In the theorems below, the algebra structure on equivariant homology of $\bb\aff$, resp. $\Gr$, is given by convolution. 
The multiplicative group $\BG_m$ acts on $\bb\aff$, resp. ${\mathsf{Gr}}$, by loop rotation.

Kostant and Kumar, \cite{KK1}, \cite{KK2},\cite{Ku} proved the following  theorem
which is a Kac-Moody generalization of a well-known result  of Bernstein-Gelfand-Gelfand,
\cite{BGG}.
\begin{thm} \label{kk} There are graded algebra isomorphisms
\[
H_\idot^{I\rtimes\BG_m}(\bb\aff)\ccong \HH_\hb,\quad
H_\idot^{\check{G}[[z]]\rtimes\BG_m}({\mathsf{Gr}})\ccong
\HS_\hb.
\]
\end{thm}

On the other hand, one has
the following result.
\begin{thm}\label{coh-bfm}  There are graded algebra isomorphisms
\[
H_\idot^{\check{G}[[z]]}(\Gr)\ccong \gr\hpp,
\quad
H_\idot^{\check{G}[[z]]\rtimes\BG_m}({\mathsf{Gr}})\ccong
\hpp _\hb.
\]
\end{thm}

Here, the first isomorphism is due to
Bezrukavnikov, Finkelberg, and Mirkovi\'c \cite{BFM}, Theorem 2.1.2(b) and Proposition 2.8(b),
and the second  isomorphism is due to
Bezrukavnikov and Finkelberg
 \cite{BF}, Theorem 3.
Thus, our Theorem \ref{main-intro} is a formal consequence of a combination of 
the second isomorphism  in  Theorem \ref{kk} and  Theorem \ref{coh-bfm}, respectively.

\subsection{The Whittaker category}\label{cat} 
Let $A$ be a $\Z$-filtered  algebra such that $\gr A$ is a finitely
generated commutative algebra. Given a finitely generated $A$-module $M$, one can
choose a good filtration on $M$ and let $\op{SS}(M)=\op{supp}(\gr M)$.
We say that $M$ is holonomic if $\dim\op{SS}(M)\leq \frac{1}{2}\dim \Spec(\gr A)$.
(Although this definition is certainly not reasonable in the generality of arbitrary filtered algebras $A$ as above,
it is sufficient for our limited purposes.)
We write $A\hol$ for the abelian category
of holonomic left $A$-modules.
In the case of the algebra $\HH$, one has  $\gr \HH=W\ltimes {\mathcal A}$, where ${\mathcal A}$
is a commutative algebra, cf. Proposition \ref{hh-class}. We define the notion of holonomicity
by replacing $\gr \HH$ by ${\mathcal A}$ in the previous definition.

Let $K$ be a connected linear algebraic group, $\Om_K$  the canonical bundle on $K$,
and $\chi: \fk=\Lie K\to \C$ a character.
We will work with right $\dd$-modules and put
$e^\chi\Om_K:=\dd_K/\fk^\chi\dd_K$.
 This is a line bundle on $K$
equipped  with a flat connection that does not necessarily 
have regular singularities,
in general. 
Let $X$ be  a smooth variety equipped with a $K$-action $a_K: K\times X\to X$.
Let $pr_1$, resp. $pr_2$, denote the first, resp. second, projection $K\times X\to X$.
A  $(K,\chi)${\em -Whittaker}  $\dd_X$-module
is, by definition,  a $\dd_X$-module $\cf$ equipped with an isomorphism
$a_K^*\cf\cong e^\chi\Om_K\boxtimes \cf$ of $\dd_{K\times X}$-modules that satisfies an appropriate cocycle condition
(here and elsewhere, we abuse notation and write $\cf'\boxtimes \cf$
for $pr_1^*\cf'\o_{\oo_{K\times X}} pr_2^*\cf$). There is an  abelian category $(\dd_X, K,\chi)\mmod$ of
$(K,\chi)$-Whittaker  $\dd_X$-modules.
In the special case where  $\chi=0$, one gets the category $(\dd_X, K)\mmod$ of
$K$-equivariant $\dd_X$-modules. 
If $X$ is affine,
 we may (and will) identify $\dd_X$-modules with $\dd(X)$-modules via the functor
of global sections. If, moreover, the group $K$ is unipotent then a  $(K,\chi)$-Whittaker structure
on $\cf$ amounts to the property (rather than an additional structure) that, for any $k\in \fk$,
the action of $k_\dd-\chi(k)$ on $\Gamma(X,\cf)$ is locally nilpotent, where $k_\dd$ denotes the action
that comes from the $\dd(X)$-action.

Following Drinfeld, we consider the category $(\dd_G, N_l\times N_r, \psi\times \psi )\mmod$
and its full subcategory, to be denoted  $\Wh$ and called 
 the {\em Whittaker category}, 
whose objects are
{\em holonomic} $\dd$-modules.
Drinfeld raised a question of finding a description of 
an $\ell$-adic counterpart of the Whittaker cate-gory in terms of $W$-equivariant 
sheaves on  the maximal torus $T$.
The following result 
provides, in particular, an answer to an analogous question  in the $\dd$-module setting.

\begin{thm}\label{cat-eq} The Whittaker category $\Wh$ is equivalent to any of the following categories:
\begin{enumerate}
\item The category of holonomic $\hpp$-modules;
\item The category of holonomic $\HH$-modules; 
\item The full subcategory, $ \Cat(T, W)$, of the category of $W$-equivariant holonomic $\dd(T)$-modules whose objects $M$ have the property that the 
following natural map is an isomorphism:
\beq{hhdd}\sym\t\,\o_{(\sym\t)^W}\, M^W\to M;
\eeq
\item The full subcategory of the category of holonomic $\dd(T)^W$-modules whose objects $L$ have the property that the map 
\[\sym\t\,\o_{(\sym\t)^W}\, L\to \dd(T)\,\o_{\dd(T)^W}\, L,\]
induced by the inclusion $\sym\t\into\dd(T)$,
is an isomorphism.
\end{enumerate}
\end{thm}

Theorem \ref{cat-eq} will be proved  in Section  \ref{pf-cat}.
\begin{rem} There is a   canonical projection $\t^*/_{\!\text{st}}W\to \t^*/W$, of a stacky quotient of $\t^*$ by $W$ to
the categorical quotient.
Isomorphism \eqref{hhdd} means that $M$, viewed as a quasi-coherent sheaf on
$\t^*/_{\!\text{st}}W$, descends to $\t^*/W$.
\erem

In the recent work  \cite{Lo1},  G. Lonergan established an equivalence of categories 
which is analogous to  the equivalence $\Wh\cong\Cat(T, W)$ in the theorem above.
However, the  important holonomicity condition is not present in the setting of   \cite{Lo1}.
The approach 
in  \cite{Lo1} is of a topological nature; it relies on  the geometric Satake equivalence (via \cite{BF})
and on some general results of  
Goresky, Kottwitz, and MacPherson, \cite{GKM}, cf. also
\cite{LLMSSZ}.  That approach is totally different from ours.

Our proof of  Theorem \ref{cat-eq} is closely related to the theorem below
 that may be of independent interest and which is, in a sense, `dual' to the statement of  \cite{BBM}, Theorem 1.5(1).

Given a map $f: X\to Y$, let $\int^i_f$ denote  the $i$th derived  push-forward functor  on $\dd$-modules.
Let  $a_{\bar N}: \bar N\times G\to G,\ (\bar n,g)\mto g\bar n$, be the action.

\begin{thm}\label{nnbar} For any $M\in (\dd_G, N_r,\psi)\mmod$, we have
$\int^j_{a_{\bar N}}(\Om_{\bar N}\boxtimes M)=0, \ \forall j\neq0$.
In particular, the functor
$ \dis (\dd_G, N_r,\psi)\mmod\to (\dd_G,\bar N_r)\mmod,\  M\mto \int^0_{a_{\bar N}}(\Om_{\bar N}\boxtimes M)$
is exact.
\end{thm}

The proof of Theorem \ref{nnbar} is given in Section  \ref{nnbar-sec}.

\subsection{The Whittaker functor}\label{wf}
Let $X$ be a smooth  $G\times G$-variety,
$G\sset G\times G$ the diagonal, and $(\dd_X, G)\mmod$ the corresponding category of $ G$-equivariant 
$\dd_X$-modules. 
Write   $N_r=1\times N\sset G\times G$.
It is easy to check that one has a well-defined  functor
\beq{Psi-intro} 
\Psi:\ (\dd_X, G)\mmod\too (\dd_X, N_l\times N_r,\psi\times\psi )\mmod,\quad  M\mto \int^0_{a_{N_r}}\,(e^{\psi}\Om_N\boxtimes  M).
\eeq

Assume now that $X$ is affine  and form a Hamiltonian reduction
$\fD:=(\dd(X)/\n_r^{\psi}\dd(X))^{N_r}$.
The map $\n_l^\psi \to \dd(X)$ descends to a map $\n^\psi \to \fD$.
We let $(\fD,\n_l^\psi)\mmod$ be    a full subcategory of the category 
of right $\fD$-modules   such that the induced action of the Lie algebra
$\n_l^\psi$ on the module is locally nilpotent.
One checks that for a right  $\dd(X)$-module $ M$ the space $ M/  M\n_r^\psi$,
of $\n_r^\psi$-coinvariants, has the natural 
 structure of a right $\fD$-module; furthermore,  
the resulting $\n_l^\psi$-action on  $ M/  M\n_r^\psi$  is locally nilpotent
if $ M$ is $G$-equivariant. Also,
the space $\dd(X)/\n_r^{\psi}\dd(X)$
 has the structure of a  $(\fD, \dd(X))$-bimodule. 
We have the following functors
\beq{psi2}
\xymatrix{
(\dd(X), G)\mmod\ \ar[rr] ^<>(0.5){ M\ \mto\   M/  M\n_r^\psi}
&&\  (\fD,\n_l^\psi  )\mmod\ \ar[rrr]^<>(0.5){-\ \o_{\fD}\ \dd(X)/\n_r^{\psi}\dd(X)}
&&&\
(\dd(X), N_l\times N_r,\psi\times\psi )\mmod.
}
\eeq

\begin{thm}\label{gpsi} {Let $X$ be an affine $G\times G$-variety. Then  
 the composite functor in \eqref{psi2} 
is isomorphic to $\Psi$ and the  second functor in \eqref{psi2} is an equivalence.
Furthermore, for any $ M\in (\dd(X), G)\mmod$ and all $j\neq0$, we have $H_j(\n_r^\psi,  M)=0$,
resp. $\int^j_{a_{N_r}}\!\!(\Om_{N_r}\boxtimes M)=0$.
\footnote{I was informed by V. Drinfeld that 
the last equation
has  also been  obtained by Sam Raskin (unpublished).}

Thus, the  first functor 
 in  diagram
\eqref{psi2} is exact  and it  takes holonomic modules to holonomic modules.}
\end{thm}

Below, we are interested in a special case of the above setting where $X=G$ is viewed as a $G\times G$-variety via left and right translations. 
The diagonal $G$-action on $X$ corresponds to the conjugation action of $G$ on itself, so
$(\dd_X, G)\mmod$ is, in this case,  the category  of $\Ad G$-equivariant 
$\dd_G$-modules. Further, 
 there is an equivalence
$(\dd(G), N_l\times N_r,\psi\times\psi )\mmod\cong\hpp\mmod$, see   Section  \ref{q-red} and Lemma \ref{all}.
Also, one checks that for any $\dd(G)$-module $ M$, the subspace
$( M/ M\n_r^\psi )^{\n_l^\psi}\sset  M/ M\n_r^\psi $ of $\n_l^\psi$-invariants   has the natural 
 structure of a $\hpp$-module, see  Section  \ref{q-red}.
From the above theorem, we will deduce  the following result.

\begin{thm}\label{psi-exact}
\vi For any $ M\in (\dd(G),\Ad G)\mmod$ and all $j\neq0$, we have $H_j(\n_r^\psi,  M)=0$,
resp. $\int^j_{a_{N_r}}\!\!(\Om_{N_r}\boxtimes M)=0$.

\vii The functor \eqref{Psi-intro} corresponds, via the equivalence
$(\dd(G), N_l\times N_r,\psi\times\psi )\mmod\cong\hpp\mmod$,
to the functor
$(\dd(G),\Ad G)\mmod\to \hpp\mmod,\  M\mto ( M/ M\n_r^\psi)^{\n_l^\psi}$.
Furthermore, the latter functor  induces via the
equivalence 
 $\Wh\cong \hpp\hol$, of  Theorem \ref{cat-eq}, an exact functor
$(\dd(G),\Ad G)\hol\to\Wh$.
\end{thm}

Theorem \ref{gpsi} will be proved in  Section  \ref{d-alg} and Theorem \ref{psi-exact} will be proved in  Section  \ref{GG1}.

\begin{rems} \vi  In the special case where $X=G$ and  $ M\in (\dd(G),\Ad G)\mmod$ is holonomic (not necessarily regular),
the vanishing of $\int^j_{a_{N_r}}\!\!(\Om_{N_r}\boxtimes M)=0$ for all $j\neq0$
 can be proved in a different way
by adapting the proof  of a result of
Bezrukavnikov, Braverman, and Mirkovic, \cite[Theorem 1.5(2)]{BBM}, 
to the $\dd$-module setting.

\vii  One can show that the functor $\Psi$
takes finitely generated
$\dd(G)$-modules to finitely generated $\hpp$-modules.  Furthermore, it is likely that this  functor  respects singular supports in the sense
that $\op{SS}(\Psi(M))$ is obtained from $\op{SS}(M)$ by a suitable classical Hamiltonian reduction.
\end{rems}

Observe  that
group multiplication $G\times G\to G$ induces a monoidal structure
on a suitably defined (bounded) derived category counterpart
$D(\dd_G,\Ad G)\hol$, resp. $D\!\Wh$,  of the abelian category $(\dd_G,\Ad G)\hol$, resp. $\Wh$.
The functor $\Psi$ in Theorem \ref{psi-exact} can be upgraded
to a {\em monoidal} functor $D(\dd_G,\Ad G)\hol\to D\!\Wh$, cf. Remark \ref{brst}.
(In the setting of $\infty$-categories for an affine Kac-Moody group an analoguous functor
has been considered in \cite{Be}.)
Using that the functor \eqref{hhdd} is exact, it is also possible 
to define   a derived version, $D\Cat(T, W)$,  of the abelian category $\Cat(T, W)$.
The group $T$ being abelian, multiplication $T\times T\to T$
gives  $D\Cat(T, W)$ the structure of  a {\em symmetric} monoidal category.
We expect that there is a derived counterpart of Theorem  \ref{cat-eq}
that provides, in particular, a monoidal equivalence
$D\!\Wh\cong D\Cat(T, W)$. Following Drinfeld, we observe that the existence of such
an equivalence would imply that the monoidal structure on $D\!\Wh$ is 
symmetric.
The symmetry of the monoidal structure on $D\!\Wh$ seems to be unknown at the time
of writing this paper.

We defer a more detailed discussion of these topics  to a separate paper.

\subsection{Acknowledgments}
This paper was inspired by a question of V.  Drinfeld.
The author  is grateful to  Gwyn Bellamy and Vladimir  Drinfeld for helpful comments
and to Tsao-Hsien Chen for useful discussions. Special
thanks are due
to Michael Finkelberg for generously sharing his ideas, \cite{FT},
on the relation of nil DAHA to the topology of the affine Grassmannian.
I am very much indebted to Gus Lonergan for pointing out
two mistakes in the original proof of Proposition
\ref{hh-morita}, suggesting how to fix one of them, and also for kind
clarifications of his work \cite{Lo2} that was used to fix the second mistake. I am grateful to Roman Bezrukavnikov
for informing me about the results of Lonergan \cite{Lo1} before they were made public.

This work was supported in part by the NSF grant DMS-1303462.

\section{Three constructions of the universal centralizer} 
\label{cent-sec} 
\noindent Given an algebraic group  $K$ with Lie algebra $\fk$,
we write $\Ad^*$ for the coadjoint action of $K$ in $\fk^*$.

Recall that we have fixed a connected and simply connected semisimple group $G$ with Lie algebra $\g$.
We write  $G_x$ for 
 the stabilizer  of
an element $x\in\g^*$  under the   $\Ad^* G$-action and
let $\g_x=\Lie G_x$. 
We say that $x$ is regular if $\dim\g_x=\rk$. 
 Let $\g\reg$ be 
 the set of regular elements of $\g^*$.

\subsection{The first construction of $\fZ$}\label{1st-sec}  One can identify the cotangent bundle on $G$ with the variety of triples
\beq{T*G}\T^*G=\{( x', x,g)\in\g^*\times\g^*\times G\mid  x'= \Ad^* g( x)\}.
\eeq
The group  $G$ acts on itself by left and right translation. The induced $G\times G$-action on
$\T^*G$ is given by $(g_1,g_2):\ ( x', x,g)\mto ((\Ad^* g_1( x'),\,\Ad^* g_2( x),\,g_1gg_2\inv)$.
This action is Hamiltonian with moment map
\beq{mulr}\mu_l\times \mu_r: \T^*G\to \g^*\times\g^*,\quad \mu_l( x', x,g)= x',\en \mu_r( x', x,g)= x.
\eeq

Next, let  $G$ act on itself by conjugation and act on
 $\g^*\times G$ diagonally. We define
\[\cg:=\{(x,g)\in \g^*\times G \mid \Ad^* g(x)=x\},\quad \cg\reg  :=\{(x,g)\in \cg\mid x\in\g\reg  \}.
\]

In terms of \eqref{T*G}, the moment map  associated with the $\Ad G$-action on  $\T^*G$ has the form:
 $\mu_{\Ad}( x', x,g)= x'- x$. Writing  $\g^*_{\text{diag}}$ for the diagonal copy of $\g^*$ in $\g^*\times\g^*$,
one obtains the following identifications:
\[\cg\ \xrightarrow[^{\cong}]{(x,g)\mto (x,x,g)}\ (\mu_l\times \mu_r)\inv(\g^*_{\text{diag}})\ =\ \mu_{\Ad}\inv(0).\]

Let 
$\T\reg G:=\mu_l\inv(\g\reg  )=\mu_r\inv(\g\reg  )$. This  is  a $G\times G$-stable Zariski open subset of $\T^*G$,
and  we have $\cg\reg=(\mu_l\times \mu_r)\inv(\g^{\text{reg}}_{\text{diag}})=\mu_{\Ad}\inv(0)\cap\T\reg G$.

The first projection   $\g^*\times G\to\g^*$ makes $\g^*\times G$  a  $G$-equivariant
group scheme on $\g^*$ with fiber $G$. 
The fiber of $\cg\reg$ over $x\in \g\reg  $
 equals $G_x$, which is a  not necessarily connected  abelian group.
This makes $\cg\reg\to\g\reg$   a  smooth $G$-equivariant
abelian group subscheme of the group scheme $\g\reg \times G \to\g\reg  $.

Let $\th: \g^*\to\fc:=\g^*\ds\Ad^*G$ be the coadjoint quotient map. 
The fibers of the composite $\g\reg\into\g^*\to\fc$ are the regular coadjoint orbits.
It follows that the map $\cg\reg\to\fc,\ (x,x,g)\mto\th(x)$, is a $G$-torsor;
in particular, $\cg\reg$ is affine. The Kostant slice, see  Section  \ref{2d-sec}, provides a section
of this $G$-torsor.

\begin{defn} The {\em universal centralizer} is defined as $\fZ:=\cg\reg  \ds G:=\Spec(\C[\cg\reg]^G)$.
The map $\cg\reg\to\fc$ descends to a morphism $\fZ\to\fc$, making
$\fZ$ an abelian  group scheme on $\fc$.
\end{defn}
We see from the above that the universal centralizer may be identified with  a Hamiltonian reduction 
of $\T\reg G$ over $0\in \g^*$:
\beq{ham-ad}
\fZ\cong \big(\mu_{\Ad}\inv(0)\cap \T\reg  G\big)\ds G=: \T\reg  G\dss(\Ad G,0).
\eeq

From now on, we fix a principal $\sll$-triple
$(\bbe,\bbh,\bbf)$. Let $\bbe+\g_\bbf$ be the Kostant slice and
 $\s\sset\g^*$  the image of $\bbe+\g_\bbf$ under the Killing form $(-,-)$.
Further, let $\cg_\s=\{(x,g)\in\s\times G\mid \Ad^* g(s)=s\}$. Using the identification
$\cg=(\mu_l\times \mu_r)\inv(\g^*_{\text{diag}})$, one can identify 
$\cg_\s$ with  the preimage of $\s$ under the map
$\mu_l|_\cg=\mu_r|_\cg:\ \cg\to\g^*,\ (x,x,g)\mto x$.
According to Kostant \cite{Ko1}, \cite{Ko3},  all elements of  $\s$ are regular; furthermore, 
the composite $\s\into\g\reg\to\fc$ is an isomorphism. We deduce that 
$\cg_\s$ is a smooth closed subvariety of $\T^*G$ contained in $\cg\reg$; moreover, the composite $\cg_\s\into\cg\reg\onto \cg\reg\ds G=\fZ$
is an isomorphism.  It follows that the group scheme $\fZ$ is smooth, so
 the Hamiltonian reduction construction in \eqref{ham-ad} provides $\fZ$ with the structure
of a smooth symplectic variety.

We leave the proof of the following lemma to the interested reader.

\begin{lem}\label{los}
The restriction of the symplectic 2-form on $\T^*G$ to the subvariety $\cg_\s$ is nondegenerate, i.e.,
$\cg_\s$ is a symplectic submanifold of  $\T^*G$. Furthermore, the isomorphism
$\cg_\s\iso \fZ$ is a symplectomorphism.
\end{lem}

\subsection{The second  construction of $\fZ$}\label{2d-sec} This construction uses a two-sided Whittaker  Hamiltonian reduction
as follows.

Let  $ \n$ be the unique maximal nilpotent subalgebra of $\g$ that contains the element $\bbf$ and $ N$ the corresponding unipotent subgroup.
Let  $\psi\in\g^*$ be defined  by $\psi(x)=(\bbe,x)$. We will abuse  notation and also write
$\psi$ for $\psi|_{ \n}$.
Thus, we have $\psi\in\s\sset pr\inv(\psi)=\psi+ \n^\perp$,
where $pr: \g^*\to \n^*$ is the natural projection.
According   to Kostant, \cite{Ko3}, the action map ${ N}\times \s\to\psi+ \n^\perp$ is an isomorphism.
This easily yields the following result, see \cite[Corollary 1.3.8]{Gi}.
\begin{lem}\label{kostant}
Let $X$ be a $G$-variety equipped with a $G$-equivariant map $f:X\to \g^*$.  Let $\oo_\psi$ be the localization of $\C[\n^*]$ at $\psi\in\n^*$.
Then, we have:
\vskip 1pt

\vi Let  $\cf$ be a $G$-equivariant quasi-coherent
$\oo_X$-module. Then, $\cf$ is  a flat $(pr\ccirc f)^\hdot\oo_\psi$-module.
In particular, the  composite map $X\xrightarrow{f} \g^*\xrightarrow{pr}\n^*$ is flat at $\psi\in\n^*$.

\vii The subscheme  $f\inv(\psi+ \n^\perp)$ is $ N$-stable.
The action of $N$ on   $f\inv(\psi+ \n^\perp)$  yields an $N$-equivariant isomorphism $ N\times f\inv(\s)\to f\inv(\psi+ \n^\perp)$.

\end{lem}

Let $X$ be a smooth (not necessarily affine) symplectic variety equipped with
a Hamiltonian action of a (not necessarily reductive) linear algebraic group $K$, with
moment map $\mu: X\to \fk^*$.
Let $\chi\in \fk^*$ be a character and assume that $\chi$ is a regular value of $\mu$
and the $K$-action on
$\mu\inv(\chi)$ admits a slice, that is, a subvariety $S\sset \mu\inv(\chi)$ such that
the action $K\times S\to\mu\inv(\chi)$ is an isomorphism of algebraic varieties. 
In such a case, the variety $S$ is  smooth  and inherits a natural symplectic structure.
Furthermore, one can show that any two slices are isomorphic,
and we write $X\dss K:=S$.

We will especially be interested  in the case where $X=\T^*G$ is viewed as a $G_l\times G_r$-variety and
$f=\mu_l\times \mu_r$ is the moment map
 $\T^*G\to \g^*\oplus\g^*$.  In this setting, the role of the Lie algebra $\bnp$ is played by
$\n_l^\psi\oplus \n_r^\psi$. 



\begin{cor}\label{2d} One has a natural isomorphism 
$\fZ= \T^*G\dss(N_l\times N_r,\psi\times \psi)$.
\end{cor}
\begin{proof} Applying Lemma  \ref{kostant}(ii) in the case of the Lie algebra 
$\n_l^\psi\oplus \brp$, we compute
\begin{align*}
 \T^*G\dss(N_l\times N_r,\psi\times \psi)&= (\mu_{l, \n^*}\inv(\psi)\cap\mu_{r, \n^*}\inv(\psi))/ N_l\times N_r\\
&=
 \big((\mu_{l}\inv(\psi+ \fn^\perp)\cap\mu_{r}\inv(\psi+\fn^\perp)\big)\big/ N_l\times N_r\\
&=\mu_{l}\inv(\s)\cap\mu_{r}\inv(\s)=\{(x,g)\in \s\times  G\mid \Ad^*g(x)\in\s\}\\
&=\{(x,g)\in \s\times G\mid \Ad^*g(x)=x\}=\fZ.
\qedhere
\end{align*}
\end{proof}

\subsection{The third  construction of $\fZ$} \label{3d}
This   construction, due to \cite{BFM}, is via affine blow-ups.

\noindent
Let $\bb$ be the flag variety of all Borel subalgebras of $\g$.
Let $\tg:= \{(x,\b)\in \g^*\times\bb\mid x\in \b^\perp\}$.
The first projection $\pi: \tg\to \g,\ (x,\b)\mto x$,
 is the Grothendieck-Springer resolution.
Let $\t$ be the universal Cartan, $T$ the universal Cartan torus,
$W$ the Weyl group, and
$\t^* \to\t^* /W\cong \g^* \ds \Ad^* G=\fc$ the quotient.
It is known that  the following map is an isomorphism:
\beq{tgreg}
\tg\reg  :=\{(x,\b)\in \tg \mid x\in \g\reg\}\ \iso\ \g\reg  \times_\fc\t^*,
\quad  (x,\b)\mto (x, x\op{mod}\b^\perp).
\eeq

Beilinson and Kazhdan \cite{BK} constructed a natural map
\beq{BK} \varkappa:\ \fZ\times_\fc\t^* \too T\times \t^* ,
\eeq
of  $W$-equivariant group schemes on $\t^* $, 
where $W$ acts naturally on $T$, resp. $\t^* $, diagonally on $T\times \t^*$, and trivially on $\fZ$.
The construction is as follows.

First, it was shown by Kostant that for any $(x,\b)\in\tg\reg$,
one has an inclusion $G_x\sset B$, where $B$ stands for the Borel subgroup
with Lie algebra $\b$.
Therefore, there is a well-defined map
\[G_x\into B\onto B/[B,B]=T,\en g\mto g\op{mod} [B,B].
\]
Hence, the assignment $(g,x,\b)\mto  (g\op{mod} [B,B], x,\b)$ gives a morphism 
$\cg\reg\times_{\g\reg  }\tg\reg  \to T\times \tg\reg$, of group schemes on $\tg\reg$.
Using \eqref{tgreg}, we obtain a morphism
\[\cg\reg  \times_{\g\reg  }\tg\reg  =\cg\reg  \times_{\g\reg  }(\g\reg  \times_\fc\t^*)=
\cg\reg  \times_\fc\t^*\too T\times \tg\reg  =T\times (\g\reg  \times_\fc\t^*),
\]
of $W\times G$-equivariant group schemes on $\g\reg$.
Taking categorical quotients by $G$ on each side yields the following map,
which is the required  map $\vkap$ in \eqref{BK}:
\[\fZ\times_\fc\t^*=\cg\reg\ds G \times_\fc\t^*=(\cg\reg\times_{\g\reg  }\t^*)\ds G \too 
(T\times \g\reg  \times_\fc\t^*)\ds G =T\times (\g\reg  \times_\fc\t^*)\ds G =T\times \t^*.
\]

The graph of $\vkap$ gives the following  variety:
\[\Lambda=\{(z,x,t)\in \fZ\times\t^*\times T\mid  \vpi(z)=x\mod W,\en \vkap(z,x)=(t,x)\},\]
where $\vpi: \fZ\to\fc$ is the canonical map. We
identify $\t^*\times T$ with $\T^*T$.

The proof of the following result is left to the reader.

\begin{prop}\label{lagr} \vi $\La$ is a smooth closed Lagrangian subvariety of $\fZ\times \T^*T$.

\vii For any $t\in T$, every irreducible compotent of 
the fiber, $pr_T\inv(t)$, of the map  $pr_T: \Lambda\to T,\, (z,x,t)\mto t$ is a
(possibly singular) Lagrangian subvariety of ~$\fZ$.
\end{prop}

The variety $\La$
is a classical counterpart of the Miura bimodule, to be introduced in  Section  \ref{miura-sec} below.

Let $R\sset \BX^*$ be the  set of roots and  the weight lattice of $G$, respectively.

\begin{examp} Let $\Ups\sset R$ be the set of simple roots and $S\sset\Ups$ a subset.
 Let $L_S$ be an associated standard Levi subgroup of $G$ and
$L_S^{\op{der}}=[L_S,L_S]$,
the derived group of $L_S$.
Let $\bbe_S$ be a principal nilpotent of the Levi subalgebra  $\Lie L_S$
and  $Z(L_S^{\op{der}},\,\bbe_S)$ its centralizer in $L_S^{\op{der}}$. Finally, put
\[\t_S\reg:=\{\la\in\t^*\mid \langle \la,\calph\rangle=0\ \&\ \langle \la, \check\beta\rangle\neq0,\
\forall \al\in S,\, \beta\in \Ups\sminus S\}.
\]
Then, one has a decomposition
\[pr_T\inv(1)\ \cong\ \bigcup_{S\sset\Ups}\ \overline{\t_S\reg}\times Z(L_S^{\op{der}},\,\bbe_S),
\]
where each piece is isomorphic to a union of irreducible components of $pr_T\inv(1)$ permuted transitively by the center of $G$.
\end{examp}

Next,  identify $\sym\t$ with $\C[\t^*]$.
Following \cite{BFM}, let $\C[T\times \t^*, \frac{t^\al-1}{\calph},\ \al\in R]$ be an algebra obtained from $\C[T\times \t^*]$
by adjoining all rational functions $\frac{t^\al-1}{\calph},\ \al\in ~R$,
 where  $t^\la$ stands for an element $\la\in \BX^*$ viewed as a regular function on  $T$.


\begin{thm}[\cite{BFM}, Proposition 2.8]\label{bfm}
The algebra map
$\dis\vkap^*: \  \C[T\times \t^*]\to \C[\fZ\times_\fc\t^*]$, induced by the morphism $\vkap$ in \eqref{BK},
 extends to a $W$-equivariant algebra isomorphism
$$\C[T\times \t^*, \mbox{$\frac{t^\al-1}{\calph}$},\ \al\in R]\ccong   \C[\fZ\times_\fc\t^*].$$
Thus, we have
$$
\fZ\ccong \big(\Spec\C[T\times \t^*, \mbox{$\frac{t^\al-1}{\calph}$},\ \al\in R]\big)\big/W.
$$
\end{thm}

An important step in the proof of  the theorem   is played by the following result,
see \cite[ Section  4]{BFM}.

\begin{prop}\label{bfm-loc} The algebra $\C[T\times \t^*, \frac{t^\al-1}{\calph},\ \al\in R]^W$ is flat over $\C[\t^*]^W$.
\end{prop}

A simple proof of this proposition based on nil Hecke algebras will be given in  Section  \ref{nil-sec} below.


\section{Review of quantum Hamiltonian reduction and 
Kazhdan filtrations}\label{ham-sec} 

\subsection{Quantum Hamiltonian reduction}\label{q-red}
In this subsection we collect some general, mostly standard results on quantum Hamiltonian reduction.

Given  a  Lie algebra $\fk$,
we  write  $E/E\fk$, resp. $\fk E\bac E$ or $E/\fk E$,
for  the space of coinvariants of a  right, resp. left, $\fk$-module $E$.
We  write $E^\fk$ for  the space of invariants.

Let $\b$, resp. $\bar\b$, be the unique  Borel subalgebra of $\g$
such that $\bbf\in\n=[\b,\b]$, resp. $\bbe\in\bn=[\bar\b, \bar\b]$.
We get  a triangular decomposition
$\g=\n\oplus\t\oplus\bn$, where we identify $\t$ with $\b\cap\bar\b$.
We have algebra embeddings $\UU\bn\into\UU\bar \b\into\ug$.
It is clear that $\bn(\UU\bar\b)=(\UU\bar \b)\bn$ is a two-sided ideal
of $\UU\bar \b$ and we have $\UU\bar \b/(\UU\bar \b)\bn\cong\ut$.
Let $Z\g$ be the center of $\ug$
and  $Q:=\ug/(\ug)\n^\psi$.

The following  result is well-known, cf. \cite{Ko3}.

\begin{prop}\label{stand} 

\vi The embedding 
$Z\g\into\ug$ induces an algebra
isomorphism $Z\g\iso(\ug/\bnp)^{\ad\n}$.

\vii $Q$ is free as a left $(\UU\bn\o Z\g)$-module.

\viii The embedding $\ut\into \ug$ 
induces an isomorphism $\ut\iso \bn Q\bac Q$ of left $\ut$-modules.

\iv For all $i>0$, we have  $H^i(\bnp, Q)=0$; furthermore, $H^0(\bnp, Q)= Q^{\bnp}\cong Z\g$, by (i).
\end{prop}

We will also use

\begin{prop}\label{van} 
Let $M$ be an $(\ug,\ug)$-bimodule such that the adjoint $\g$-action on $M$ is locally finite.
Then, for any $j>0$, we have  $H_j(\n^\psi, M)=0$, where we 
view $M$ as an $\bnp$-module via the right action.
\end{prop}
\begin{proof}
Let $M$ be as in the  statement and assume, in addition, that
$M$ is finitely generated as a left, equivalently  right,
$\ug$-module.
 In this case, the required statement follows from  \cite[Lemma ~4.4.1]{Gi}.
More precisely, the lemma was stated  under an additional assumption
that the center of $\ug$ acts on $M$ by scalars both on the left  and on the right.
However, this additional assumption was not used in the proof of the lemma.

In the general case, we can find  a family of  sub $(\ug,\ug)$-bimodules  $M_\al\sset M$,
 such that  each $M_\al$  is finitely generated as a left
$\ug$-module and
we have $M=\underset{^\to}\lim\, M_\al$.
The functor $H_\idot(\n^\psi,-)=\Tor_\idot^{\un}(\psi,-)$  commutes with direct limits.
Hence, we have $H_\idot(\n^\psi,M)=\underset{^\to}\lim\, H_\idot(\n_r^\psi,M_\al)$.
This yields the proposition by the first paragraph of the proof.
\end{proof}


From now on, we let  $A$ be a finitely generated  left noetherian associative algebra and $\ug\to A$ 
 an algebra homomorphism such that
\beq{cond1}
\text{the  adjoint
action $\ad x: a\mto xa-ax$, of  $\g$ on $A$, is locally finite.}
\eeq

Let $\fk$ be a nilpotent Lie subalgebra of $\ug$, such as $\n,\bnp$, or $\bn$.
We  write $(A,\fk)\mmod$ for a full subcategory of the category of $A$-modules  such that the induced $\fk$-action on the module is
\textbf{locally nilpotent}.
Condition \eqref{cond1} implies that the adjoint
action of $\fk$ on $A$ is locally nilpotent. In particular, one has
$A/A\fk\in (A,\fk)\mmod$.

The space
$A\dss\fk:=(A/A\fk)^{\ad\fk}$, resp. $\fk\drs A:=(A/\fk A)^{\ad\fk}$, acquires an algebra structure.
One has an algebra isomorphism $\big(A\dss\fk)^{op}\cong \fk^{op}\drs A^{op}$.
A left, resp. right, $A$-module structure on $E$ induces  a natural left, resp. right, $A\dss\fk$-module structure
on $E^{\fk}$, resp. $E/E\fk$. More generally, for each $j\geq 0$, the cohomology group $H^j(\fk, E)$, resp. homology
group $H_j(\fk,E)$, has the structure of a left, resp. right, $A\dss\fk$-module.
Similarly, there is a  left, resp. right, $\fk\drs A$-module structure on $H_\idot(\fk, E)$, resp. $H^\hdot(\fk, E^\fk)$.
In particular, $A/A\fk$ is an $(A, A\dss\fk)$-bimodule,
resp. $\fk A\bac A$ is an $(\fk\drs A, A)$-bimodule.

In this Section , we will only consider the case $\fk=\bnp$ and put $\A=A\dss\bnp$.
The composite map $Z\g\into\ug\to A$
descends to an algebra homomorphism $Z\g\to \A$.
The map  $\ug\to A$ induces a morphism
$Q\to A/A\bnp$, of $(\ug, Z\g)$-bimodules, and also a morphism
$\ut=\bar\n(\ug)\backslash\ug/(\ug)\bnp\to \bar\n A\backslash A/A\bnp$, of $(\ut, Z\g)$-bimodules.


Below, we work with left modules; similar results hold for right modules.

\begin{prop}\label{a-lem} 
\vi For any $M\in (A,\bnp)\mmod$ and $i\neq0$, we have $H^i(\n^\psi, M)=0$. Furthermore, the 
functor $(A,\bnp)\mmod\to \A\mmod,\ M\mto M^\bnp,$ is an equivalence and
$L\mto A/A\bnp \o_{\A} L$ is its quasi-inverse.

\vii For  $L\in \A\mmod$ and all $i\neq 0$,
we have
$H_i(\bn,\, A/A\bnp \o_{\A} L)=0$; furthermore, there are  natural
 isomorphisms $H_0(\bn, \, A/A\bnp \o_{\A} L)\cong (\bar\n A\backslash A/A\bnp)\,\o_{\A}\, L\cong \ut\o_{Z\g} L$.

\viii The algebra $\A$ is left noetherian.

\iv For   all $i\neq 0$,
we have
$H_i(\n^\psi, A)=0$, where $\n^\psi$ acts on $A$ by right multiplication.
\end{prop}
\begin{proof} Put $Z=Z\g$. Let $L$ be a $Z$-module.
View   $Q\o_Z L$ as an  $\bnp$-module and let $C^\hdot(\bnp, Q\o_Z L)$
be  the corresponding cohomological Chevalley-Eilenberg complex.
Thus, we have
\[C^\hdot(\bnp, Q\o_Z L)=(\wedge^\hdot\bnp)^*\o (Q\o_Z L)=((\wedge^\hdot\bnp)^*\o Q)\o_Z L.\]

Assume first that  the $Z$-module $L$ is flat. Then, we find
\[H^i(\bnp, Q\o_Z L)=H^i(C^\hdot(\bnp, Q\o_Z L))=H^i\big(((\wedge^\hdot\bnp)^*\o Q)\o_Z L\big)=H^i(\bnp,Q)\,\o_Z L.\]
Proposition \ref{stand}(iv) says that the cohomology group on the right equals $L$ if $i=0$, resp.
vanishes if $i>0$. 
Next, let $L$ be an arbitrary, not necessarily flat, $Z$-module and choose a resolution $E^\hdot$, of $L$, by flat
 $Z$-modules. Since $Q$ is free over $Z$, see
Proposition \ref{stand}(ii), we deduce that
the complex $Q\o_Z E^\hdot$ is a resolution of  $Q\o_Z L$.
Since all $E^i$ are flat, each term of  the latter resolution is acyclic with respect to the functor  $(-)^\bnp$,
by the above.
It follows that one has
 \[H^i(\bnp, Q\o_Z L)=H^i((Q\o_Z E^\hdot)^\bnp)=H^i(E^\hdot),\en\text{where}\en H^0(E^\hdot)=L,\en H^i(E^\hdot)=0\ \,\forall  i>0.\]
This implies part (i)  in the special case $A=\ug$,
cf. also \cite[Section  6.2]{GG}. 

To complete the proof of (i) in the general case, we view an object $M\in (A,\bnp)\mmod$ as an
object of $(\ug,\bnp)\mmod$. This yields the cohomology vanishing of
part (i) and shows that the
functor $M\mto M^\bnp$ restricts to an exact functor $(A,\bnp)\mmod\to \A\mmod$.
Further, applying  the equivalence in the special case above
to $M=A/A\bnp$, viewed as an object of  $(\ug,\bnp)\mmod$, and using that $(A/A\bnp)^{\bnp}=\A$, we deduce that the natural map
$Q\o_{Z\g}\A\to A/A\bnp$ is an isomorphism. This yields isomorphisms of functors
\beq{AQ}A/A\bnp\,\o_{\A}\,(-)\ =\ (Q\o_{Z}\A)\,\o_{\A}\,(-)\ =\ Q\o_{Z}(-).
\eeq
Using the special case $A=\ug$, we deduce that 
the  map $L\to (Q\o_{Z}L)^\bnp=(A/A\bnp\o_{\A}L)^\bnp$ is an isomorphism,
for any $L\in \A\mmod$.
It follows that the functors in (i) are quasi-inverse to each other.
Further, since $Q$ is free over $\UU\bn\o Z$, the isomorphism
$Q\o_{Z}\A\cong  A/A\bnp$ implies that $A/A\bnp$ is   free as a left $(\UU\bn\o \A^{op})$-module.
Therefore, one has $H_i(\bn,\,A/A\bnp\o_{\A}L)=H_i(\bn,\,Q)\o_{Z}L$,
for any free $\A$-module $L$.
Moreover, the group on the right vanishes for all $i\ne 0$, and we have
$H_0(\bn,\,Q)\o_{Z}L=Q/\bn Q \o_{Z}L\cong \ut\o L$.
The case of an arbitrary $\A$-module $L$ then follows  by considering a resolution of $L$ by free
$\A$-modules, similar  to the proof of (i). This proves (ii).

Now, let $L$ be a finitely generated left $\A$-module.
Then  ${A/A\bnp\o_{\A}L}$ is a  finitely generated $A$-module, hence it is noetherian.
The equivalence of part (i) implies that $L$ is also noetherian, proving that  $\A$ is left  noetherian. 
Finally, view $A$ as a $(\ug,\ug)$-bimodule via left and right multiplication.
This bimodule satisfies the assumptions of Proposition \ref{van}, by \eqref{cond1}.
Thus, part  (iv) follows from  Proposition \ref{van}.
\end{proof}

\begin{rem}\label{brst} \vi Associated with the homomorphism
$\n^\psi\to A$, there is a BRST complex
$D\A:=\wedge^\hdot\big(\n^\psi\oplus (\n^\psi)^*\big)\o A$ that has the natural structure of a dg algebra.
Using that the groups $H_i(\n^\psi,A)$ and $H^i(\n^\psi, A/A\n^\psi)$ vanish
for all $i\neq0$, see parts (i), (iv) of Proposition \ref{a-lem},
it is immediate to deduce that $H^0(D\A)\cong\A$ and $H^i(D\A)=0$, for all $i\neq0$.

\vii Assume that the algebra $A$ has finite global dimension, i.e.,
there is an integer $n=n(A)>0$ such that one has $\Ext_A^{>n}(M,M')=0$, for all
left $A$-modules $M,M'$. Then, 
for all left $\A$-modules $L,L'$, by  Proposition \ref{a-lem},
we get $\Ext_\A^{>n}(L,L')=\Ext_A^{>n}(A/A\n^\psi\o_\A L,\,A/A\n^\psi\o_\A L')=0$.
Hence, the algebra $\A$ has finite global dimension. This can also be deduced from (i).

Let $D^b(A,\n^\psi)\mmod$ be a full triangulated subcategory
of $D^b(A\mmod)$ whose objects $M^\hdot$ have the property
that $H^i(M^\hdot)\in (A,\n^\psi)\mmod,\ \forall i$.
The proof of Proposition \ref{a-lem} can be easily upgraded to show that the functor
$\wedge^\hdot\big(\n^\psi\oplus (\n^\psi)^*\big)\o (-)$  yields a triangulated equivalence of
$D^b(A,\n^\psi)\mmod$ and the bounded derived category of dg $D\A$-modules.
The latter category is  triangulated equivalent to $D^b(\A\mmod)$,
by (i).
\end{rem}
\subsection{Kazhdan filtrations}\label{kazh} Several proofs in the paper will  involve a `filtered-to-graded' reduction.
In this subsection we introduce the necessary definitions and notation.

Let $E$ be  a vector space equipped with a $\Z$-grading $E=\oplus_\ell\ E(\ell)$ and an ascending
$\Z$-filtration  $E_{\leq}$ by graded subspaces $E_{\leq j}=\oplus_\ell \ E_{\leq j}(\ell),\, j\in\Z$.
The {\em Kazhdan filtration}  associated 
with the given grading and filtration   is  an ascending
$\Z$-filtration  $F_\idot E$ defined as follows:
\[
F_nE:= \oplus_{\ell\in\Z}\ F_nE(\ell),\qquad F_nE(\ell):=\ E_{\leq (n-\ell)/2}(\ell).\]
If $E$ is an algebra and the  $\Z$-grading and filtration $E_{\leq}$ are compatible with
the algebra structure, then so is the Kazhdan filtration.

The  adjoint action on $\g$ of  the semisimple element $\bbh$, of the fixed $\sll$-triple,
gives a $\Z$-grading
\beq{weight}
\g = \bplus_{\ell
\in \mathbb{Z}}\ \g(\ell)\quad\text{where}\quad\g(\ell) = \{ x \in \g \mmid [\bbh,x]=\ell\cdot x \},
\eeq
by {\em even} integers. The $\Z$-grading on $\g$ induces one on $\ug$. 
We equip   $\ug$ with the  Kazhdan filtration,
$F_\idot\ug$, associated with that  $\Z$-grading and  the standard PBW filtration  on  $\ug$.
The  Kazhdan filtration  will be our default filtration on $\ug$.
The quotient filtration $F_\idot Q$, on $Q=\ug/(\ug)\bnp$ has no negative terms, i.e., $F_{-1}Q=0$.
We write $\uh\g,\,\zh,\, Q_\hb$, etc., for the corresponding  Rees objects.
Note that the Kazhdan and PBW filtrations on $\ug$ restrict to the same filtration on  $Z\g$.

The algebra $\uh\g$ may be identified, in a natural way  with a 
 $\chb$-algebra generated by a copy of the
vector space $\g$ such that the space $\g(\ell)\sset\g,\ \ell\in\Z$, is placed in degree $\ell+2$, 
with defining relations $x\o y- y\o x=\hb[x,y],\ \forall x,y\in\g$.
With this identification, the space $\bnp=\{n-\psi(n),\ {n\in\n}\}$ becomes a {\em graded} Lie
subalgebra of $\uh\g$; moreover, one has a natural graded algebra isomorphism
$\uh\n=\UU(\bnp)[\hb]$.
Furthermore, the $\uh\g$-module $Q_\hb$ may be identified with $\uh\g/(\uh\g)\bnp$.

\begin{prop}\label{stand2} (1) Analogues of the statements  {\em (ii)--(iv)} of 
Proposition \ref{stand} hold with
$\ug$, resp. $Z\g$, and $Q$, being replaced by  $\gr\ug$, resp. $\gr Z\g$ and $\gr Q$.

(2)   The functors $M\mto M^\bnp$ and $L\mto \uh\g/(\uh\g)\bnp\o_{Z_\hb\g} L$
give quasi-inverse equivalences\break
$(\UU_\hb,\bnp)\grmod\rightleftarrows (Z_\hb\g)\grmod$\
 of the corresponding categories of 
$\Z$-graded $\chb$-modules. 

(3) 
Let $M\in (\ug,\bnp)\mmod$ and  let $F_\idot M$ be a $\Z$-filtration  which
is compatible with the Kazhdan filtration on $\ug$ and such that each of the spaces $F_i M,\ i\in\Z$ is
$\n$-stable. Equip $M^\bnp$ with the filtration induced from
the one on $M$ by restriction.
Then the natural map  $\gr Q\o_{\gr Z\g}\gr (M^\bnp)\to \gr M$,  resp. 
$\gr (M^\bnp)\to (\gr M)^{\bnp}$,  is an isomorphism.

(4) Let $M$ be an $(\ug,\ug)$-bimodule 
equipped
with a $\Z$-filtration $F_\idot M$ which is compatible with the Kazhdan filtration of $\ug\o\ug$.
Assume that  the $\ad\g$-action on $M$ is locally finite and
each of the spaces $F_i M$ is $\ad\g$-stable.
Let  $\C_\psi$ denote  a 1-dimensional $\UU(\bnp)[\hb]$-module  such that the vector space
$\bnp$ and the element $\hb$  kill $\C_\psi$.
Then $\Tor_j^{\UU(\bnp)[\hb]}(\C_\psi, M_\hb)=0\ $ for all $j>0$, where  the algebra $\UU(\bnp)[\hb]=\uh\n\sset\uh\g$ 
acts on $M_\hb$ on the right.
\end{prop}

\begin{proof} Part (1) is a direct consequence of  results of Kostant, \cite{Ko3}.
Given (1), the proof of (2)  is completely analogous to the proof of Proposition \ref{a-lem}(i) in the
case $A=\ug$. Now, let $M$ be as in (3) and let $M_\hb$ be the corresponding Rees $\uh\g$-module.
It follows from (2) that the natural map
$Q_\hb\o_{Z_\hb\g}\,(M_\hb)^\bnp\to M_\hb$ is an isomorphism.
The functor $\chb/(\hb)\o_{\chb}(-)$ is right exact.
Hence, applying this functor yields a surjection
$\gr Q\o_{Z\g}\gr(M^\bnp)\onto\gr M$. Here, each side is an $N$-equivariant 
$\C[\psi+\n^\perp]$-module. It follows from Lemma \ref{kostant}(ii) that taking 
$\n$-invariants is an exact functor on the category of such modules,
and that 
\beq{grgr}(\gr Q\o_{\gr Z\g}\gr(M^\bnp))^{\bnp}=\gr(M^\bnp),
\quad\text{resp.}\quad
\gr M=\gr Q\o_{\gr Z\g}(\gr M)^\bnp.
\eeq
Also, since $F_i(M^\bnp):=M^\bnp\cap F_iM$, we have $\gr ((M_\hb)^\bnp)=\gr(M^\bnp)$.
Thus, from the surjection above we deduce that the map
\[\gr(M^\bnp)=(\gr Q\o_{Z\g}\gr(M^\bnp))^{\bnp}\too
(\gr M)^{\bnp}\]
is surjective. The injectivity of this map is  immediate from its definition.
Thus,  from the isomorphism $\gr(M^\bnp)\iso (\gr M)^{\bnp}$, using \eqref{grgr}, we deduce
that the composite  map $\gr Q\o_{\gr Z\g}\,\gr(M^\bnp)\to \gr Q\o_{\gr Z\g}\,(\gr M)^{\bnp}=\gr M$
is an isomorphism.

The proof of (4) is similar to the proof of Proposition \ref{van}.
\end{proof}

\subsection{Hamiltonian reduction for filtered algebras}\label{AK} Fix an algebra $A$ as in Section \ref{q-red}. Let  $A_{\leq }$ be
a multiplicative $\ad\g$-stable filtration 
 such that $\cup_j\ A_{\leq  j}=A$ and $A_{\leq -1}=\{0\}$.
We will assume  that the homomorphism $\ug\to A$
sends the PBW filtration on $\ug$ to the filtration $A_\leq$.
The adjoint action of the element $\bbh$ gives a $\Z$-grading on $A$.
Associated with this $\Z$-grading and filtration,
one has the  Kazhdan filtration $F_\idot A$ on $A$.
The latter induces a quotient filtration $F_\idot(A/A\bnp)$. By restriction, we get a $\Z$-filtration on 
the algebra ~$\A=A\dss\bnp$. 
We write $A_\hb=\sum_{n\in\Z}\ \hb^n F_nA$, resp. $\A_\hb$,
for the  Rees algebra  associated with  the Kazhdan filtration.
Note that since the Rees algebra  associated 
with the  Kazhdan filtration is canonically isomorphic as a $\chb$-algebra
(but not as a graded algebra), to  the Rees algebra   associated 
with the  filtration $A_{\leq}$, one has  $\cap_{i\in\Z}\hb^i A_\hb=0$.
It follows that the Kazhdan filtration on  $A$ is separating.

We identify $\bnp$ with a graded Lie subalgebra of
$\UU(\bnp)[\hb]=\uh\n$ as in Proposition \ref{stand}(4). It follows from definitions that there is a natural 
graded space, resp. graded algebra, isomorphism
\beq{rees-nat}
(A_\hb/A_\hb\bnp)|_{\hb=0}\cong \gr A/(\gr A)\bnp,\quad\text{resp.}\quad
(A_\hb\dss\bnp)|_{\hb=0}\cong (\gr A)\dss\bnp.
\eeq

\begin{lem}\label{afilt} Assume that 
$\gr_{\leq } A$ is a  finitely generated commutative algebra.
Then, $A_\hb$  is a finitely
generated left noetherian algebra and $A_\hb/A_\hb\bnp$ is flat over $\chb$.

Assume, in addition, that the Kazhdan filtration on $A/A\bnp$  is separating. Then, we have:
\begin{enumerate}

\item  The Kazhdan filtration on  $\A$  is separating and there are natural graded algebra
isomorphisms
\[A_\hb\dss\bnp\cong\A_\hb,\quad\text{resp.,}\quad
(\gr A)\dss\bnp\cong\gr\A.
\]

\item For any finitely generated module
$ M\in (A,\bnp)\mmod$, one has
 \[\dim\op{SS}(M)-\dim\op{SS}(M^\bnp)\ =\ 
\half(\dim\Spec(\gr A)-\dim\Spec(\gr \A)).
\]
\end{enumerate}
\end{lem}
\begin{proof} The algebra $\gr_{\leq } A$ being   finitely generated and commutative, it follows easily that 
the Rees algebra of $A$ associated with  the filtration $A_{\leq}$, hence also $A_\hb$,
  is   finitely generated and  left noetherian.
To prove that  $A/A\bnp$ is flat over $\chb$ we apply  Proposition \ref{stand}(4)
 in the case $M=A$. We deduce 
that $\Tor^{\chb}_j(\chb/(\hb),\ A_\hb/A_\hb\bnp)=0$ for all $j>0$.
It follows that $A_\hb/A_\hb\bnp$ is $\hb$-torsion free. Since $\C[\hb,\hb\inv]\o_{\chb}A_\hb/A_\hb\bnp\cong
\C[\hb,\hb\inv]\o_{\chb}A/A\bnp$, we deduce that $A_\hb/A_\hb\bnp$ is flat over $\chb$.

To prove statement (1),  we let  $\UU\bnp$ act  on $A$ by right multiplication.
The filtrations on $A$ and $\bnp$ make  the corresponding Chevalley-Eilenberg
complex $C_\idot(\bnp,  A )$ a filtered complex. An associated graded complex,
$\gr C(\bnp,  A )$,  may be identified with $C_\idot(\bnp, \gr A)$, the Koszul complex of $\gr A$.
Here, we view $\gr A$  as a
$\C[(\bnp)^*]$-module via the  map $\gr\UU\bnp\to \gr A$ induced by the homomorphisms $\UU\n^\psi\to\ug\to A$.  
There is a standard spectral sequence for homology:
\beq{spect}
E_1=H(\gr C_\idot(\bnp,  A ))\quad\Rightarrow
\quad E_\infty=\gr H(C_\idot(\bnp,  A )).
\eeq

The locally finite $\ad\g$-action on $A$ can be exponentiated to a
$G$-action. Therefore,  the map $\gr \ug\to \gr A$, induced by $\ug\to A$, gives a $G$-equivariant morphism
$\Spec\gr A\to \Spec(\gr\ug)=\g^*$.
From Lemma \ref{kostant}, we deduce
that $\Spec\gr A$ is flat over
$\oo_\psi$. It follows that  the homology groups
$H_j(C_\idot(\bnp, \gr A))$ vanish for all $i\neq 0$. Hence,
our spectral sequence degenerates at the $E_1$-term.

Next, we claim that the above spectral sequence is convergent. This is not immediately
clear since the Kazhdan filtration on $ A $ is not necessarily bounded below.
To prove convergence, observe first that the action of the algebra $ A $ on itself
by {\em left} multiplication makes $C_\idot(\bnp,  A )$ a complex of finitely generated
left $A$-modules. The Chevalley-Eilenberg differential respects the
$A $-module structure. 
We deduce, as in  the  proof of  \cite{Gi}, Theorem 4.14(i), 
that in the spectral sequence \eqref{spect}, one has $Z^\infty=\cap_{r\geq0}\ Z^r_p=0,\ \forall\ p$,
i.e., the first equation in formula (4.3.5) from \cite{Gi} holds.
Further, we know that  $H_j(\n^\psi, A)=0,\ \forall j>0$, see Proposition  \ref{a-lem}(iv).
This, combined with the assumption that the filtration on $A/A\n^\psi$ is separating 
 implies that the second equation in  \cite{Gi}, formula (4.3.5), holds.
Thus, the
convergence   criterion stated in \cite[Lemma 4.3.7]{Gi} applies in our present
setting. It follows that the spectral sequence in \eqref{spect} does converge.
By \cite[Lemma ~4.3.3]{Gi}, we deduce that the canonical map
$\gr A/(\gr A)\bnp\to \gr(A/A\bnp)$ is an isomorphism.
Explicitly, this means that the following  surjective map is an isomorphism:
\beq{2filt}
F_nA\big/\mbox{$\sum_i$}\ (F_{n-1-i}A)\n^\psi_i\ \to\ (F_nA +A\bnp)/(F_{n-1}A+ A\bnp),\qquad\forall\ n\in\Z.
\eeq
Here, 
$\n^\psi_i$ stands for the $i$-th homogeneous component of  $\bnp\,\UU(\bnp)[\hb]=\uh\n$, i.e.,
$\n^\psi_i$ is a copy of the vector space $\n\cap\g(i+2)$ if $i<-2$ and
$\n^\psi_0=\{n-\psi(n),\ n\in \n\cap \g(-2)\}$.

The isomorphisms in (1) are immediate consequences of \eqref{2filt} and Proposition \ref{stand2}(3)
applied to $M=A/A\bnp$. Also, 
the  assumption of the proposition clearly implies that the filtration $F_\idot\A$ 
is separating.
 
Now, let $M$ be as in   (2). Choose a finite-dimensional $\bnp$-stable subspace $ M_0\sset M$, such that
$M=A\, M_0$. The filtration on $M$ defined
by  $F_iM:=(F_iA) M_0,\ i\in\Z$,
satisfies the assumptions of Proposition \ref{stand2}(3). 
Hence, we have $\gr M\cong \gr Q\o_{_{\gr \zg}}\, \gr (M^\bnp)$. 
Using that $\gr Q\cong\C[\psi+\n^\perp]\cong \C[N\times\s]$ and $\s\cong\fc\cong\Spec(\gr Z\g)$, we compute
\begin{align*}
\dim\supp(\gr Q\o_{_{\gr \zg}}\, \gr (M^\bnp)&=\dim \supp\gr Q-\dim\fc+ \dim\supp\gr (M^\bnp)\\
&=\dim N+\dim\supp\gr (M^\bnp).
\end{align*}
In particular, for $M=A/A\bnp$, we get $\dim\supp((\gr A/(\gr A)^\bnp))=
\dim N+\dim\Spec\gr \A$. On the other hand, since $\Spec\gr A$ is flat over
$\oo_\psi$, using part (ii) of Lemma \ref{kostant}, we get
\[\dim\supp((\gr A/(\gr A)^\bnp))=\dim \Spec\gr A-\dim\n^*.\] The required equation follows.
\end{proof}
\pagebreak[3]

\section{Proofs of Theorem \ref{class-intro} and
 Theorem \ref{gpsi}}

\subsection{Whittaker coinvariants of $\dd$-modules}\label{U-sec} Let $X$ be a smooth $G$-variety. The algebra of differential operators  comes equipped 
with an algebra homomorphism $\ug\to\dd(X)$. The  homomorphism intertwines
the PBW filtration on $\ug$ with 
the natural filtration $\dd_{\leq}(X)$ by  order of the differential operator.
Following Section \ref{AK}, we get a $\Z$-grading on $\dd(X)$
and an associated  Kazhdan filtration  $F_\idot\dd(X)$.

\begin{rem}\label{gamma}
 Let $\gamma:\ \BG_m\to T$ be a 1-parameter subgroup generated by the element $\bbh$
of our fixed $\sll$-triple. 
We let the group $\BG_m$ act on $X$
 by $z: x\mto \gamma(z)x$. The  induced $\BG_m$-action
on $\dd(X)$ by algebra automorphisms gives a $\Z$-grading on $\dd(X)$ that agrees
with the
 $\Z$-grading on $\dd(X)$ considered above.
\erem

Let $K$ be a connected algebraic subgroup of $G$ and $\chi:\fk=\Lie K\to\C$ a character.
We have $(\fk^\chi)^{op}\cong \fk^{-\chi}$. The Lie algebra map $\fk^\chi\to\dd(X)$ gives a  Lie algebra map
$\fk^{-\chi}\to\dd(X)^{op}$. Assume, in addition, that $X$ is affine and 
the quotient $X\to X/\!/K$ is a $K$-torsor. Then, there is a canonical isomorphism
$\dd(X/\!/K)\cong \dd(X)\dss\fk$. In the special case where $K$ is unimodular,
e.g.,  unipotent,  and the canonical
bundle on $X$ has a $K$-invariant trivialization, we get $\dd(X)\cong\dd(X)^{op}$, resp.
$\dd(X/K)\cong \dd(X/K)^{op}$. We deduce an isomorphism
$\dd(X)\dss\fk\cong (\dd(X)\dss\fk)^{op}$.
More generally, with the same assumptions one can construct
a natural isomorphism $\dd(X)\dss\fk^{-\chi}\cong {(\dd(X)\dss\fk^\chi)^{op}}$.
On the other hand, by general properties of Hamiltonian reduction there is 
an isomorphism 
$(\dd(X)\dss\fk^\chi)^{op}\cong\fk^{-\chi}\drs\dd(X)^{op}$, see  Section  \ref{q-red}.
Thus, under the same assumptions, one also has an isomorphism $(\dd(X)\dss\fk^\chi)^{op}\cong\fk^{-\chi}\drs\dd(X)$.

Now, let $X$ be a smooth  $G\times G$-variety.
We will abuse  notation and write $G\times G=G_l\times G_r$.
Thus, we have a natural algebra map $\UU\g_l\o\ug_r\to\dd(X)$,
so  $\UU\g_l\o\ug_r$ acts on any  (say, right) $\dd(X)$-module $M$.
Changing a   sign of the $\g$-action, we may view
$M$ as an $(\ug,\ug)$-bimodule. Let $(\dd(X), G)\mmod$ be the category of 
right $\dd(X)$-modules equivariant with respect to the diagonal
$G$-action. For such a module $M$, 
 the $\ad\g$-action on $M$ is locally finite.
We conclude that Proposition ~\ref{van} yields the following result.

\begin{cor}\label{cor-psi} For any $M\in  (\dd(X), G)\mmod$ and $j\neq0$, we have
$H_j(\n_r^\psi, M)=0$. In particular, the functor
$\dis (\dd(X), G)\mmod\mto (\dd(X)\dss\n_r^\psi)\mmod,\
M\mto M/M\n_r^\psi$, is exact.\qed
\end{cor}

\subsection{Proof of Theorem \ref{gpsi}}\label{d-alg} 
Let $X$ be a smooth affine $G\times G$-variety and $M$ a  right $\dd(X)$-module.
We may  view $M$ as a left $\dd(X)^{op}$-module, in particular, an $\n_r^{-\psi}$-module.
The Lie algebra $\n$ being unimodular,
one has a canonical isomorphism of functors $H_\idot(\n,-)\cong H^{d-\hdot}(\n,-)$,
where $d=\dim N$.
It follows that there is a canonical isomorphism  $H^\hdot(\n_r^{-\psi}, M)\cong H_{d-j}(\n_r^\psi,M)$.
By  Section  \ref{q-red}, this space has the  natural structure of  a left $\fD$-module,
where $\fD=\dd(X)^{op}\dss\n_r^{-\psi}$.
We have the identification
$\dd(X)^{op}/\dd(X)^{op}\n_r^{-\psi}=\dd(X)/\n_r^\psi\dd(X)$, and it follows from Proposition \ref{a-lem}(i)
that   the functor $\dd(X)^{op}/\dd(X)^{op}\n_r^{-\psi}\o_\fD(-) :\ \fD\mmod\to (\dd(X), \n_r^\psi)\mmod$ is an equivalence.
Also, as has been mentioned in  Section  \ref{cat}, there is a canonical
equivalence of $(\dd(X), \n_r^\psi)\mmod\cong(\dd_X, N_r, \psi)\mmod$.
We deduce that the second functor in \eqref{psi2} is an equivalence.
The first functor in \eqref{psi2} is exact, by Corollary ~\ref{cor-psi}.

It remains to indentify the composite functor in   \eqref{psi2} with
the  functor \eqref{Psi-intro}. 
It is well-known, cf. \cite{MV},  that an averaging functor
$\int^{-d}_{a_N}(\Om_K\boxtimes-)$ is a right adjoint to
the natural full embedding   $(\dd(X), N_r)\mmod\to \dd(X)\mmod$
(a homological shift by $d$ is due to a shift 
 in adjunction formulas, 
e.g.,  \cite{HTT}, Theorem 3.2.14).
Similarly, one shows that the functor  $\int^{-d}_{a_N}(e^{-\psi}\Om_N\boxtimes-)$
is a right adjoint  to
the natural full 
embedding   $(\dd(X), \n_r^\psi)\mmod\to \dd(X)\mmod$,
to be denoted $\mathbf{I}$.

On the other hand, the functor $\mathbf{I}$ can be factored as a composition
$(\dd(X), \n_r^\psi)\mmod\to \fD\mmod\to \dd(X)\mmod$,
where  the first functor is the equivalence $M\mto M^{\n_r^{-\psi}}$, to be denoted $\mathbf{J}$,
and the second functor
is  $\dd(X)^{op}/\dd(X)^{op}\n_r^{-\psi}\o_\fD (-)$.
From the canonical isomorphisms
\[\Hom_{\dd(X)^{op}}(\dd(X)^{op}/\dd(X)^{op}\n_r^\psi\o_\fD M,\ \cf)\ccong \Hom_\fD(M, \cf^{\n_r^{-\psi}})\cong
\Hom_\fD(M, \, H_d(\n_r^\psi,\cf)),\]
we see that the functor $\mathbf{J}\inv\ccirc H_d(\n_r^\psi,-)$ is a right adjoint of $\mathbf{I}$.
Hence, the functor $\mathbf{J}\inv\ccirc H_d(\n_r^\psi,-)$ must be isomorphic
to the averaging functor  $\int^{-d}_{a_N}(e^{-\psi}\Om_N\boxtimes-)$.
These two functors are left exact. The corresponding $j$-th right derived functors are
$\mathbf{J}\inv\ccirc H_{d-j}(\n_r^\psi,-)$ and $\int^{j-d}_{a_N}(e^{-\psi}\Om_K\boxtimes-)$, respectively.
It follows that these derived functors are isomorphic.
The composite functor in   \eqref{psi2}  is nothing but
the functor $\mathbf{J}\inv\ccirc H_0(\n_r^\psi,-)$, proving the required
isomorphism of functors.\qed

\subsection{Proof of Theorem \ref{class-intro}}\label{class-intro-pf}
Let $\bxp\sset\BX^*$ be the set of dominant weights.
We write $\la\leq\mu$ if $\mu-\la$ is a linear combination of
simple roots with nonnegative integer coefficients.
 
Let $V_\la$  be a simple $G$-representation with highest weight $\la\in \bxp$,
i.e.,   $V_\la$  has a unique, up to a constant factor,
nonzero vector $v_\la\in V_\la$ such that $tv_\la=\la(t)v_\la,\ \forall t\in T$.
Let $V_\la^*$ be the contragredient representation.
We have the Peter-Weyl decomposition
$\C[G]=\oplus_{\mu\in \bxp}\ V_\mu\o V_\mu^*$.
We put $R_\la:=\oplus_{\mu\leq\la}\ V_\mu\o V_\mu^*$.
The spaces  $R_\la,\ \la\in\bxp$,
form  an  exhaustive multiplicative   filtration of the algebra $\C[G]$.

Now, view $X=G$ as a $G\times G$-variety where
$G\times G=G_l\times G_r$ acts on $G$ by left and right translations.
Thus, we have an algebra homomorphism ~$\UU(\g_l\oplus\g_r)=\ug_l\o\ug_r\to\dd(G)$. 
The diagonal $G$-action on $X$ corresponds to the action of $G$ on itself by conjugation.

We have a natural embedding $\C[G]\into \dd(G)$.
For  $\la\in \bxp$, let $P_\la\dd(G):=R_\la\cdot \UU(\g_l\oplus\g_r)$ be a right $\UU(\g_l\oplus\g_r)$-submodule of $\dd(G)$
generated by the subspace $R_\la\sset \C[G]$. 
We have 
\[\dd(G)=\C[G]\cdot \ug_l\  =\   \cup_{\la\in \bxp}\ R_\la\cdot \ug_l\ \sset\  \cup_{\la\in \bxp}\ R_\la\cdot 
\UU(\g_l\oplus\g_r)\  =\   \cup_{\la\in \bxp}\ P_\la\dd(G).
\]
Thus, the spaces $P_\la\dd(G),\ \la\in \bxp$, form  an  exhaustive  filtration
of $\dd(G)$. 

\begin{rem}\label{lrfilt}
 The filtration  $P_\la\dd(G),\ \la\in \bxp$ has a natural geometric interpretation in 
in terms of the {\em wonderful compactification} of
$G^{ad}$, the adjoint group. It is immediate from this interpretation 
 that $\UU(\g_l\oplus\g_r)\cdot R_\la=R_\la\cdot \UU(\g_l\oplus\g_r)$ and each of the spaces
$P_\la\dd(G)$ is stable under the  conjugation action of $G$ on itself. Furthermore, 
for any $\la,\mu\in\bxp$ one has $P_\la\dd(G)\cdot P_\mu\dd(G)
\sset P_{\la+\mu}\dd(G)$. 
\end{rem}

The adjoint action of the element $\bbh_l+\bbh_r\in
\g_l\oplus\g_r$ gives a $\Z$-grading on $\UU(\g_l\oplus\g_r)$.
The corresponding  Kazhdan filtration  $F_\idot\UU(\g_l\oplus\g_r)$
is a tensor product of the Kazhdan filtrations
on $\ug_l$ and $\ug_r$, respectively.
Recall  that the  Kazhdan filtration  on $Q=\ug/(\ug)\bnp$ has no components of negative degrees.
It follows that the quotient  Kazhdan filtration on
$\UU(\g_l\oplus\g_r)/\UU(\g_l\oplus\g_r)\n^\psi_{lr}=Q_l\o Q_r$
 has no components of negative degree.
Here, we have used simplified notation $\n^\psi_{lr}=\n^\psi_l\oplus\n^\psi_r$.

The  conjugation action of the 1-parameter 
subgroup $\gamma$ on $G$ induces a  $\Z$-grading on $\dd(G)$.
 Let $F_\idot\dd(G)$ be the corresponding Kazhdan filtration on $\dd(G)$.
By definition, one has  $(P_\la\dd(G))\n^\psi_{lr}=R_\la\cdot\UU(\g_l\oplus\g_r)\n^\psi_{lr}$.
We deduce from the above that,  for every $\la$
there exists $m=m(\la)\gg0$
such that the quotient  Kazhdan filtration on
$(P_\la\dd(G)+\dd(G)\n^\psi_{lr})/\dd(G)\n^\psi_{lr}$  
 has no components of degrees $\leq -m(\la)$.
This implies that  the  Kazhdan filtration on
$\dd(G)/\dd(G)\n^\psi_{lr}$  is separating.

To complete  the proof of Theorem \ref{class-intro}, we 
consider the setting of  Section  \ref{ham-sec} in the case 
where $A=\dd(G)$ and the Lie algebra $\g$, resp. $\bnp$,  is replaced by $\g_l\oplus\g_r$,
resp.  $\n^\psi_{lr}$. Thus, Lemma \ref{afilt} is applicable.
The filtration on $\dd(G)/\dd(G)\n^\psi_{lr}$ being separating,  from the second isomorphism in statement (1) of the lemma
we deduce that the natural map
\[(\gr^F\dd(G))\dss (N_l\times N_r,\psi\times \psi)
\to \gr^F(\dd(G)\dss\n^\psi_{lr})
\]
is an isomorphism. Further,  we have
$\C[\T^*G\dss (N_l\times N_r,\psi\times \psi)]=(\gr^F\dd(G))\dss (N_l\times N_r,\psi\times \psi)$,
by  Corollary \ref{2d}. Theorem  \ref{class-intro} follows.\qed
\medskip

\begin{cor}\label{noeth} The algebra $\hpp=\dd(G)\, \dss\, (\n_l^\psi+\n_r^{\psi})$ is simple
and the Kazhdan filtration on this algebra is separating.
\end{cor}
\begin{proof} 
The filtration on $\dd(G)\dss(\n^\psi_l+ \n^\psi_r)$
is separating by Lemma \ref{afilt}.
Therefore, for a  nonzero two-sided ideal $I\subsetneq\hpp$,
we have that $\gr I\subsetneq\gr\hpp$ 
is a nonzero  Poisson ideal of $\gr\hpp=\C[\fZ]$,
where  $\C[\fZ]$ is viewed as a Poisson algebra.
 But this  Poisson algebra contains
no nontrivial Poisson ideals, since $\fZ$ is a smooth
symplectic variety, see  Section  \ref{1st-sec}. 
\end{proof}

\begin{rem} Associated with the filtration $R_\la$ on $\C[G]$, one has
the corresponding  Rees algebra
$\oplus_{\la\in\bxp}\ R_\la$. This  $\bxp$-graded algebra
is the coordinate ring of the Vinberg semigroup 
associated with $G$. Similarly, the filtration $P_\la\dd(G)$ on $\dd(G)$ induces
a quotient filtration on $\dd(G)/\dd(G)\n^\psi_{lr}$.
By restriction, one gets a filtration $P_\la\hpp$ on $\dd(G)\dss\n^\psi_{lr}$.
The corresponding  Rees algebra
$\oplus_{\la\in\bxp}\ P_\la\hpp$ may be viewed as a quantum counterpart
of the  coordinate ring of  a Vinberg type deformation of the universal
centralizer $\fZ$. This deformation, as well as an analogue of the wonderful compactification of
$\fZ$ has been studied in \cite{Ba1}--\cite{Ba2}.
\end{rem}

\section{Hamiltonian  reduction of $\T^*(G/\bar N)$ and $\dd(G/\bar N)$}\label{class-W} 

\subsection{The affine closure of $\T^*(G/\bar N)$} \label{gklem-sec}

Let $B$, resp. $\BO$, be   the  Borel subgroup of $G$,
corresponding to the Borel subalgebra $\b$, resp. $\bar\b$.
Thus,  $N$, resp. $\BN$, is the unipotent radical of  $B$, resp. $\BO$,
and $B\cap \BO$ is a maximal torus.
 We identify the flag variety $\bb$ with $G/\BO$ and let   $\tb:=G/\Nn$. 

The group $G$ acts on $\tb$ on the left
and the torus $T=\BO/\Nn$ acts on $\tb$ on the right.
Thus we have  a $G_l\times T_r$-action on $\tb$
and an induced Hamiltonian
$G_l\times T_r$-action on $\T^*\tb$, with 
 moment map $\mu_{G,\T^*\tb}\times \mu_{T,\T^*\tb}:\ \T^*\tb\to\g^*\times\t^*$.
We have a natural isomorphism $(\T^*\tb)/T_r=\tg$, so the quotient map
$q: \T^*\tb\to (\T^*\tb)/T_r=\tg$
makes $\T^*\tb$ a  $G_l$-equivariant $T$-torsor over $\tg$.
The  moment map $\mu_{G,\T^*\tb}\times \mu_{T,\T^*\tb}$ factors as a composition
$\T^*\tb\xrightarrow{q}\tg\to\g^*\times\t^*$.

Let $z:\xi\mto z\cdot\xi$ denote the standard $\BG_m$-action on a cotangent bundle
by  dilations  along the 
 fibers. 
We define a $\bullet$-action of $\BG_m$ on $\T^*G$, resp. $\g^*$, by $z:\ \xi\mto z\bullet\xi:=z^{-2}\cdot \Ad^*\gamma(z)(\xi)$,
where the $\Ad^*\gamma(z)$-action on $\T^*G$ is induced by $\gamma(z):\ g\mto \gamma(z)g\gamma(z)\inv$.
The action  on $\T^*G$ of the torus $T_l\times T_r\sset G_l\times G_r$  descends to a well-defined
action on $\T^*\tb$. Hence, the $\bullet$-action descends to  a well-defined
action on $\T^*\tb$.

The variety $\T^*\tb$ is quasi-affine; furthermore, it was shown in  \cite[ Section  5.5]{GR} that the algebra  $\C[\T^*\tb]$ is finitely generated.
Thus, $\taff:=\Spec\C[\T^*\tb]$ is an affine algebraic variety that contains $\T^*\tb$ as a Zariski open and
dense subvariety.  It is clear that any  action of a linear
algebraic group
on $\T^*\tb$ extends to an action on $\taff$.
The moment map $\mu_{G,\T^*\tb}\times \mu_{T,\T^*\tb}$ extends
to a  $G_l\times T_r$-equivariant morphism $\bar\mu_{G,\T^*\tb}\times \bar\mu_{T,\T^*\tb}:
\taff\to \g^*\times\t^*$.

The following result, \cite[ Section  5.5]{GR},  \cite{GK}, is a commutative counterpart of the Gelfand-Graev
construction mentioned in the introduction.

\begin{prop}\label{GR-prop} There is a $W$-action on $\taff$ such that:
\begin{enumerate}
\item The $T_r$-action and the $W$-action combine together to give
a $W\ltimes T$-action on  $\taff$.

\item The  $W\ltimes T$-action commutes  both with  the $G_l$-action and with the  $\bullet$-action
on $\taff$.

\item The map $\bar\mu_{G,\T^*\tb}^*\times \bar\mu_{T,\T^*\tb}$ is
$W\ltimes T$-equivariant, where $W$ acts naturally on $\t^*$ and $W\ltimes T$ acts trivially on $\g^*$.
\end{enumerate}
\end{prop}

\begin{rem} Let $\T\reg\tb:=\mu_{G,\T^*\tb}\inv(\g\reg)$. We have open embeddings
$\T\reg\tb\sset \T^*\tb\sset \taff$.
The set  $\T\reg\tb$ is   stable under the $W$-action,
but the set $\T^*\tb$ is  {\em not}.
\erem

\begin{lem}\label{treg-lem} The morphism $\mu_{G,\T^*\tb}|_{\T\reg\tb}:\ \T\reg\tb\to\g\reg$
is flat and we have $\bar\mu_{G,\T^*\tb}\inv(\g\reg)=\T\reg\tb$.
Further, let  $S\sset\g^*$ be a closed subvariety such that $S\sset\g\reg$, and
let $I_Y\sset \C[\g^*]$ be the corresponding ideal. Then, we have $\C[\mu_{G,\T^*\tb}\inv(S)]=
\C[\T^*\tb]/\C[\T^*\tb] I_S$.
\end{lem}
\begin{proof} To simplify the notation, we write $\mu=\mu_{G,\T^*\tb}$,
resp. $\bar\mu=\bar\mu_{G,\T^*\tb}$.
Let $\ca:=\mu_*\oo_{\T^*\tb}$.
This is a quasi-coherent sheaf of $\oo_{\g^*}$-algebras and we have
$\C[\taff]=\C[\T^*\tb]=\Ga(\g^*,\ca)$. 
Let  $i_x: \{x\}\into \g^*$ be a one point  embedding and $I_x\sset \C[\g^*]$ the
ideal of $x$. Then $\bar\mu\inv(x)$ is a closed affine subvariety of
$\taff$ and we have
\[\C[\bar\mu\inv(x)]=\C[\taff]/\C[\taff]I_x=
\Ga(\g^*,\ca)/\Ga(\g^*,\ca)I_x=i^*_x\ca.
\]

The map $\mu: \T\reg\tb\to\g\reg$ factors as $\pi|_{\tg\reg}\ccirc q$, where $q$ is a $T$-torsor,
$\pi$ is the Grothendieck-Springer map, and $\pi|_{\tg\reg}$ is a  finite morphism.
Therefore, $\T\reg\tb\to\g\reg$ is a flat affine morphism.
 Hence, for $x\in\g\reg$, the  fiber $\mu\inv(x)$ is affine.
Writing $i: \mu\inv(x)\into \T\reg\tb$ for the corresponding embedding, by base change,
we get $i^*_x\ca=\mu_*i^*\oo_{\T^*\tb}=\mu_*\oo_{\mu\inv(x)}$.
Thus, from  the previous paragraph, we deduce that
\[\C[\bar\mu\inv(x)]=i_x^*\ca=
\Ga(\mu\inv(x),\, i^*\oo_{\T^*\tb})=\C[\mu\inv(x)].
\]
We conclude that 
$\mu\inv(x)=\bar\mu\inv(x)$ for any $x\in\g\reg$.
The proof of the last statement of the lemma is very similar.
\end{proof}

The remainder of   Section  \ref{gklem-sec} will not be used in the paper and is only given for completeness.

Let $\tb\aff$ be the affine closure of $\tb$ and $\bar{\mathsf p}: \taff\to\tb\aff$ the canonical extension of
the projection ${\mathsf p}: \T^*\tb\to\tb$. For every $w\in W$, let $\bar{\mathsf p}_w:\taff\to\tb\aff$ be a map
defined by $\bar{\mathsf p}_w(y)=\bar{\mathsf p}(w(y))$. Thus, we have $\bar{\mathsf p}=\bar{\mathsf p}_1$.

\begin{prop}\label{bbaff} For $y\in \taff$ the following properties are equivalent:

\vi We have $\bar\mu_{G,\T^*\tb}(y)\in\g\reg$;

\vii  The 
$T_r$-orbit, 
$T_r\cdot y$, of $y$ is a closed orbit in $\taff$
of maximal dimension, i.e., $\dim (T_r\cdot y)=\dim T$;

\viii For every $w\in W$ we have $\bar{\mathsf p}_w(y)\in\tb$.
\end{prop}

Observe first that
the fibers of the map $\bar\mu_{G,\T^*\tb}\times \bar\mu_{T,\T^*\tb}$ are $T_r$-stable.
 We need the following standard result.

\begin{lem}\label{l1}
Every fiber of the map $(\bar\mu_{G,\T^*\tb}\times \bar\mu_{T,\T^*\tb}$ contains
exactly one closed $T_r$-orbit; furthermore
this orbit is contained in the closure
of any other $T_r$-orbit in that fiber.
\end{lem}
\begin{proof}
We have $\C[\taff]^{T_r}=\C[\T^*\tb]^{T_r}=\C[\tg]$.
The algebra $\C[\tg]$ may be identified with  the pullback of
$\C[\g^*\times_{\t^*/W}\t^*]$ via $\bar\mu_{G,\T^*\tb}\times \bar\mu_{T,\T^*\tb}$.
Hence, 
for any $s\in \g^*\times_{\t^*/W}\t^*=\Spec \C[\taff]^{T_r}$, the
fiber of the map $\taff\to \Spec \C[\taff]^{T_r}$ may be identified
with $(\bar\mu_{G,\T^*\tb}\times \bar\mu_{T,\T^*\tb})\inv(s)$.
The lemma follows since $T$-invariant functions on an affine variety separate
closed $T$-orbits.
\end{proof}

\begin{proof}[Proof of Proposition \ref{bbaff}]
Let $y\in \taff$ be such that $\bar\mu_{G,\T^*\tb}(y)\in\g\reg$.
Then, it follows from Lemma \ref{treg-lem}
that $y\in\T\reg\tb$. Hence,
the $T_r$-orbit of $y$ is a closed orbit of  maximal dimension.
Also,  for every $w\in W$, we  have $\bar{\mathsf p}_w(y)\in \tb$,
since the set  $\T\reg\tb$ is $W$-stable.

Next, fix $y\in\taff$  such that  the element $\bar\mu_{G,\T^*\tb}(y)$
is not regular.
We claim that the $T_r$-orbit of $y$ is 
either  not closed or it
does not have maximal dimension.
Let $F$ be the fiber of the map $\bar\mu_{G,\T^*\tb}\times \bar\mu_{T,\T^*\tb}: \taff\to \g^*\times_{\t^*/W}\t^*$
over the point $(\bar\mu_{G,\T^*\tb}\times \bar\mu_{T,\T^*\tb})(y)$.
It suffices to show that there are no
closed  $T_r$-orbits in $F$ of maximal dimension.
To prove this, let $F':=F\cap T^*\tb$.
All $T_r$-orbits in  $T^*\tb$, in particular those in $F'$,  have maximal
dimension. Further, since $\bar\mu_{G,\T^*\tb}(y)$ is not regular in $\g^*$,
the Springer fiber over  $\bar\mu_{G,\T^*\tb}(y)$ has dimension $> 0$.
It follows that $F'$, hence also $F$,  contains infinitely many $T_r$-orbits
of maximal dimension.
On the other hand,  Lemma \ref{l1}
says that $F$ contains
exactly one closed $T_r$-orbit $O$ and, moreover,
 $O$ is contained in the
closure of any other $T_r$-orbit in $F$. Hence, none of the infinitely
many $T_r$-orbits above is closed and $O$ is contained in the closure of each
of those orbits.
It follows that  the dimension of $O$ is not  maximal.

To complete the proof, let  $y\in\taff\sminus\T\reg\tb$ be as above and suppose by contradiction
that  $\bar{\mathsf p}_w(y)\in \tb$ for every $w\in W$.
Since
all orbits of the $T_r$-action 
on $\tb$ have maximal dimension,
the $T_r$-orbit of  $y$ has maximal
dimension. Therefore,
this orbit is not closed in $\taff$, by the paragraph above.
Hence, by Hilbert-Mumford,  there is a 
 $1$-parameter
subgroup $\tau: \gm\to T$
such that the limit $y':=\limp\ \tau(z) \cdot y$
exists and $y'\not\in T_r\cdot y$.
Hence, for every $w\in W$,
we have $\limp\ \tau(z) \cdot \bar{\mathsf p}_w(y)=\bar{\mathsf p}_w(y')$.

Let $h\in \Lie T=\t$ be a generator of the  $1$-parameter
subgroup $\tau$.
Observe that the existence of  a well-defined
limit
$\limp\ \tau(z) \cdot x \in \tb\aff$
for some $x\in \tb$ implies that $\mu(h)\geq0$ for
every dominant integral weight $\mu\in\t^*$.
Since the limit
$\limp\ \tau(z) \cdot  \bar{\mathsf p}_w(y)$
exists for every $w\in W$, it follows that we must
have  $\mu(h)\geq0$ for all integral weights $\mu\in\t^*$.
This forces  $h=0$, a contradiction.
\end{proof}

\subsection{Classical Hamiltonian reduction of $\T^*(G/\bar N)$}\label{Y-sec}
The following result goes back to Kostant.

\begin{lem}\label{gk-lem}
Let $x\in \bbe+ \b$ and  $\b_+\in\bb$ a Borel subalgebra  such that $x\in\b_+$.
Then the $ N$-orbit  through the point  $\b_+\in\bb$ is the (unique)
open $ N$-orbit in $\bb$.
\end{lem}

\begin{proof} 
Recall the $\Z$-grading \eqref{weight} and introduce  a $\Z$-filtration
$\g^{\leq k}:= \oplus_{\ell \leq k}\ \g(\ell),\ k\in\Z$,  on $\g$.
We have $\b=\g^{\leq 0}$, resp.  $\n=
\g^{\leq -2}$, and $\bar\b=\oplus_{\ell\geq0}\ \g(\ell)$.
 Also, we have $\bbf\in\g(-2)$, resp. $\bbe\in\g(2)$.

Now, let $x$ and $\b_+$ be as in the statement of the lemma. Conjugating the pair $(x,\b_+)$ 
by an element of $N$ if
necessary, we may  (cf.  Section  \ref{2d-sec}) and will assume that $x\in \bbe+\g_\bbf$.
Thus, since $\g_\bbf\sset  \g^{\leq -2}$, in $\gr\g$, we have $\gr x=x (\op{mod} \g^{\leq 0}) =\bbe$.
Observe next that
the filtration on $\g$ induces a filtration
$\b_+ \cap \g^{\leq k},\  k\in\Z$, on the
Borel subalgebra $\b_+ $.  It 
is clear that $\gr\b_+  $ is a solvable Lie subalgebra of $\gr\g$ that contains the element $\bbe=\gr x$. 
Since $\dim\gr\b_+ =\dim\b_+ $ and the Lie algebra $\gr\g$ is isomorphic to $\g$, we conclude
that  $\gr\b_+ $ is a Borel subalgebra of $\gr\g$. But the only  Borel subalgebra  that contains the element $\bbe$
is the algebra $\bar\b$. Therefore,
 we must have that $\gr\b_+  =\bar\b$.
It follows that $\gr_i\b_+  =0$  for all $i<0$.
We deduce that $\b_+ \cap\n=\b_+ \cap\g^{\leq -2}=0$,
and the lemma follows.
\end{proof}

\begin{lem}\label{GU-prop} Let $Y:=\mu_{G,\T^*\tb}\inv(\psi+\n^\perp)$.
Then, the image of $Y$ under the
projection ${\mathsf p}: \T^*\tb\to\tb$ is contained in the open set
$B\bar N/\bar N\sset G/\bar N=\tb$.
Furthermore, the map
\[\mu_{T,\T^*\tb}\times {\mathsf p}:\ Y\too\t^*\times 
B\bar N/\bar N=N\times T\times \t^*=N\times \T^*T\]
is   an $N$-equivariant isomorphism.
\end{lem}

\begin{proof} We  identify $\g^*$ with $\g$ and $\bb$ with $G/\BO$, respectively. Then, Lemma \ref{gk-lem} translates into
the first statement of Lemma \ref{GU-prop}. Further,
by Lemma \ref{kostant}  one has an  $N$-equivariant isomorphism 
$N\times \mu_{G,\T^*\tb}\inv(\s) \iso Y$.
Since $q : \T^*\tb\to\tg$ is  a $T$-torsor, we deduce 
that $\mu_{G,\T^*\tb}\inv(\s)$  is  a $T$-torsor over $\pi \inv(\s)$, where $\pi$ is the Grothendieck-Springer map.
Recall that  $\s\sset\g\reg  $ and $\tg\reg  \cong \g\reg  \times_\fc\t^*$. 
The map $\pi $ corresponds to the first projection $\g\reg  \times_\fc\t^*\to\g\reg  $. Hence, we obtain,
$
\pi \inv(\s)= \pi \inv(\s)\,\cap \,
(\g\reg \times_\fc\t^*)=
\s\times_\fc \t^*\cong \t^*
$.
We conclude that $\mu_{G,\T^*\tb}\inv(\s)$
is  a $T$-torsor over $\t^*$. The result follows.
\end{proof}

Let $\Y:=\T^*\tb\dss(N,\psi)$.  This means, according to  the definition given in  Section  \ref{2d-sec} and  Lemma \ref{GU-prop}, 
that the map $\mu_{T,\T^*\tb}\times {\mathsf p}$  yields
an isomorphism $p_\Y: \Y\iso \T^*T$. Further,
from Lemma \ref{treg-lem} in the case $S=\psi+\n^\perp$, we find
\[\C[\Y]=\C[Y]^N=\C[\mu_{G,\T^*\tb}\inv(\psi+\n^\perp)]^N=
\big(\C[\T^*\tb]/\C[\T^*\tb]\bnp\big)^N=\C[\T^*\tb]\dss\bnp.
\]
Thus, we have $\Y=\Spec(\C[\T^*\tb]\dss\bnp)$.


The subset $\psi+\n^\perp\sset\g^*$ is stable under the $\bullet$-action.
It follows from Proposition \ref{GR-prop}(2) that the  $W\ltimes T_r$-action, resp. 
 $\bullet$-action, on $\C[\T^*\tb]$ by algebra automorphisms descends to an action on the algebra $\C[\T^*\tb]\dss\bnp$. 
Therefore, this action induces a  $W\ltimes T_r$-action, resp. $\bullet$-action, on $\Y$.

Let $W$ act on $\T^*T=\t^*\times T$ diagonally, resp. $T$ act by translations along the second factor.

\begin{lem}\label{YT} The map $p_\Y$ is a $T\ltimes W$-equivariant
symplectomorphism $\Y\iso \T^*T$. Furthermore, this map intertwines the $\bullet$-action
on $Y$ with the $\BG_m$-action on $\T^*T$ by dilations by $z^{-2}$ along the factor $\t^*$.
\end{lem}

\begin{proof} All statements, except for the $W$-equivariance, are straightforward and are left
for the reader. Compatibility of the  $W$-actions  is a consequence of properties  of the $W$-action on $\C[\T^*\tb]$,
see  \cite[Lemma ~5.2.5]{GR}.
\end{proof}

\subsection{Quantum Hamiltonian reduction of $\dd(\tb)$}
\label{quantW}

One has   a chain of algebra isomorphisms 
\[ \dd( N T\Nn/\Nn)\dss\bnp\ccong \dd( N \times T)\dss\bnp\ccong
 (\dd( N)\dss\bnp)\o \dd(T)\ccong \C\o \dd(T)\ccong \dd(T).
\]

Let  $j:  N T\Nn/\Nn\into G/\Nn=\tb$ be the open embedding.

\begin{thm}\label{HDD} 
\vi The following functors are mutually inverse equivalences:
\beq{loc-diag}
\xymatrix{
(\dd_{\tb}, N,\psi)\mmod\ \ar@<0.5ex>[rr]^<>(0.5){\Ga(\tb,-)} &&\ (\dd({\tb  }), N,\psi)\mmod\
\ar@<0.5ex>[rrr]^<>(0.5){(-)^\bnp} \ \ar@<0.5ex>[ll]^<>(0.5){^{\dd_{\tb}
\o_{\dd(\tb  )}(-)}}&&&\
(\dd(\tb)\dss\bnp)\mmod.
\ar@<0.5ex>[lll]^<>(0.5){^{\dd(\tb)/\dd(\tb)\bnp\o_{\dd(\tb)\dss\bnp}(-)}}
}
\eeq

\vii The  map
$\dis j^*:\ \dd(\tb)\dss\bnp\to  \dd( N T\Nn/\Nn)\dss\bnp\cong\dd(T),$
 induced by restriction of differential operators via $j$, is an algebra isomorphism.

\end{thm}

We begin with the following result.

\begin{lem}\label{loc-lem} \vi For any $\cf\in (\dd_{\tb}, N,\psi)\mmod$ the canonical map $\cf\to j_\idot j^\hdot\cf$ is an isomorphism. 

\vii The functor $\dd_T\mmod\to (\dd_{\tb}, N,\psi)\mmod,\ M\mto j_\idot(e^\psi\Om_N\boxtimes M)$, is an equivalence.
An inverse is given by the functor $\cf\mto \big(pr_\idot  j^\hdot\cf\big)^\bnp$,
where  $pr: B\bar N/\bar N=N\times T\to T$ is the second projection.
\end{lem}
\begin{proof} The argument is an adaptation of the argument from \cite{MS} to the setting of not necessarily holonomic $\dd$-modules.
Given  a point $x\in \tb$, let $i_x: \{x\}\into \tb$ denote the embedding,
$N_x$  the stabilizer of $x$  in $N$, and $\psi_x$ the restriction of $\psi$ to the Lie algebra $\n\cap \Lie N_x$.
Note that if $\psi_x\neq0$, then the $\dd_{N_x}$-module $e^{\psi_x}\Om_{N_x}$
cannot be obtained as a pullback via the map $N_x\to\pt$  of a $\dd$-module on a point.
 Also, it is clear that for any $\ce \in(\dd_{\tb}, N,\psi)\mmod$ and $x\in\tb$, the $\dd$-module
$R^ki_x^!\ce$, on $\{x\}$, is an $(N_x,\psi_x)$-Whittaker $\dd$-module.
It follows that if $\psi_x\neq0$, then $R^ki_x^!\ce=0$ for all $k\in\Z$.

Now let
$\ce \in(\dd_{\tb}, N,\psi)\mmod$ be a nonzero $\dd_\tb$-module such that $\supp\ce\sset\tb\sminus B\bar N/\bar N$.
Then, for a sufficiently general point $x$ in the support of $\ce$, there exists $k$ such that $R^ki_x^!\ce\neq0$.
Write $x=g\bar N/\bN$, so we have
  $\Lie N_x=\Ad g(\bar\n)$. Lemma \ref{GU-prop}(ii)
implies that $\n\cap \Lie N_x=\Ad^*g(\bar\n)^\perp\cap (\psi+\n^\perp)=\emptyset$. This means that
$\psi_x\neq0$. It follows from the above that $\ce=0$, a contradiction.
Thus, there are no  nonzero objects $\ce \in(\dd_{\tb}, N,\psi)\mmod$ such that  $\supp\ce\sset\tb\sminus B\bar N/\bar N$.

To complete the proof of (i) observe that, for any   $\cf\in (\dd_{\tb}, N,\psi)\mmod$, the kernel, resp.
cokernel, of the map  $\cf\to j_\idot j^\hdot\cf$ is an object of  $\ce \in(\dd_{\tb}, N,\psi)\mmod$ supported on $\tb\sminus B\bar N/\bar N$.
Hence $\ce=0$, proving (i).
Further, it is immediate from definitions that the functors $e^\psi\Om_N\boxtimes (-)$
and $(pr_\idot (-))^\bnp$ give  mutually inverse equivalences
$\dd(T)\mmod\leftrightarrows (\dd_{N\times T}, N,\psi)\mmod$.
Therefore, part (ii)  follows from (i).
\end{proof}

\begin{proof}[Proof of Theorem \ref{HDD}]
First, we  apply Proposition \ref{a-lem} in the case where $A=\dd(\tb)$.
It follows from part (i) of the proposition that the functors
$(-)^\bnp$ and $\dd(\tb)/\dd(\tb)\bnp\o_{\dd(\tb)\dss\bnp}(-)$, on the right of diagram \eqref{loc-diag},
are mutually inverse equivalences.

Next, let $\cf\in (\dd_{\tb}, N,\psi)\mmod$.  
The variety $B\bar N/\bar N\cong T$ being affine,
from Lemma \ref{loc-lem}(i),
we deduce $H^i(\tb ,\cf)= H^i(\tb , j_\idot j^\hdot \cf)=H^i(B\bar N/\bar N,  j^\hdot\cf)=0$, for all $i\neq0$.
It follows that $\Ga(\tb,-)$ is an exact functor on $(\dd_{\tb}, N,\psi)\mmod$. Furthermore, we have
\beq{glob}
\Ga(\tb, \cf)^{\n^\psi}=\Ga(\tb, j_\idot j^\hdot\cf)^{\n^\psi}=\Ga(B\bar N/\bar N,  j^\hdot\cf)^{\n^\psi}
=\Ga(T,\, pr_\idot j^\hdot\cf)^{\n^\psi}=\Ga(T,\, (pr_\idot j^\hdot\cf)^{\n^\psi}).
\eeq
We see that the action of the algebra $\dd(\tb)\dss\bnp$ on  $\Ga(\tb, \cf)^{\n^\psi}$ factors
through the  restriction map $j^*: \dd(\tb)\dss\bnp\to \dd(T)$. Moreover, the
resulting functor $\Ga(\tb, \cf)^{\n^\psi}: (\dd_{\tb}, N,\psi)\mmod\to \dd(T)\mmod$ is an equivalence,
by Lemma \ref{loc-lem}(ii), 
and an inverse equivalence is the functor
$j_\idot(e^\psi\Om_N\boxtimes-)$.

We apply $\Ga(\tb,-)^\bnp$, an exact functor on $(\dd_{\tb}, N,\psi)\mmod$, to the morphism
 $\dd_\tb\o\bnp\to \dd_\tb$  given by right multiplication
by $\bnp$. The cokernel of this morphism equals 
$\dd_\tb/\dd_\tb\bnp$. We deduce that $\Ga(\tb,  \dd_\tb/\dd_\tb\bnp)^\bnp=(\dd(\tb)/\dd(\tb)\bnp)^\bnp=
\dd(\tb)\dss\bnp$.
On the other hand, it is clear that $j^\hdot(\dd_\tb/\dd_\tb\bnp)=e^\psi\Om_N\boxtimes \dd_T$.
Hence,  $(pr_\idot j^\hdot(\dd_\tb/\dd_\tb\bnp))^\bnp=(pr_\idot j^\hdot(e^\psi\Om_N\boxtimes \dd_T)^\bnp=\dd_T$. Thus, using
\eqref{glob}, we obtain 
\[\dd(\tb)\dss\bnp=\Ga(\tb,\, \dd_\tb/\dd_\tb\bnp)^\bnp=\Ga(T,\, pr_\idot j^\hdot(\dd_\tb/\dd_\tb\bnp))^\bnp=\Ga(T,\dd_T)=\dd(T).
\]

Thus, we have shown that 
the restriction map $j^\hdot: \dd(\tb)\dss\bnp\to \dd(T)$, in Theorem \ref{HDD}(ii),
is an algebra isomorphism and
the functor $\Ga(\tb, -)^{\n^\psi}:
(\dd_{\tb}, N,\psi)\mmod\to \dd(T)\mmod$ is an equivalence.
We conclude that both the  functor $(-)^{\n^\psi}$ and the composite  functor $\Ga(\tb, -)^{\n^\psi}$, in  diagram \eqref{loc-diag}, 
are equivalences. It follows that 
the functor $\Ga(\tb, -)$, in  diagram \eqref{loc-diag},  is also an equivalence.
This proves 
part (i) of Theorem \ref{HDD}. 
\end{proof}


\subsection{Proof of Theorem \ref{psi-exact}}\label{GG1}
We begin with  the following result.

\begin{lem}\label{all}
The following natural maps   are algebra isomorphisms
\begin{gather*}
(\dd(G)\dss \n_l^\psi)\dss\bar\n_r \ \xrightarrow{u}\  \dd(G)\dss(\n_l^\psi+\bar\n_r)\
\ \xleftarrow{v}\ (\dd(G)\dss\bn_r)\dss\n_l^\psi;\\
(\dd(G)\dss \n_l^\psi)\dss \n_r^{\psi} \ \xrightarrow\  \dd(G)\dss (\n_l^\psi+\n_r^{\psi})\
\xleftarrow\ (\dd(G)\dss\n_r^{\psi})\dss\n_l^\psi;\\
\bn_r\drs\big(\dd(G)\dss\n_l^\psi\big)\
\xrightarrow\ \big(\dd(G)/(\bn_r\dd(G)+\dd(G)\n_l^\psi)\big)^{N_l\times\bN_r}\
\xleftarrow\ \big(\bn_r\drs\dd(G)\big)\dss\n_l^\psi. 
\end{gather*}
\end{lem}

\begin{proof} 
Let $N=N_l$, resp. $\bN=\bN_r$. First, we prove that 
the  map $v$ is an  isomorphism.  
To prove this, consider an open embedding $j_G:\ NT\bN/\bN\into G/\bN$. 
We have a commutative diagram
of  algebra homomorphisms
\[
\xymatrix{
\dd(G/\bN)\dss\bnp\ \ar[d]^<>(0.5){j^*}\ar[rr]^<>(0.5){v}&&\  \dd(G)\dss(\n_l^\psi+\bn_r)\ \ar[d]^<>(0.5){j^*_G}\\
\dd(NT\bN/\bN)\dss\bnp\ \ar@{=}[r] &\ \dd(T)\ \ar@{=}[r] & \ \dd(NT\bN/\bN)\dss(\n_l^\psi+\bn_r).
}
\]
By commutativity, it suffices to show that the map $j_G^*$ on the right
is  an algebra isomorphism. 
The latter  is proved by mimicking  the arguments of  Section  \ref{quantW}.

Next, we prove that that the map $u$ is an  isomorphism.  
We simplify the notation and write $\dd=\dd(G)$,
resp.  $M=\dd /\dd \n_l^\psi$, and $ D_l=\dd(G)\dss\n_l^\psi=M^{\n_l^\psi}$.
The action of $\n_r^\psi$ on $\dd $ by right multiplication
descends to $M$.
Observe that $M$  is an object of the category
$(\dd ,\n_l^\psi)\mmod$. 
In that category,
one has an exact
sequence 
\beq{star1}M\o \n_r^\psi\xrightarrow{a} M\to M/M\n_r^{\psi}\to0,
\eeq
where the map $a$ is given by
the $\n_r^\psi$-action  on $M$. 
We apply the functor $(-)^{\n_l^\psi}$, which is   an exact functor on 
$(\dd ,\n_l^\psi)\mmod$, to \eqref{star1}. We have
\begin{gather*}
\big(M/M\n_r^{\psi}\big)^{\n_l^\psi} =\big((\dd /\dd \n_l^\psi)/(\dd /\dd \n_l^\psi)\n_r^{\psi}\big)^{\n_l^\psi} =
\big(\dd /\dd (\n_l^\psi+\n_r^{\psi})\big)^{\n_l^\psi},\\
(M\n_r^{\psi} )^{\n_l^\psi}=(\im a)^{\n_l^\psi}=\im[(M\o \n_r^\psi)^{\n_l^\psi}\to
M^{\n_l^\psi}]=(M^{\n_l^\psi})\n_r^{\psi}=
 D_l\n_r^{\psi}.
\end{gather*}
Thus, from the exact sequence \eqref{star1}, we obtain the following exact sequence in $(\dd ,\, \n_l^\psi+\n_r^{\psi})\mmod$:
\beq{star}
0\to  D_l\n_r^{\psi}\to  D_l\to \big(\dd /\dd (\n_l^\psi+\n_r^{\psi})\big)^{\n_l^\psi}\to0.
\eeq

We now apply  the functor 
$(-)^{\n_r^\psi}$, which is  an exact functor on  $(\dd ,\, \n_l^\psi+\n_r^{\psi})\mmod$,
 to \eqref{star}. 
It is immediate to see that the exactness of the resulting sequence implies  that the map $u$ is an isomorphism. 
All other isomorphisms of the lemma are proved in a similar way.
\end{proof}


\begin{proof}[Proof of Theorem \ref{psi-exact}]
Let  $D_r:=\dd(G)\dss\n^\psi_r$. The embedding $\n_l^\psi\into \dd(G)$ induces a map
$\n^\psi\to D_r$.
The functor
$(D_r,\bnp )\mmod\to (D_r\dss\n^\psi)\mmod,\ M\mto M^{\n^\psi}$, is an equivalence, by Proposition \ref{a-lem}.
Further, by Lemma \ref{all}, we have an isomorphism
$D_r\dss\n^\psi\cong\hpp$. The result now follows from Theorem \ref{gpsi}.
\end{proof}

\subsection{The Gelfand-Graev action}\label{GG}
Let  $\rho$ be  the half-sum of positive roots. 
Let $W$ act on ${\sym\t}$ via the dot-action:
$w\cdot a:=w(a)+\langle w(a)-a,\rho\rangle,\ w\in W,\ a\in\t$.
We will identify  $Z \g$ with ${\sym\t}^{W\cdot}$ via  the Harish-Chandra isomorphism.
The differential of the $G\times T$-action on $\tb$  gives an algebra map
$\ug\to\dd(\tb),\ a\mto a_l$, resp. $\ut={\sym\t}\to \dd(\tb),\ a\mto a_r$. 
It is known that the map $a\o a'\mto a_l\o a'_r$ factors through a homomorphism
$\ug\o_{Z\g}\sym\t\to\dd(\tb)$, cf. \cite{BB}.

Restricting the $G\times T$-action  on $\tb$ to the subgroup 
$T\times T\sset G\times T$ 
 gives a weight
decomposition $\C[\tb]=\oplus_{\mu,\la\in\BX^{^*}}\ {}^{(\mu)}\C[\tb]^{(\la)}$,
resp. $\C[\T^*\tb]=\oplus_{\mu,\la\in\BX^{^*}}\ {}^{(\la)}\C[\T^*\tb]^{(\mu)}$.
Further, the  $\BG_m$-action on $\tb$ given by
$\gamma(z): g\bar N\mto \gamma(z)g\gamma(z)\inv\bN$
induces a $\Z$-grading on $\dd(\tb)$. This $\Z$-grading 
may be expressed in terms of the weight decomposition as follows:
\[\dd(\tb)=\oplus_{\ell\in\Z}\ \dd(\tb)(\ell),\quad \dd(\tb)(\ell):=\oplus_{\{\mu,\la\in\BX^*\mid \langle\la,\bbh\rangle+\langle\mu,\bbh\rangle=\ell\}}\ {}^{(\mu)}\dd(\tb)^{(\la)}.
\]
Let $F_\idot\dd(\tb)$ be the Kazhdan filtration  associated with the $\Z$-grading
 and filtration $\dd_{\leq}(\tb)$, as in  Section  \ref{U-sec}.

We have maps $\bnp=\n_l^\psi\to \ug_l\to \dd(\tb)$.
The right $T$-action  on $\dd(\tb)$ survives the  Hamiltonian reduction by $\bnp$.
We get
an $\BX^*$-grading
$\dd(\tb)\dss\bnp=\oplus_{\mu\in\BX^*}\ (\dd(\tb)\dss\bnp)^{(\mu)}$.
The Kazhdan filtration  induces a $\Z$-filtration $F_\idot (\dd(\tb)\dss\bnp) $ which is compatible with the $\BX^*$-grading.
 Thus, we have
$F_n (\dd(\tb)\dss\bnp)  =\oplus_{\mu\in\BX^*}\ F_n(\dd(\tb)\dss\bnp)^{(\mu)}$.
The  (Kazhdan shifted) principal symbol map $\gr^F_\idot\dd(\tb)\to \C[\T^*\tb]\dss\bnp$ 
descends to a map $\gr^F(\dd(\tb)\dss\bnp)  \to \C[\T^*\tb]\dss\bnp$,
of  $\Z\times\BX^*$-graded algebras.

\begin{prop}\label{gr-tb} The map $\gr^F \dd(\tb)\dss\bnp  \to \C[\T^*\tb]\dss\bnp$\
is an  isomorphism.
\end{prop} 
\begin{proof} According to \cite[Corollary 3.6.1]{GR}, the principal symbol map
$\gr\dd(\tb)\to\C[\T^*\tb]$ is a graded algebra isomorphism,
where $\gr(-)$ is taken with respect to the order filtration on differential operators.
This isomorphism respects the $\BX$-gradings. Therefore,  it gives a 
 graded algebra isomorphism $\gr^F\dd(\tb)\cong \C[\T^*\tb]$, where
the grading on $\C[\T^*\tb]$ comes from the  $\bullet$-action.
Further, an argument similar to the one used in Section \ref{class-intro-pf}
shows that the Kazhdan filtration on $\dd(\tb)$ is separating.
The statement of the proposition now follows from the isomorphism of Lemma  \ref{afilt}(1)
in the case where ~${A=\dd(\tb)}$.
\end{proof}

\bigskip
In the 1960's,  Gelfand and Graev constructed, using partial Fourier transforms,  a  Weyl group action on $\dd(\tb)$ by algebra
automorphisms. 
The  Gelfand-Graev action on $\dd(\tb)$ commutes with the $G$-action, in particular with the $N$-action,  by left translations.
Therefore, the Gelfand-Graev action  descends
to a $W$-action on $\dd(\tb)\dss\bnp$  by algebra
automorphisms. 
%


We view the function $t^\rho$ as a zero order differential operator on $T$ and
define a $W$-action on the algebra $\dd(T)$ by the formula $w\cdot u= w(t^\rho\ccirc u\ccirc t^{-\rho})$.
Here $w(-)$ denotes the {\em natural} action of $w\in W$ on $\dd(T)$ induced by
the action of $w$ on $T$. Note that $w\cdot u=w(u)$ for any $u\in\C[T]\sset \dd(T)$.
Also, the embedding $\sym\t=\ut\into \dd(T)$ intertwines the dot-actions on $\sym\t$ and ~$\dd(T)$.

It is clear that the map $j^*: \dd(\tb)\dss\bnp\to \dd(B\bN/\bN)\dss\bnp=\dd(T)$
takes $F_n(\dd(\tb)\dss\bnp)$ to $\dd_{\leq n}(T)$ for all $n$.

\begin{prop}\label{w-act} The isomorphism $j^*: \dd(\tb)\dss\bnp\to \dd(T)$  of Theorem \ref{HDD} induces
an isomorphism $\gr^F j^*: \gr^F(\dd(\tb)\dss\bnp)\iso  \gr\dd(T)=\C[\T^*T]$ of graded algebras. Furthermore,
the map $j^*$
intertwines the $W$-action
on $\dd(\tb)\dss\bnp$ and the dot-action on  $\dd(T)$.
\end{prop}

To prove the proposition, recall that $v_\la\in V_\la$ denotes a  nonzero $\Nn$-fixed vector 
of the irreducible representation $V_\la$ such that $tv_\la=\la(t)v_\la,\ \forall t\in T$.
Let  $v^*_{-\la}\in V_\la^*$ be a nonzero $ N$-fixed  vector
and
 $f^\la(g):=\langle v^*_{-\la}, gv_\la\rangle$. This function is right $\Nn$-invariant, hence
it descends to a function $f^\la$ on $\tb$, such that
$f^\la(n\cdot t\cdot\Nn/\Nn)=\la(t)$, for all $t\in T,\, n\in  N,\, \bar n\in\Nn$. 

To simplify the notation, put ${\mathsf A}=\dd(\tb)\dss\bnp$.
The function $f^\la$, viewed as an element of $\dd(\tb)$, 
survives the Hamiltonian reduction giving  an element $f^\la_{\mathsf A}\in  {\mathsf A} $.
Let
$\F_w,\ w\in W$ denote the Gelfand-Graev automorphisms of the algebra $\dd(\tb)$, resp. ${\mathsf A}$.

\begin{proof}[Proof of Proposition \ref{w-act}] 
The map    ${\sym\t} \to \dd(\tb),\ u\mto u_r$ survives  the  Hamiltonian reduction by $\bnp\sset\dd(\tb)$, giving 
an algebra map  ${\sym\t}\to {\mathsf A},\ u\mto u_{\mathsf A}$. It is clear that for any $u\in\sym\t$,
resp. $f^\la_{\mathsf A},\ \BX^{++}$, we have $j^*(u_{\mathsf A})=u$, resp. $j^*(f^\la_{\mathsf A})=t^\la$.
The elements $\{u,\, t^\mu \mid u\in\sym\t,\ \mu\in\BX^{++}\}$ generate the algebra $\dd(T)$.
Therefore, the elements $\{u_{\mathsf A},\, f_{\mathsf A}^\mu \mid u\in\sym\t,\ \mu\in\BX^{++}\}$ 
generate the algebra ${\mathsf A}$, by  Theorem \ref{HDD}.

Let $\ord u$ denote the order of a differential operator $u\in\dd(\tb)$.
Clearly, one has  $f^\mu\in {}^{-\mu}\dd_{\leq0}(\tb)^{\mu}$.
Thus, we have
\[
f^\mu_{\mathsf A}\in F_0  {\mathsf A} ^{(\mu)},\en\ \forall \mu\in\bxp,\quad\text{resp.}
\quad 
u_{\mathsf A}\in F_{2n} {\mathsf A} ^{(0)},\en\ \forall u\in\sym^n\t.
\]
Furthermore, $\gr_0(f^\mu_{\mathsf A})\neq0$, resp. $\gr_{2n}(u_{\mathsf A})\neq 0$.
Let $\op{symb}$ be the isomorphism of Proposition \ref{gr-tb},
and recall the isomorphism $p_\Y$ from  Section  \ref{Y-sec}.
One has the following diagram of graded algebra maps:
 \beq{gr-A}
\xymatrix{
\C[\T^*T]\ &\ \C[\T^*\tb]\dss\bnp\ \ar[l]_<>(0.5){p_\Y}^<>(0.5){^{\cong}}& \ \gr {\mathsf A}\ \ar[l]^<>(0.5){^{\cong}}_<>(0.5){\op{symb}}
\ar[rr]^<>(0.5){\gr j^*}&& \ \gr\dd(T)=\C[\T^*T].
}
\eeq
Let $p_{\mathsf A}=p_\Y\ccirc\op{symb}$.  It is immediate from the construction of the maps that,  for any element
$a$ from our generating set $\{u_{\mathsf A},\, f_{\mathsf A}^\mu \mid u\in\sym\t,\ \mu\in\BX^{++}\}$ of ${\mathsf A}$,
one has 
$ p_{\mathsf A} (\gr a)=(\gr j^*)(\gr a)$. 
It follows that the map $(\gr j^*)\ccirc  p_{\mathsf A}\inv: \C[\T^*T]\to\C[\T^*T]$ is the identity.
We deduce that  the map $\gr j^*$ is a graded algebra isomorphism, proving the first statement of the proposition.

For any, not necessarily dominant,  $\la\in\BX^*$, we let $f_{\mathsf A}^{(\la)}$ be the preimage of $t^\la$ under the isomorphism
$j^*$ of Theorem \ref{HDD}. This agrees with the notation $f_{\mathsf A}^\la$ for dominant $\la$.

It is known that the map  ${\sym\t} \to \dd(\tb),\ u\mto u_r$, intertwines  the dot-action on ${\sym\t}$
and the Gelfand-Graev action on $\dd(\tb)$, cf. \cite{BBP}. 
It follows that the map $u\mto u_{\mathsf A}$  intertwines  the dot-action on ${\sym\t}$
and the $W$-action on ${\mathsf A}$. Thus, to prove that the map $j^*$  is $W$-equivariant, it remains to
show that for any $\mu\in \BX^{++}$ and $w\in W$, one has
$\F_w(f^\mu_{\mathsf A}))=f_{\mathsf A}^{w(\mu)}$. First, 
it is immediate from the construction of the action that for any $\la\in\BX^*$ and $w\in W$, the map $\F_w$ yields
an isomorphism $\dd(\tb)^{(\la)}\iso \dd(\tb)^{(w(\la))}$, cf. \cite{BBP}, \cite{GR}. 
Also,  it was shown in the course of the proof of \cite[Lemma 3.18]{BBP}, that one has
\[
\ord \F_w(u)= \ord u+\half \langle \la, \bbh\rangle-\half \langle w(\la), \bbh\rangle,
\qquad \forall\,\la\in \BX^*,\,u\in\dd(\tb)^{(\la)},\ w\in W.
\]
It follows that
the automorphisms $\F_w$, of $\dd(\tb)$, respect the Kazhdan filtration.
Hence, the induced automorphisms of the algebra $ {\mathsf A} $ have similar
propertities. In particular, the map $\F_w: {\mathsf A}^{(\la)}\to {\mathsf A}^{(w(\la))}$ is an isomorphism
and the $W$-action on $ {\mathsf A} $ induces   a $W$-action on  $\gr^F H$
by $\Z$-graded algebra automorphisms.

Next,
we observe that for any $\la\in\BX^*$, the space $\dd(T)^{(\la)}$ is a rank-one free
(say, left) $\sym\t$-module with generator $t^\la$.
Therefore,  the space ${\mathsf A}^{(\la)}$ is a rank one free 
$\sym\t$-module with generator $f^\la_{\mathsf A}$. 
On the other hand, we know that the map $\F_w: {\mathsf A}^{(\mu)}\to {\mathsf A}^{(w(\mu))}$ is an isomorphism and,
for all $u\in\sym\t,\ \mu\in\BX^{++}$, we have $\F_w(f^\mu_{\mathsf A} u_{\mathsf A})=\F_w(f^\mu_{\mathsf A}) (w\cdot u)_{\mathsf A}$.
It follows that the element $\F_w(f^\mu_{\mathsf A})$ is  a  generator of ${\mathsf A}^{(w(\mu))}$ as a 
$\sym\t$-module.  The generator of  a rank-one free
$\sym\t$-module is determined uniquely up to a nonzero constant factor.
We deduce that
$\F_w(f_{\mathsf A}^\mu)=c_{w,\mu}\cdot f_{\mathsf A}^{w(\mu)}$, where $c_{w,\mu}\in\C$ is a nonzero constant.
To complete the proof of  $W$-equivariance of the map $j^*$,
we must show that $c_{w,\mu}=1$ for all $w\in W$ and $\mu\in\BX^{++}$. 
We know that the elements
$\F_w(f_{\mathsf A}^\mu)$ and $f_{\mathsf A}^{w(\mu)}$ both belong to $F_0 {\mathsf A} ^{(w(\mu))}$. Hence,
in $\gr_0 {\mathsf A}$, we have an equation $\gr(\F_w(f_{\mathsf A}^\mu))=c_{w,\mu}\cdot \gr f_{\mathsf A}^{w(\mu)}$.
Further, it follows from \cite[Proposition 5.4.1]{GR} that the isomorphism
$ p_{\mathsf A} =p_\Y\ccirc\op{symb}$, in \eqref{gr-A}, respects the $W$-actions.
Therefore, we compute
 \begin{align*}
 p_{\mathsf A} (\gr(\F_w(f_{\mathsf A}^\mu)))&= p_{\mathsf A} (\F_w(\gr f_{\mathsf A}^\mu))=w\cdot p_{\mathsf A} (\gr f_{\mathsf A}^\mu)\\
&=
w\cdot t^\mu =t^{w(\mu)}=(\gr j^*)(\gr f_{\mathsf A}^{w(\mu)})=p_{\mathsf A}(\gr f_{\mathsf A}^{w(\mu)}).
\end{align*}
It follows that $\gr(\F_w(f_{\mathsf A}^\mu))=\gr f_{\mathsf A}^{w(\mu)}$, hence $c_{w,\mu}=1$, and we are done.
\end{proof}

\section{The Miura bimodule}\label{miura-Sec} 
\subsection{Proof of Theorem \ref{nnbar}}\label{nnbar-sec}
In this subsection, we consider the actions of $N$ and $\bN$ on $G$ by right translations.
Let  $  D_r =\dd(G)\dss\n_r ^\psi$.
The natural projection $p: G\to G/\bN_r $ is an affine morphism.
Therefore,   for any open affine  $U\sset G/\BN_r $
the set $p\inv(U)$ is affine. 
Let $A_U:=\Ga(p\inv(U), \dd_G)$. Given  a left $D_r$-module $F$, put
$F_U=A_U\dss\n_r^\psi\o_{  {D_r} }F$. 
We have
$A_U\o_{\dd(G)}(\dd(G)/\dd(G)\n_r^\psi\,\o_{  {D_r} }\,F)\ =\
A_U/A_U\n_r^\psi\,\o_{A_U\dss\n_r^\psi}\,F_U$, since $A_U$ is flat over $\dd(G)$.
Applying Proposition \ref{a-lem}(ii) in the case $A=A_U$ we deduce
that $H_i(\bn_r ,\ A_U/A_U\n_r ^\psi\,\o_{A_U\dss\n_r^\psi}\,F_U)=0,\ \forall i\neq 0$, and
$H_0(\bn_r,\ A_U/A_U\n_r^\psi\,\o_{A_U\dss\n_r^\psi}\,F_U)=\ut\o_{Z\g}F_U$.

We have  a chain of equivalences of categories of left modules
\beq{dnn}  {D_r} \mmod=\dd(G)\dss \n_r ^\psi\mmod\ccong(\dd(G), N_r ,\psi)\ccong
(\dd_G,N_r ,\psi)\mmod.
\eeq
Let $\cf\in (\dd_G,N_r ,\psi)\mmod$ be the object that corresponds to  an object $F\in   D_r \mmod$ via the
above equivalences.
Thus, $\Ga(p\inv(U), \cf)=A_U/A_U\n_r ^\psi\,\o_{A_U\dss\n_r ^\psi}\,F_U$.
For any  $j\in \Z$, by the definition of the functor 
\[\int^{\text{derived}}_p:\ \dd_G\mmod\ =\ 
\dd^{op}_G\mmod\too\dd^{op}_\tb\mmod\ =\ \dd_\tb\mmod,\]  
one has $\Ga(U, \int_p^j\cf)=H_{-j}\big(\bn_r,\, \Ga(p\inv(U), \cf)\big)$. 
We deduce that  $\Ga(U, \int_p^j\cf)=0,\ \forall j\neq 0$, and
  there is an isomorphism $\Ga(U, \int^0_p\cf)\cong \ut\o_{Z\g}F_U$.
In particular, the functor$\int^0_p:   D_r \mmod\to \dd_{G/\bN_r}\mmod$ is exact.

To complete the proof, we recall that  the averaging functor 
$\cf\mto \int^{\text{derived}}_{a_{\bar N_r}} (\Om_{\bar N_r}\times \cf)$
on the derived category of $\dd_G$-modules is isomorphic to the functor $p^*\ccirc\int^{\text{derived}}_p$,
cf.  \cite[ Section  2.5]{MV}. 
The functor $p^*:\ \dd_{G/\BN_r}\mmod\to (\dd_G, \bar N_r)\mmod$  is well-known to be an equivalence,
and the theorem follows.\qed

\bigskip
Let
$\mmg=\bn_r\dd_G\bac\dd_G/\dd_G\n_r^\psi$. This sheaf of coinvariants is a sheaf of vector spaces on $G$, to be called the
{\em Miura sheaf}. Using an isomorphism $\dd_G\cong \oo_G\o\ug$ and the PBW theorem for $\ug$, it is easy to show that
 for any affine open
$V\sset G$, one has $\Ga(V,\mmg)=\bn_r\dd(V)\bac\dd(V)/\dd(V)\n_r^\psi$.
According to   Section  \ref{q-red}, the  sheaf $p_\idot\mmg$
has the natural structure of a 
$\big(\bn_r\drs(p_\idot\dd_G),\,\dd(G)\dss\n_r^\psi\big)$-bimodule. 
Recall that $\dd(G)\dss\n_r^\psi=D_r$. Also,
by  Section  \ref{q-red}, one has algebra isomorphisms $\bn_r\drs(p_\idot\dd_G)\cong (p_\idot\dd^{op}_G)\dss\bn_r^{op}\cong (p_\idot\dd_G)\dss\bn_r\cong \dd_\tb$.
Thus, we can view $p_\idot\mmg$ as a $(\dd_\tb,   D_r)$-bimodule. 
By Proposition \ref{a-lem}(ii), there is a canonical isomorphism
\beq{mm-co}p_\idot\mmg\ccong \ut\,\o_{Z\g}\, (p_\idot\dd_G)\dss\n_r^\psi,
\eeq
of sheaves of $\big(\ut,  D_r \big)$-bimodules.
Furthermore, the proof of Theorem \ref{nnbar}
shows that  on the category $D_r\mmod$, there is an isomorphism of  functors
\beq{int-m}
 \int_p (\dd_G/\dd_G\n_r^\psi\o_{D_r}-)\ccong p_\idot\mmg\o_{D_r}(-).
\eeq

\subsection{From $\hpp$-modules to $\dd(T)$-modules}\label{miura-sec} 
We are going to  mimic formulas \eqref{mm-co}-\eqref{int-m}  in the setting where the ring
$\dd(G)$ is replaced by 
the ring $ D_l:=\dd(G)\dss\n^\psi_l$.

View $ D_l$ as an $(\UU\bn_r, \UU\n_r)$-bimodule,
where the left, resp. right, action is provided by left, resp. right,  multiplication
inside the algebra $ D_l$ by the elements of $\bn_r$, resp. $\n_r$.
The analogue of the Miura sheaf $\mmg$ is played by $\BM:=\bn_r D_l\backslash D_l/  D_l\n^\psi_r$.
By  Section  \ref{q-red},  this space of coinvariants
has 
 the structure of
a $(\bnr\drs D_l,\,  D_l\dss\n^\psi_r)$-bimodule. 
By Lemma \ref{all}, we have $D_l\dss\n^\psi_r\cong\hpp$, and there is a chain of
algebra isomorphisms
\[\bnr\drs D_l=\bnr\drs\big(\dd(G)\dss\n^\psi_l\big)=\big(\bnr\drs\dd(G)\big)\dss\n^\psi_l=\big(\dd(G)\dss\bnr\big)\dss\n^\psi_l
=\dd(\tb)\dss\n^\psi_l=\dd(T),
\]
where the last isomorphism comes from Theorem \ref{HDD}(i). We conclude that $\BM$ has the structure  of
a $(\dd(T),\, \hpp)$-bimodule.
Further, according to Proposition \ref{a-lem}(ii), one has  natural
isomorphisms   \[\BM=\bn_r D_l\backslash D_l/  D_l\n^\psi_r\cong
\sym\t\,\o_{_{Z\g}}\,( D_l\dss\n^\psi_r)\cong \sym\t\o_{_{Z\g}} \hpp.\]

We let the group $W$ act on $\sym\t\,\o_{_{Z\g}}\,\hpp$ by
$w(a\o h)=(w\cdot a)\o h$, where $a\mto w\cdot a$ is the dot-action.
 It follows from the construction that the left action
$\dd(T)\o \BM\to \BM$ is a $W$-equivariant map.
Hence, this action  can be extended to a left action of
$W\ltimes\dd(T)$ on $\BM$.
Thus, $\BM$ acquires the structure of a   $(W\ltimes\dd(T),\,\hpp)$-bimodule,
to be called the {\em Miura bimodule}.

For any $\hpp$-module $F$, the  $W$-action on $\BM$ gives the $\dd(T)$-module  $\BM\o_{\hpp} F$ a   $W$-equivariant structure.
Thus we obtain a functor $\hpp\mmod\to W\ltimes\dd(T)\mmod,\ F\mto\BM\o_{\hpp} F$.
This functor is exact since  $\BM\o_{\hpp} (-)=(\sym\t\o_{_{Z\g}}\hpp)\o_\hpp(-)=\sym\t\o_{_{Z\g}}(-)$,
 and $\sym\t$ is free over $Z\g$.

\begin{prop}\label{M-exact} The  functor $\hpp\mmod\to W\ltimes\dd(T)\mmod,\ F\mto \BM\o_{\hpp} F$,
takes holonomic $\hpp$-modules to holonomic $\dd(T)$-modules.
\end{prop}
\begin{proof} 
Using an analogue of the composition of the chain of equivalences in \eqref{dnn}, we obtain
an equivalence $(\dd_G, N_l\times N_r,\psi\times\psi)\mmod\iso\hpp\mmod$,
to be denoted ${\mathbf I}$.
Further,  Theorem \ref{HDD}(ii) yields an equivalence $(\dd_\tb, N_l,\psi)\mmod\iso\dd(T)\mmod$,
which we denote by $\mathbf J$.
Recall the notation $D_l=\dd(G)\dss\n_l^\psi$, resp. $D_r=\dd(G)\dss\n_r^\psi$, and  the projection $p: G\to \tb=G/\bar N_r$.

We consider the following diagram of functors
where $\dd$ stands for $\dd(G)$ and horizontal inclusions in the middle of the diagram are the natural full embeddings:

\[
\xymatrix{
\hpp\mmod\ \ar@{=}[r]^<>(0.5){{\mathbf I}}\ar[d]^<>(0.5){\BM\o_{\hpp}(-)}&\ 
(\dd,N_l\times  N_r,\psi\times\psi )\mmod\ 
\ar[d]^<>(0.5){\int^0_p}\ar@{^{(}->}[r]&
\ (\dd, N_r,\psi)\mmod\ \ar[d]_<>(0.5){\int^0_p}\ar@{=}[r]&
\ D_r\mmod\ar[d]_<>(0.5){p_\idot\mmg\o_{D_r}(-)}
\\
\dd(T)\mmod\ \ar@{=}[r]^<>(0.5){\mathbf J}&\
(\dd,N_l\times \bN_r,\psi\times0)\mmod\ \ar@{^{(}->}[r]&\
(\dd,\bN_r)\mmod\ \ar@{=}[r]&\ \dd_{G/\bN_r}\mmod
}
\]

Using that  $\mathbf J$ is an  equivalence   and formula \eqref{int-m}, it  is easy to verify
 that two composite functors $\hpp\mmod\to \dd_{G/\bN_r}\mmod$, along the perimeter of the diagram, are isomorphic.
It follows that the  functor $\int^0_p\ccirc {\mathbf I}$ is isomorphic
to the functor ${\mathbf J}\ccirc \BM\o_{\hpp}(-)$.
This implies the statement of the proposition 
since the functor
$\int^0_p$ sends  holonomic $\dd$-modules to  holonomic $\dd$-modules.
\end{proof}

\begin{rem}
For $\cf\in (\dd_G,N_l\times N_r,\psi\times\psi )\mmod$,
the  $W$-equivariant structure on the $\dd(T)$-module
$J(\int^0_p\cf)$ is not  immediately
visible  from the definition of the functor ${\mathbf J}\inv\ccirc\int^0_p$.
 \erem

The proof of the following result is left for the reader.
\begin{lem}\label{SSM} View $\BM$ as a $\dd(T)\otimes \hpp^{op}$-module.
Then, we have $\op{SS}(\BM)=\Lambda$, the Lagrangian subvariety from Proposition \ref{lagr}. In particular,
$\BM$ is a holonomic  $\dd(T)\otimes \hpp^{op}$-module.
\end{lem}

\section{Nil-Hecke algebras}\label{Nil-sec}
\subsection{Degenerate nil-Hecke algebras} \label{nil-sec}
Let $\Ups\sset\cR\sset\h^*$ be a reduced root system, where $\Ups$ denotes the set of simple roots. Let
$\W$ be the corresponding Coxeter group,  $\{s_\al,\ \al\in\Ups\}$ the set of generators of $\W$,
and $\ell: \W\to \Z_{\geq 0}$ the length function.
For $\al,\be\in\Ups$, let $m_{\al,\be}$ denote the order of the element
$s_\al s_\be\in\W$. The 
{\em nil-Hecke algebra} $\nil(\W)$ is defined as a $\C$-algebra with generators
$\bt_\al,\ \al\in\Ups$, subject to the relations
\beq{braid}
(\bt_\al)^2=0,\quad (\bt_\al\bt_\be)^{m_{\al,\be}}=(\bt_\be\bt_\al)^{m_{\al,\be}},\quad \forall\al,\be\in\Ups.
\eeq

For each $w\in \W$, there is an element $\bt_w\in\W$  defined as a product $\bt_{\al_{i_1}}\cdots \bt_{\al_{i_k}}$ for a reduced factorization
$w=s_{\al_{i_1}}\cdots s_{\al_{i_k}}$ into simple reflections. It is known that $\bt_w$ is independent of such a factorization
and one has
$\bt_w \bt_y=\bt_{wy}$ if $\ell(w)+\ell(y)=\ell(wy)$ and $\bt_w \bt_y=0$, otherwise, see \cite{Ku}.
Moreover, the set $\{\bt_w,\ w\in \W\}$ is a $\C$-basis of $\nil(\W)$.

Let $\C(\h^*)$ be the field of fractions of the algebra $\sym\h=\C[\h^*]$.
For every  $\al\in R$, define an element $\th_\al=\frac{1}{{\check\al}}(s_\al-1)\in W\ltimes \C(\h^*)$.
There is a natural algebra map $\nil(\W)\into\C(\h^*)\rtimes \W$
given on the generators by $\bt_\al\mto \th_\al,\ \al\in\Ups$.
Let $\nil(\h,\W)$ be a free left $\sym\h$-submodule of $\C(\h^*)\rtimes \W$ with basis
$\th_w,\ w\in\W$. It is immediate to check that     $\nil(\h,\W)$ is a subalgebra of $\C(\h^*)\rtimes \W$
and that
 $\nil(\h,\W)$ is free as a  $\sym\h$-module via right multiplication, \cite{Ku}.

\begin{prop}[\cite{Ku}, Theorem 11.1.2] The algebra  $\nil(\h,\W)$ is generated by the algebras $\nil(\W)$ and $\sym\h$ subject to
the following commutation relations:
\beq{def-rel}
\th_\al\cdot s_\al( h)- h\cdot\th_\al=\langle\al, h\rangle,\qquad \forall\  h\in\h,\ \al\in\Ups.
\eeq
\end{prop}

The subspace $\sym\h\sset\C(\h^*)$ is stable under the action of the subalgebra $\nil(\h,\W)\sset \C(\h^*)\rtimes \W$
on  $\C(\h^*)$. Conversely, one has
\begin{thm}[\cite{Ku},  Section   11.2] \label{ku} The $\nil(\h,\W)$-action on $\sym\h$ is faithful, and we have 
\[\nil(\h,\W)=\{u\in \C(\h^*)\rtimes \W\mid u(\sym\h)\subseteq\sym\h\}.
\]
\end{thm}

We equip  $\sym\h$ with  the natural grading and extend it to 
 a grading on  $\nil(\h,\W)$ by placing $\th_w$ in degree $-\ell(w)$. This makes $\nil(\h,\W)$ a $\Z$-graded
algebra, resp. $\sym\h$ a $\Z_{\geq0}$-graded $\nil(\h,\W)$-module.

The assignment  $s_\al\mto \check\al\cdot\th_\al+1,\ \al\in\Ups$, extends to an algebra embedding
$\C\W\into \nil(\h,\W)$. We will identify $\C\W$ with its image, which is contained in the degree zero 
homogeneous component of $\nil(\h,\W)$.
It is clear that for any $w\in \W$ and $\al\in \cR$, inside $\W\ltimes \C(\h^*)$,  one has 
$\th_{w(\al)}=w\cdot\th_\al\cdot w\inv$. Taking $\al$ to be a simple root, we see that
the element $\th_\be$ belongs to  $\nil(\h,\W)$ for any root $\be$.

Given an $\sym\h\rtimes \W$-module $M$ and $\al\in\Ups$, let
$M^\pm=\{m\in M\mid s_\al(m)=\pm m\}$. It is clear that the action of $\check\al\in \h\sset\sym\h$
on $M$ sends $M^\pm$ to $M^\mp$.
Observe also that
the action of $\nil(\h,\W)$ on 
any  $\nil(\h,\W)$-module  is $(\sym\h)^\W$-linear
since the center of  $\nil(\h,\W)$ equals $(\sym\h)^\W$.

\begin{lem}\label{morita} Assume that the group  $\W$ is finite. Then
the $\nil(\h,\W)$-action  on $\sym\t$ yields an algebra isomorphism
\beq{gkv} \nil(\h,\W)\cong \End_{_{(\sym\h)^\W}}\,\sym\h.
\eeq

Furthermore, for a $\sym\h\rtimes \W$-module $M$, the following are equivalent:

\vi 
The natural map $\sym\h\,\o_{_{(\sym\h)^\W}}\, M^\W\to M$ is an isomorphism.

\vii The action of  $\W\ltimes\sym\h$ on $M$ admits a (necessarily unique) extension to an $\nil(\h,\W)$-action on ~$M$.

\viii For every $\al\in\Ups$, the action map $\check\al: M^+\to M^-$ is a bijection.
\footnote{The  equivalence of (ii) and (iii) was pointed out to me  by Gwyn Bellamy.}
\end{lem}
\begin{proof}
Recall that $\sym\h$ is a free $(\sym\h)^\W$-module of  rank $\#\W$.
 Isomorphism \eqref{gkv} is  a simple consequence of Theorem \ref{ku}, cf. also 
\cite[Proposition 2.3]{GKV}.
The equivalence of (i) and (ii) now follows from
Morita equivalence of the
algebras $(\sym\h)^\W$ and $\End_{_{(\sym\h)^\W}}\,\sym\h$.
The equivalence of (ii) and (iii)  easily follows from the formula $s_\al=\check\al\cdot\th_\al+1$.
\end{proof}

\subsection{Degenerate nil-DAHA}\label{daha-sec}
Let  $\Ups\sset R\sset\t^*$  be  a finite reduced root datum,
and   $\Ups\aff=\Ups\sqcup\{\al_0\}\sset R\aff\sset\t^*\aff$
an associated affine root datum. 
Let $\BX^*$ and $Q$ be  the weight and root lattice
of $R$, respectively. 
Let $\wt W=W\ltimes \BX^*\supset W\aff=W\ltimes Q$
be the extended affine Weyl group. Thus, $\wt W= (\BX^*/Q) \ltimes W\aff$.

Similar to the case of affine Hecke algebras,
we define $\nil(\t\aff,\wt W):=\BX^*\ltimes_Q \nil(\t\aff,W\aff)$.
To avoid confusion, we write elements of
$\BX^*$, viewed as a subgroup of $\wt W$, in the form $e^\mu$ and also use the same symbol
for the image of that element 
under the algebra embedding $\C\wt W\into \nil(\t\aff,\wt W)$. 
Thus, for any root $\al\in R$ in $\nil(\t\aff,\wt W)$,
there is an element $\th_{e^\al}$ ($=\th_w$ for $w=e^\al$), and also a different element,
$\th_\al$, the  Demazure element
associated with $\al$   viewed as a (not necessarily simple)  root in $R\aff$.

Let $\hb\in \t\aff$ be the minimal imaginary coroot. We identify $\t\aff$ with $\t\oplus \C\hb$,
resp. $\sym(\t\aff)$ with $\C[\t^*][\hb]$.
The natural linear action of $\WW$  on $\t\aff$ induces
a $\WW$-action on $\C[\t^*][\hb]$, by algebra automorphisms. Explicitly,
for $\mu\in\BX^*$, the action of  $e^\mu$ on $\C[\t^*][\hb]$
is given by the formula
$(e^\mu f)(x, \hb)=f(x-\hb\mu,\hb)$. 
The action of the algebra $\nil(\t\aff,W\aff)$ on $\C[\t^*][\hb]$
extends to a faithful $\nil(\t\aff,\WW)$-action which agrees with the
$\WW$-action,
where  $\WW$ is viewed as a subset of
 $\nil(\t\aff,\WW)$ via the canonical embedding. 
From the above formula for the action of $e^\mu$, one finds
the following  commutation relations in $\nil(\t\aff,\wt W)$:
\beq{ddh}
\xi\cdot e^\mu=e^\mu\cdot (\xi +\langle\mu,\xi\rangle\cdot \hb),\qquad \mu\in \BX^*,\ \xi\in\t.
\eeq

Let $\dh(T)$ be the Rees algebra of the algebra $\dd(T)$ of differential operators, equipped with the
filtration by  order of the differential operator. Recall the notation
$t^\mu\in\C[T]$ for the function on $T$ associated with $\mu\in\BX^*$.
For $\xi\in\t$, let $\partial_\xi$ denote  the corresponding translation invariant vector field on $T$.
We may view $t^\mu$ and $\partial_\xi$ as elements of $\dh(T)$ placed in degrees $0$ and $1$,
respectively.
The commutation relations \eqref{ddh} are identical to the 
commutation relations  in the algebra $\dh(T)$ between the elements $t^\mu$ and $\partial_\xi$.
Further, it follows from  \eqref{ddh} that the set $\op{S}:=(\sym\t\aff)\sminus\{0\}$ is an Ore subset of the algebra
$\dh(T)$. The  corresponding noncommutative localization $\op{S}\inv\cdot\dh(T)$
may be viewed as
a kind of  microlocalization of $\dh(T)$.
We obtain  algebra embeddings
\beq{ncloc}
\xymatrix{
W\ltimes \dh(T)\ \ar@{^{(}->}[rrrr]^<>(0.5){w\mto w,\ t^\mu\mto e^\mu,\  \xi\mto \partial_\xi}&&&&
\ \HH\ \ar@{^{(}->}[r]&\ \WW\ltimes \Q(\t\aff)=W\ltimes \op{S}\inv\cdot\dh(T).
}
\eeq

Given a $\C[\hb]$-algebra $A$ and $c\in\C$, we let $A|_{\hb=c}:=A/(\hb
-c)A$ denote its specialization at $\hb=c$. We have
$\dh(T)|_{\hb=1}=\dd(T)$,
resp. $\dh(T)|_{\hb=0}=\C[\T^*T]$.
We define  $\HH=\nil(\t\aff,\wt W)|_{\hb=1}$.
The  graded algebra embedding $W\ltimes \dh(T)\into\HH$ induces
an embedding  $W\ltimes \dd(T)\into\HH$, resp. $W\ltimes \C[\T^*T]\into
\gr\HH$,
of specializations at $\hb=1$, resp. $\hb=0$.

Recall from the introduction, the notation $\pe=\frac{1}{\#W}\sum_{w\in
  W}\, w\in \C W$ and
the  spherical  algebra  $\pe\nil(\t\aff,\wt W)\pe$, resp.
$\hsph:=\pe\HH\pe$. 
The space $\nil(\t\aff,\wt W)\pe$, resp. $\HH\pe$, has the natural structure
of an $(\nil(\t\aff,\wt W),\,\pe\nil(\t\aff,\wt W)\pe)$-bimodule, resp. $(\HH,\HS)$-bimodule.
The  embedding $W\ltimes \dh(T)\into\HH$ induces an
 embedding $\dd(T)^W\into\HS$, resp. $\C[\T^*T]^W\into \gr\HS$.

\begin{lem}\label{morita-prop}  Let $c\in \C$ and put $\nil_c=\nil(\t\aff,\wt W)|_{\hb=c}$. The algebras $\nil(\t\aff,\wt W)$ and $\pe\nil(\t\aff,\wt W)\pe$,
resp. $\nil_c$ and $\pe\nil_c\pe$, are Morita equivalent
and the action map   induces an  algebra isomorphism
\[\nil(\t\aff,\wt W)\ \iso\ \End_{\pe\nil(\t\aff,\wt W)\pe} \nil(\t\aff,\wt W)\pe,
\quad\text{\em resp.}\quad
\nil_c\ \iso\  \End_{\pe\nil_c\pe}\,\nil_c\pe.
\]
\end{lem}
\begin{proof} It is immediate from \eqref{gkv} that 
$\pe\nil(W,\t)\pe=(\sym\t)^W$ and, moreover, there exist elements $h'_i,h''_i\in \nil(\t,W)$ such that 
one has $1=\sum_i\ h'_i\cdot\pe\cdot h''_i$. 
Since $\nil(W,\t)$ is a subalgebra of $\nil(\t\aff,\wt W)$, this equation may be viewed as an equation
in $\nil(\t\aff,\wt W)$. Thus, we have  $\nil(\t\aff,\wt W)\pe\nil(\t\aff,\wt W)=\nil(\t\aff,\wt W)$, which   is known 
to imply  all statements of the lemma involving $\nil(\t\aff,\wt W)$.
The case of specializations at $\hb=c$ is similar.
\end{proof}

\begin{prop}\label{hh-morita} \vi   Let $M$ be a $W\ltimes
  \dh(T)$-module. 
The action of  $W\ltimes \dh(T)$ on $M$ can
 be extended (necessarily uniquely) to an $\nil(\t\aff,\WW)$-action on $M$ if and
 only if the natural map $\C[\t^*]\o_{\C[\t^*]^W}M^W\to M$ is an isomorphism. A
 similar statement holds in the case of specializations at $\hb=c$, for any $c\in\C$.

\vii Let  $L$ be a $\dd(T)^W$-module. The action of  $\dd(T)^W$ on $L$ can
 be extended (necessarily uniquely) to an $\HS$-action on $L$ if and only if the map $\sym\t\o_{(\sym\t)^W}L\to \dd(T)\o_{\dd(T)^W}L$,
induced by the inclusion $\sym\t\into\dd(T)$, is an isomorphism.
\end{prop}
\begin{proof} 
The simple reflection $s_{\al_0}:\t\aff\to$ is a reflection with respect to the hyperplane
$\delta(x)=\hb$, where  $\delta\in  R$ is  the highest root. The corresponding Demazure operator 
acts on $\C[\t^*][\hb]$ as follows:
\beq{al0} (\th_{\al_0}f)(x,\hb)= \mbox{$\frac{f(s_{\al_0}(x,\hb))-f(x,\hb)}{\langle x,\check\delta\rangle-\hb}$}=
\mbox{$\frac{f(x-(\langle x,\check\delta\rangle-\hb)\delta)-f(x,\hb)}{\langle x,\check\delta\rangle-\hb}$}.
\eeq

For any weight $\mu\in\BX^*$, we compute
\begin{align}
(e^\mu\ccirc \th_{\delta}\ccirc e^{-\mu} f)(x,\hb)&=(e^\mu\ccirc \th_{\delta}f)(x+\hb\mu,\hb)=
e^\mu \Big(\frac{f(x-\langle x,\check\delta\rangle\delta+\hb\mu,\,\hb)- f(x+\hb\mu,\hb)}{\langle x,\check\delta\rangle}\Big)
\nonumber\\
&=
\frac{f(x-\hb\mu-\langle x-\hb\mu,\check\delta\rangle\delta+\hb\mu,\,\hb)- f(x-\hb\mu+\hb\mu,\,\hb)}{\langle x-
\hb\mu,\check\delta\rangle}\label{th-0}\\
&=
\frac{f(x-(\langle x,\check\delta\rangle-\hb\langle \mu,\check\delta\rangle)\delta)- 
f(x,\,\hb)}{\langle x,\check\delta\rangle-\hb\langle \mu,\check\delta\rangle}.
\nonumber
\end{align}

Now let $M$ be as in (i). By Lemma \ref{morita}, the action of $W\ltimes \sym\t$ on $M$ can be extended
to an action of the algebra $\nil(\t,W)$. Next, we use the action
of the subalgebra $\C[T]\sset \dh(T)$ on $M$.
Specifically, since any root is $W$-conjugate to a simple root, there exists a weight $\mu\in\BX^*$ such that we have
$\langle \mu,\check\delta\rangle=1$.
For such a $\mu$, formula \eqref{th-0} shows that in $\nil(\t\aff,\WW)$,
one has $\th_{\al_0}=e^\mu\ccirc \th_{\delta}\ccirc e^{-\mu}$. Accordingly, we  let 
$\th_{\al_0}$ act on $M$ by  the operator $e^\mu\ccirc \th_{\delta}\ccirc e^{-\mu}$.\footnote{This correction of an original, incorrect,
construction of $\th_{\al_0}$ was suggested to me
by Gus Lonergan.} We claim that    the action of $\th_{\al_0}$
 thus defined and  the actions  of $\nil(\t,W)$ and $\C[T]$ on $M$  combine  together to 
give $M$ the structure of an $\nil(\t\aff,\WW)$-module. Thus, we must check that the relations
\eqref{braid} and  \eqref{def-rel} hold. The relations which do not involve $\th_{\al_0}$ hold by construction.
The equation $(\th_{\al_0})^2=0$ is clear. Also, using \eqref{ddh} and \eqref{th-0}
it is straightforward to check that   \eqref{def-rel} holds for $\al=\al_0$.

To prove the braid relations we use \cite{Lo2}. Specifically, let  ${\mathcal I}(\t\aff)$ be the algebra
associated, as   in  \cite[ Section   2.7]{Lo2}, with the  Coxeter group $\Gamma:=W\aff$ and  its reflection representation
$\h:=\t\aff$. 
The relations we have checked so far say that
$M$  has the structure of an ${\mathcal I}(\t\aff)$-module. The
algebra $\nil(\t\aff,W\aff)$ is a quotient of  ${\mathcal I}(\t\aff)$ by a two-sided
ideal. Furthermore,
the main result of   \cite{Lo2} is essentially  equivalent,  see \cite[ Section   2.7]{Lo2},
 to the statement that the canonical  map ${\mathcal I}(\h)\onto \nil(\h,\Gamma)$ is, in fact, an isomorphism,
which is what we want. We remark that in   \cite{Lo2}
the Coxeter  group  is assumed  to be finite. The braid relation
$(\th_\al\th_\be)^{m_{\al,\be}}=(\th_\be\th_\al)^{m_{\al,\be}}$ was proved in \cite[ Section  3]{Lo2} under that 
assumption. However, in the case of an infinite Coxeter group
with two generators,  $\bt_\al,\bt_\be$, the element $\bt_\al\bt_\be$ has an infinite order,
i.e., one has  $m_{\al,\be}=\infty$. Therefore, the corresponding braid relation  is, in that case, vacuous. This proves that
the map ${\mathcal I}(\h)\onto \nil(\h,\Gamma)$ is  an isomorphism
for any, not necessarily finite, Coxeter group $\Gamma$.

Thus, we have extended the action of $W\ltimes\dh(T)$ on $M$ to an $\nil(\t\aff\WW)$-action.
Note that if   such an extension exists,  then it is unique since each of the operators
$\th_\al,\ \al\in\Ups\aff$, is defined uniquely due to the Morita equivalence \eqref{gkv} in the special
case of the rank-one nil Hecke algebra associated with the root system $\{\al,-\al\}$.
Part  (i) of Proposition \ref{hh-morita} follows.

Combining part (i) for the specialization at $\hb=1$  
with Morita equivalence of $W\ltimes \dd(T)$ and $\dd(T)^W$ yields part (ii).
\end{proof}

By definition, one has $\nil(\t\aff,\wt W)|_{\hb=0}=\gr\HH$, resp. $
\pe\nil(\t\aff,\wt W)\pe|_{\hb=0}=\gr\HS$.
Note that the last statement of the  proposition below implies  Proposition \ref{bfm-loc}.

\begin{prop}\label{hh-class} 
There is an algebra isomorphism
$$\gr \HH\ccong W\ltimes\C[T\times \t^*, \frac{t^\al-1}{\calph},\ \al\in R], 
\quad{\textrm{resp.}}\quad\gr\HS\ccong\C[T\times \t^*, \frac{t^\al-1}{\calph},\ \al\in R]^W.$$
Furthermore, we have $(\gr\HH)\pe\cong \C[T\times \t^*, \frac{t^\al-1}{\calph},\, \al\in R]$, and each of the above objects  is flat over ~$\C[\t^*]^W$.
\end{prop}
\begin{proof}
It is clear from \eqref{ddh} that the elements $e^\mu\in \nil(\t\aff,\wt W)\ho,\ \mu\in \BX^*$, and $\xi\in\t$, generate 
a copy of the algebra $\C[\T^*T]$ inside $\nil(\t\aff,\wt W)\ho$.
For $w\in \WW$, let $\bar\th_w$ denote the image of the element $\th_w\in \nil(\t\aff,\wt W)$ in $\nil(\t\aff,\wt W)\ho$.
Then, the commutation relations \eqref{def-rel} imply easily that the assignment
$$\bar\th_\al\ \mto\ \mbox{$\frac{1}{\calph}$}(s_\al-1),\quad  \bar\th_{\al_0}\
\mto\ \mbox{$\frac{1}{\check\delta}$}(s_\delta-1),\quad\al\in \Ups,$$
extends by multiplicativity to a well-defined algebra homomorphism
\beq{Phi} \Phi:\ \gr\HH=\nil(\t\aff,\wt W)|_{\hb=0}\too W\ltimes\C[T\times \t^*, \mbox{$\frac{t^\al-1}{\calph}$},\ \al\in R].
\eeq
 
For any $\mu\in\BX^*$,
using that
$t^\mu  s_\al t^{-\mu}s_\al =t^{\langle \mu,\calph\rangle\al}$, we find that
\beq{Phi2}t^\mu  \mbox{$\frac{1}{\calph}$}(s_\al-1)t^{-\mu}s_\al +\mbox{$\frac{1}{\calph}$}(s_\al-1)=\mbox{$\frac{1}{\calph}$}(t^{\langle \mu,\calph\rangle\al}-s_\al)
+\mbox{$\frac{1}{\calph}$}(s_\al-1)=
\mbox{$\frac{t^{\langle \mu,\calph\rangle\al}-1}{\calph}$}.
\eeq
Letting $\mu$ run through the set of fundamental weights, we see that the image of the map $\Phi$
contains all elements $\frac{t^\al-1}{\calph},\ \al\in\Sigma$.
The  algebra on the right of \eqref{Phi} is  generated by these elements
and its subalgebra $W\ltimes\C[T\times \t^*]$.
It follows that the map $\Phi$  is surjective.
Also, it is clear that applying $\C(\t^*)\o_{\C[\t^*]}(-)$, a localization functor,
to the map $\Phi$ yields an isomorphism between the corresponding localized $\C(\t^*)$-modules.
Hence, $\Ker \Phi$ is a torsion $\sym\t$-module.

Further, the algebra $\nil(\t\aff,\wt W)$ is a free $\C[\t\aff^*]$-module with basis $\th_w,\ w\in \WW$, by definition.
It follows that $\nil(\t\aff,\wt W)\ho$ is  a free $\C[\t^*]$-module with basis $\bar\th_w,\ w\in \WW$,
in particular, this module is torsion free. Therefore, we have $\Ker \Phi=0$, so the map $\Phi$ is an isomorphism.
Since $\C[\t^*]$ is  free over $\C[\t^*]^W$ and $\nil(\t\aff,\wt W)\ho$ is free over $\C[\t^*]$, we deduce
that $W\ltimes\C[T\times \t^*, \frac{t^\al-1}{\calph},\ \al\in R]$ is a free $\C[\t^*]^W$-module.

Finally, the map $\Phi$ being an isomorphism, it follows that its restriction to the spherical subalgebra
yields an isomorphism $\gr\HS\ho\iso \C[T\times \t^*, \frac{t^\al-1}{\calph},\ \al\in R]^W$.
Viewed as a $\C[\t^*]^W$-module,  the algebra  $\gr\HS$ is a direct summand
of $\gr\HH$, hence it is flat over  $\C[\t^*]^W$. This implies that $\C[T\times \t^*, \frac{t^\al-1}{\calph},\ \al\in R]^W$
is a flat  $\C[\t^*]^W$-module, completing the proof.
\end{proof}

  \section{Spherical degenerate nil DAHA via Hamiltonian reduction}\label{pf-sec}

\subsection{The action of $\HH$ on the Miura bimodule $\BM$} The following result provides the link between nil Hecke algebras and Hamiltonian reduction.

\begin{prop}\label{act} The left action of $W\ltimes\dd(T)$ on the Miura bimodule $\BM$
 can be extended to
an $\HH$-action,  making $\BM$ an $(\HH,\hpp)$-bimodule.
\end{prop}

\begin{proof} This follows from Proposition \ref{hh-morita} since $\BM^W=\hpp$, so the map
$\sym\t\,\o_{_{(\sym\t)^W}}\BM^W\to \BM=\sym\t\,\o_{_{(\sym\t)^W}}\,\hpp$
is an isomorphism. 
\end{proof}

Recall that the algebra $\hpp$ comes equipped with the Kazhdan filtration.
We equip  $\sym\t$ with the natural filtration induced by the grading and
equip $\BM=\sym\t\,\o_{_{\sh^W}}\,\hpp$ with a tensor product filtration.
Let, $\BM_\hb$ be an associated Rees module.
It is straightforward to see from the construction that the $(\HH,\hpp)$-bimodule
structure on $\BM$ is compatible with the filtrations. Therefore, 
the left $\HH$-action on the Miura bimodule $\BM$ can be lifted
to an $\HH_\hb$-action on $\BM_\hb$. Thus, $\BM_\hb$ acquires
the structure of a $\Z$-graded $(\HH_\hb,\hpp_\hb)$-bimodule.
We will identify $\BM_\hb$ with $(\sym\t)[\hb]\,\o_{_{(\sym\t)^W[\hb]}}\,\hpp_\hb$,
resp. $\BM_\hb^W$ with $\hpp_\hb$.
Let $1_\BM=1\o 1\in (\sym\t)[\hb]\,\o_{_{(\sym\t)^W[\hb]}}\,\hpp_\hb$ be the generator.
Recall that $\HH_\hb=\nil(\t\aff,\WW)$.

The following theorem and its corollary are an extended version of Theorem \ref{main-intro} from the introduction.

\begin{thm}\label{main} \vi The map $\nil(\t\aff,\WW)\to\BM_\hb,\ u\mto u(1\o1)$, induces an isomorphism
$\nil(\t\aff,\WW)\pe\iso\BM_\hb$, of  graded left  $\nil(\t\aff,\WW)$-modules.

\vii The map $\pe\nil(\t\aff,\WW)\pe\to \BM_\hb^W=\hpp_\hb,\ \pe u\pe\mto \pe u(1\o1)$, yields a graded  algebra isomorphism
\[\pe\nil(\t\aff,\WW)\pe\ \iso\  \End_{\hpp ^{op}_\hb}\hpp_\hb=\hpp _\hb,\]  such that
the following diagram commutes:
$$
\xymatrix{
(\sym\t)^W[\hb]\ar[d]\ar@{=}[r]&Z_\hb\g\ar[d]\\
\pe\nil(\t\aff,\WW)\pe\ar[r]&\hpp_\hb.}
$$
\end{thm}

From this theorem, combined with  Lemma \ref{morita-prop}, we deduce 
\begin{cor} The action of $\HH$ on $\BM$ yields the following isomorphisms of filtered algebras:
\[\HH\iso \End_{_{\hpp ^{op}}}\, \BM,\quad\text{\em resp.}\quad
\HS\iso \hpp.\]
\end{cor}

\subsection{Proof of Theorem \ref{main}} \label{m-pf}
To simplify the notation, we put $\nil=\nil(\t\aff,\WW)=\HH_\hb$,
resp. $\nil^{\op{sph}}=\pe\nil(\t\aff,\WW)\pe=\HS_\hb$.
Also, write $\sh$ for $(\sym\t)[\hb]=\C[\t\aff^*]$, resp.
$\C(\t\aff^*)$ for the field of fractions of ~$\sh$.

The natural $\nil$-action on $\sh$ makes $\sh$ 
a cyclic  left $\nil$-module with generator $1_{\sh}$. 
The annihilator of $1_{\sh}$ is   a left ideal $J\sset \nil$  generated by the
elements $\th_w,\ w\in\wt W$. Since $\pe1_{\sh}=1_{\sh}$, the map $\nil\to \sh,\ h\mto h1_{\sh}$, descends to
a surjection $\nil\pe\onto\sh$ with kernel $J\pe$. 
This gives an $\nil^{\op{sph}}$-module surjection $\nil^{\op{sph}}=\pe\nil\pe\onto \sh^W=\pe\sh$, with kernel
$\js:=\pe J\pe$, a left ideal of $\nil^{\op{sph}}$.

The action of $\nil$ on $\BM_\hb$ gives a map 
$\tau: \pe\nil\pe\to\hpp_\hb=\BM^W_\hb,\ h\mto h 1_{\hpp_\hb}$.
The image of $\js \sset\nil^{\op{sph}}$ under the map $\tau$
is a left $\nil^{\op{sph}}$-submodule $\tau(\js)\sset \hph$. This submodule is not,
a priori, stable under left multiplication by $\hph$. We
let $K:=\hpp_\hb\cdot\tau(\js)$ be a
 left ideal of the algebra $\hph$ generated by $\tau(\js)$.
Let $\ks=\cap_{i\in \Z}\ (K+\hb^i\hph)$ be the `closure' of $K\sset \hph$
in  $\hb$-adic topology. 

\begin{lem}\label{hn} One has a  decomposition
$\hph=Z_\hb\g\oplus \ks$ as a direct sum of $Z_\hb\g$-stable graded subspaces. 
\end{lem}
\begin{proof} 
We first prove an analogue of the statement of the lemma at $\hb=0$. To this end, let
 $\nil_0:=\nil|_{\hb=0}$, resp. $\nil_0^{\op{sph}}=\pe\nil_0\pe$.
We have a chain of 
isomorphisms
\beq{chain1}
(\gr\HH)\pe=\nil_0\pe\ccong \C[T\times \t^*, \mbox{$\frac{t^\calph-1}{\calph}$},\ \al\in R]\ccong
\C[\fZ\times_\fc\t^*]\ccong \sym\t\o_{(\sym\t)^W}\gr\hpp,
\eeq
where the first, resp. second, and third, isomorphism is Proposition \ref{hh-class},
resp. Theorem \ref{bfm}  and Theorem \ref{class-intro}.

Let $J_0\sset \nil_0\pe$ be the image of $J$ in $\nil_0=\nil/\hb\nil$
and let $I\sset \C[T\times \t^*, \frac{t^\calph-1}{\calph},\ \al\in R]$, resp. $L\sset \sym\t\o_{(\sym\t)^W}\gr\hpp$,
be an ideal that corresponds to $J_0$ via the first, resp. the composite,
 isomorphism  in \eqref{chain1}. The proof of Proposition
\ref{hh-class}, combined with  \eqref{Phi}-\eqref{Phi2},
shows that  the ideal $I$ is generated
by the elements $\frac{t^\mu-1}{\cmu},\ \mu\in\BX^*$. 
It is clear from this description, cf. also \eqref{Phi2}, that one has 
 a direct sum decomposition
 $\C[T\times \t^*, \frac{t^\al-1}{\calph},\ \al\in R]=
\C[\t^*]\oplus I$, of graded $\sym\t$-stable subspaces. Hence, one has a decomposition
$\sym\t\o_{(\sym\t)^W}\gr\hpp=\sym\t\oplus L$. 
Taking $W$-invariants and writing   $\js_0= J^W_0$, we deduce the following isomorphism that respects the direct sum decompositions:
\beq{chain2}
\nil^{\op{sph}}_0=(\sym\t)^W\oplus \js_0\ \ccong\  \gr\hpp=(\sym\t)^W\oplus  L^W.
\eeq

It is straightforward to check that the isomorphism $\nil^{\op{sph}}_0\iso\gr\hpp$, in \eqref{chain2},
agrees with $\tau_0: \nil^{\op{sph}}_0\iso\gr\hpp$, the specialization  of the map $\tau: \nil^{\op{sph}} \to \hpp_\hb$ at $\hb=0$.
The image of   the ideal $\js_0$  equals $\tau_0(\js_0)$, by definition.
We conclude that $L^W=\tau_0(\js_0)$. Thus, the decomposition on the right of \eqref{chain2} reads
 \beq{zzz}\hpp_\hb/\hb\hpp_\hb=(Z_\hb\g/\hb\cd Z_\hb\g)\ \oplus\ \tau_0(\js_0).
\eeq

This   is equivalent
to a pair of equations
\beq{aaa1}
\hpp_\hb= Z_\hb\g+\tau(\js)+\hb\hpp_\hb,
\qquad  Z_\hb\g\ \cap\ (\tau(\js)+\hb\hpp_\hb)\ =\ \hb\hpp_\hb.
\eeq

\begin{rem} The   decomposition in  \eqref{zzz} does not necessarily imply that $\hpp_\hb=Z_\hb\g\oplus \tau(\js)$,
 since  the $\Z$-grading on $\hpp_\hb$ is
 not bounded below. 
\erem

It is immediate from definitions that
$\ks$ is a graded left ideal of $\hpp_\hb$. Furthermore,  the ideals $\ks$ and $K$ have the same image 
 in $\gr\hpp=\hpp_\hb/\hb\hpp_\hb$. Also, the image of $K$ in  $\gr\hpp$ equals  $\gr\hpp\cdot\tau_0(\js_0)$,
by definition.
We have shown that  $L^W=\tau_0(\js_0)$, where  $L^W$  is an ideal of  $\gr\hpp$. 
Therefore,
$\tau_0(\js_0)$ is  an ideal. Hence, inside $\gr\hpp$, we have
 $\js_01_{\gr\hpp}=\gr\hpp\cdot\tau_0(\js_0)$. Combining these observations together, we deduce
\beq{aaa}
\tau(\js)+\hb\hpp_\hb=K+\hb\hpp_\hb=\ks+\hb\hpp_\hb.
\eeq
Using \eqref{aaa}, we see that  equations \eqref{aaa1} take the following form:
\beq{aaa2}\hpp_\hb= Z_\hb\g+\ks+\hb\hpp_\hb,
\qquad  Z_\hb\g\ \cap\ (\ks+\hb\hpp_\hb)\ =\ \hb\hpp_\hb.
\eeq

We claim that the equations in \eqref{aaa2} imply the required direct sum decomposition
$\hpp_\hb= {Z_\hb\g\oplus \ks}$. Indeed, let $a\in \hpp_\hb$. Separating homogeneous components, one may
assume without loss of generality that $a$ is homogeneous of some  degree $n\in\Z$. Note that $ Z_\hb\g$
has no  homogeneous components of negative degree. Hence, in the case $n<0$, the first equation
in \eqref{aaa2} implies that there exist $k_1\in \ks$ and $a_1\in \hpp_\hb$, of degrees $n$ and $n-1$, respectively,
such that $a=k_1+\hb a_1$. Applying the same argument to $a_1$, one gets $a_1=k_2+\hb a_2$, and so on.
Thus, for any $m\geq1$, one can find $k_1,\ldots, k_m\in K$ such that one has
$a=\hb^{m+1}a_{m+1}+\sum_{i=1}^m \hb^{i-1} k_i$.
Since $K$ is an ideal, we have $\sum_{i=1}^m \hb^{i-1} k_i\in K$ so
$a\in \hb^{m+1}\hpp_\hb+K$. We conclude that $a\in \cap_{m\in\Z}\ (\hb^{m+1}\hpp_\hb+K)=\ks$.
In the case $n\geq0$, one shows by induction on $n$ using a similar argument  that $a\in  Z_\hb\g+\ks$. 
Finally, from  the second equation in \eqref{aaa2}, one deduces that $ Z_\hb\g\cap \ks=0$.
\end{proof}

\begin{proof}[Proof of Theorem \ref{main}] 
We know that $\nil^{\op{sph}}$ and $\nil\pe$ are torsion free
$\sh^W$-modules. Therefore, the map
$\nil^{\op{sph}}\to \C(\t\aff^*)^W\o_{\sh^W}\nil^{\op{sph}}$
 is injective. 
Further, the algebra $\fZ$ is a flat scheme over $\fc$, see  Section  \ref{1st-sec}.
Hence, the algebra $\gr\hpp\cong\C[\fZ]$ is a  torsion free
$\sh^W$-module.
It follows, since the filtration on $\hpp$ is separating by Lemma \ref{noeth},
that  the algebra $\hpp_\hb$ is  torsion free over $\sh^W=Z_\hb\g$.
We conclude that the map
$\hph\to\C(\t\aff^*)^W\o_{\sh^W}\hph$ is injective. It follows 
 that the map $\tau: \nil^{\op{sph}}\to \hph$ is injective. 

To prove that $\tau$ is surjective, recall from Section  \ref{daha-sec}   that $\nil$ is a subalgebra of $\wt W\aff\ltimes \C(\t\aff^*)$.
The latter algebra acts faithfully on $\C(\t\aff^*)$.
We have a decomposition $\C(\t\aff^*)=\pe\C(\t\aff^*)\oplus (1-\pe)\C(\t\aff^*)$.
The action of the  subalgebra $\pe(W\ltimes \C(\t\aff^*))\pe$  respects this decomposition,
moreover, the action of this subalgebra kills the second direct summand.
Further, we have $\nil^{\op{sph}}\sset \pe(W\ltimes \C(\t\aff^*))\pe$, so the action of $\nil^{\op{sph}}$ kills
$(1-\pe)\C(\t\aff^*)$ and $\sh^W=\pe\sh\sset \pe\C(\t\aff^*)$ is an $\nil^{\op{sph}}$-stable subspace.

Next, we use the map $\tau$ to pullback the  $\nil^{\op{sph}}$-action to a $\hph$-action on  $\pe\C(\t\aff^*)=\C(\t\aff^*)^W$.
Thanks to the injectivity of $\tau$, the resulting  $\hph$-module is faithful.
We claim that the subspace $\sh^W\sset \C(\t\aff^*)^W$ is stable under the $\hph$-action.
To see this, we use Lemma \ref{hn}. Specifically, it follows from
the lemma  that the map $\tau$ induces a bijection
$\nil^{\op{sph}}/\js\cong\sh^W\ \iso\ Z_\hb\g\cong\hpp_\hb/\ks$. Via this bijection, the 
action of the algebra $\nil^{\op{sph}}$ on $\sh^W$ can be extended to a $\hph$-action on $\sh^W\cong\hpp_\hb/\ks$. 
Is is immediate from the construction that the action so defined agrees with the above defined  $\hph$-action on 
$\C(\t\aff^*)^W$. The claim follows.

We now apply Theorem \ref{ku} to conclude that inside $\End_\C \C(\t\aff^*)$,
one has an inclusion $\hph\subseteq\tau(\nil^{\op{sph}})$.  Thus $\tau$ is surjective, hence, it is an isomorphism.
All other statements of the theorem easily follow from this.
\end{proof}
\subsection{Proof of Theorem \ref{cat-eq}}\label{pf-cat}
It follows from Proposition \ref{a-lem} and  the isomorphism $(\dd(G)\dss\n_l^\psi)\dss\n_r^\psi$ of Lemma \ref{all},
that the functor $M\mto M^{\n_l^\psi\times\times\n_r^\psi}$
gives an equivalence $(\dd_G, N_l\times N_r,\psi\times\psi)\mmod\cong\hpp\mmod$. 
The equation in Lemma \ref{afilt}(3) shows that this functor induces, by restriction to holonomic 
modules, an equivalence $\Wh\cong\hpp\hol$.
Also, by Theorem \ref{main-intro}, we have an equivalence
$\hpp\hol \cong \HS\hol$. The corresponding
equivalence $\hpp\hol \cong \HH\hol$ follows from this by Morita equivalence,
see Lemma \ref{morita-prop}, since the latter  equivalence
takes holonomic modules to holonomic modules, by definition.

Let ${\scr{C}}$ be the category of not necessarily holonomic 
$W\ltimes \dd(T)$-modules $M$ such that \eqref{hhdd} is an isomorphism. 
Further, let ${\scr{C}}^{\op{sph}}$ be the category of not necessarily holonomic 
$\dd(T)^W$-modules $L$
such that the map $\sym\t\,\o_{(\sym\t)^W}\,L\to \dd(T)\o_{\dd(T)^W} L$ is an isomorphism. 
It follows from  Morita equivalence of the algebras $\dd(T)^W$ and  $W\ltimes \dd(T)$
 that the functor $\dd(T)\o_{\dd(T)^W}(-)$ provides an equivalence
${\scr{C}}^{\op{sph}}\iso {\scr{C}}$.

Further, we have a functor $\HH\mmod\to{\scr{C}}$
 induced by the algebra embedding
$W\ltimes \dd(T)\into\HH$.
The proof of Proposition \ref{hh-morita}  shows that  this functor is  fully faithful.
Thus, this functor is an equivalence.  Using Morita equivalences above, we deduce
that the  functor  
$\HS\mmod\to {\scr{C}}^{\op{sph}}$, induced by the algebra embedding $\dd(T)^W\into\HS$,
is also  an equivalence. 

To complete the proof of the
theorem we must show that an object  $L\in {\scr{C}}^{\op{sph}}$ is holonomic  as
a $\dd(T)^W$-module if and only if it is holonomic as
an $\HS$-module. 
Assume first that  $L$ is holonomic as an $\HS$-module, equivalently, as a  $\hpp$-module.
Then, Proposition \ref{M-exact} implies that $\BM\o_\hpp L=\sym\t\o_{(\sym\t)^W} L$ is a
holonomic  $\dd(T)$-module. Hence, $L$ is holonomic as a $\dd(T)^W$-module.

Conversely,  assume that  $L$ is holonomic  as a $\dd(T)^W$-module $L$.
Choose a good filtration on $L$,  and
equip $M:=\sym\t\o_{(\sym\t)^W} L$, resp. $\HH\pe\o_{\dd(T)^W} L$ and   $\hpp\o_{\dd(T)^W} L$, with the tensor product filtrations.
The map $\T^*T\to(\T^*T)/W$ being finite, we deduce that $\supp \gr M$ is a Lagrangian subvariety of $\T^*T$. 
We have the following 
diagram, cf. Proposition \ref{lagr}:
\[
\xymatrix{
\fZ=\Spec(\gr\hpp)\ =\ \Spec(\gr\HS)\
&\ \Lambda\ \ar[r]^<>(0.5){p_{\T^*T}}
\ar[l]_<>(0.5){p_\fZ}&
\ \T^*T=\Spec(\gr\dd(T)).
}
\]
where $p_\fZ$, resp. $p_{\T^*T}$, is the first, resp. second,
projection and we have used Theorem \ref{main-intro} to identify $\gr\hpp$ with
$\gr\HS$, resp. Theorem \ref{class-intro} to identify $\Spec\gr\hpp$
with $\fZ$.
Thus, 
we have
\[\supp\gr(\hpp\o_{\dd(T)^W} L)\ =\ \supp(\gr\hpp\,\o_{\gr\dd(T)^W}\,\gr L)\ \subseteq\ p_\fZ(p_{\T^*T}\inv(\supp\gr M)).\]
By  Proposition \ref{lagr},  $\Lambda$ is a Lagrangian subvariety of $\fZ\times\T^*T$.
It follows
that $p_\fZ(p_{\T^*T}\inv(\supp\gr M))$ is an isotropic subvariety of $\fZ$. 
We deduce that $\hpp\o_{\dd(T)^W} L$ is a holonomic $\hpp$-module.
On the other hand, the action of the algebra   $\HS\cong\hpp$ on $L$ gives a canonical surjection
$\hpp\o_{\dd(T)^W} L\onto L$, of $\HS$-modules. This implies that
$L$ is holonomic as an $\hpp$-module.
\qed

\bibliographystyle{plain}

\begin{thebibliography}{LLMSSZ}
\bibitem[Ba1]{Ba1} A. Balibanu, {\em
The Peterson variety and the wonderful compactification.} Represent. Theory 21 (2017), 132--150.
\bibitem[Ba2]{Ba2} \bysame, {\em 
    The partial compactification of the universal centralizer.}
arXiv:1710.06327.

\bibitem[BK]{BK}
A. Beilinson, D. Kazhdan, {\em Flat projective connections.} (1991), Unpublished manuscript.
\begin{verbatim}
http://www.math.stonybrook.edu/~kirillov/manuscripts.html
\end{verbatim} 

\bibitem[Be]{Be} D. Beraldo, {\em Loop group actions on categories and Whittaker invariants.}
arXiv:1310.5127.
\bibitem[BGG]{BGG} J. Bernstein, I. Gelfand, S. Gelfand, {\em Schubert cells, and the cohomology of the spaces G/P.}
(Russian) Uspehi Mat. Nauk 28 (1973), no. 3 (171), 3 -- 26.


\bibitem[BBP]{BBP} R.  Bezrukavnikov, A. Braverman, L.
Positselskii, {\em Gluing of abelian categories and differential operators on the basic affine space.}
 J. Inst. Math. Jussieu 1 (2002),  543 -- 557.
 \bibitem[BBM]{BBM} \bysame, \bysame, I. Mirkovic, {\em Some results about geometric Whittaker model.} Adv. Math. 186 (2004),  143 -- 152. 
\bibitem[BF]{BF} \bysame, M. Finkelberg,
{\em Equivariant Satake category and Kostant-Whittaker reduction.}
 Mosc. Math. J. 8 (2008),  39–-72, 183.

 
\bibitem[BFM]{BFM}\bysame, \bysame, I. Mirkovic,
{\em  Equivariant homology and K-theory of affine Grassmannians and Toda lattices.} Compos. Math. 141 (2005),  746 -- 768.


\bibitem[BB]{BB} W. Borho, J.-L. Brylinski, {\em  Differential operators on homogeneous spaces. I.
 Irreducibility of the associated variety for annihilators of induced modules.} Invent. Math. 69 (1982),  437 -- 476.
\bibitem[CF]{CF} I. Cherednik, B. Feigin, {\em Rogers-Ramanujan type identities and Nil-DAHA.}
 Adv. Math. 248 (2013), 1050 -- 1088.

\bibitem[FT]{FT} M.  Finkelberg, A. Tsymbalyuk, {\em Multiplicative slices,
relativistic Toda and shifted affine quantum groups.} arXiv:1708.01795.

\bibitem[GG]{GG} W.L. Gan, V. Ginzburg, {\em
   Almost-commuting variety, $D$-modules, and Cherednik Algebras.}
 IMRP Int. Math. Res. Pap. 2006, 26439, 1–-54. 


\bibitem[Gi]{Gi} V. Ginzburg,  {\em   Harish-Chandra bimodules for quantized Slodowy slices.}
 Represent. Theory 13 (2009), 236–-271.
\bibitem[GKV]{GKV} \bysame, M. Kapranov, E. Vasserot, {\em Residue construction of Hecke algebras.} 
Adv. Math. 128 (1997),  1 -- 19.
\bibitem[GK]{GK} \bysame, D. Kazhdan, {\em
Differential operators on $G/U$ and the Miura bimodule.} In preparation.
\bibitem[GR]{GR} \bysame, S. Riche, {\em
Differential operators on $G/U$ and the affine Grassmannian.}
J. Inst. Math. Jussieu 14 (2015),  {493 ~-- ~575.}
\bibitem[GKM]{GKM} M. Goresky, R. Kottwitz, R. MacPherson, {\em Equivariant cohomology, Koszul duality, and the localization theorem.}
 Invent. Math. 131 (1998),  25 -- 83.
\bibitem[HTT]{HTT} R.
Hotta, K. Takeuchi, T. Tanisaki, {\em
$D$-modules, perverse sheaves, and representation theory.}
Progress in Mathematics, 236. Birkh\"auser Boston, Inc., Boston, MA, 2008. 
\bibitem[Ko1]{Ko1} B.  Kostant, {\em The principal three-dimensional subgroup and the Betti numbers of a complex simple Lie group.}
Amer. J. Math. 81 (1959) 973–1032. 
\bibitem[Ko2]{Ko2}\bysame, {\em Lie group representations on polynomial rings.} Amer. J. Math. 85 (1963), 327 -- 404. 
\bibitem[Ko3]{Ko3}\bysame, {\em On Whittaker vectors and representation theory.} Invent. Math. 48 (1978),  101 -- 184.
\bibitem[KK1]{KK1} \bysame, S. Kumar, {\em T-equivariant K-theory of generalized flag varieties.}
 J. Differential Geom. 32 (1990), 
549 -- 603. 
\bibitem[KK2]{KK2}\bysame,\bysame, {\em
 The nil Hecke ring and cohomology of G/P for a Kac-Moody group G.} Adv. in Math. 62 (1986), 187 ~-- ~237. 
\bibitem[Ku]{Ku} S. Kumar, {\em  Kac-Moody groups, their flag varieties and representation theory.}
 Progress in Mathematics, 204. Birkh\"auser Boston, Inc., Boston, MA, 2002.
\bibitem[LLMSSZ]{LLMSSZ} T. Lam, L. Lapointe, J. Morse, A. Schilling, M. Shimozono, M. Zabrocki, {\em
 k-Schur functions and affine Schubert calculus. Fields Institute Monographs}, 33. Springer, New York; 
Fields Institute for Research in Mathematical Sciences, Toronto, ON, 2014.
\bibitem[Lo1]{Lo1} G. Lonergan, {\em The Fourier transform for  the quantum Toda lattice.} arXiv:1706.05344.
\bibitem[Lo2]{Lo2} \bysame, {\em  A remark on descent for Coxeter groups.}
 arXiv:1707.01156.
\bibitem[MS]{MS}  D. Mili{\v c}i\'c, W. Soergel,
{\em Twisted Harish-Chandra sheaves and Whittaker modules: 
the nondegenerate case.} Developments and retrospectives in Lie theory, 183 -- 196,
 Dev. Math., 37, Springer, Cham, 2014.

\bibitem[MV]{MV} I. 
Mirkovi\'c, K. Vilonen,  {\em Characteristic varieties of character sheaves.} Invent. Math. 93 (1988), 405 -- 418.
\end{thebibliography}

\end{document}